\pdfoutput=1
\documentclass[a4paper,10pt,reqno]{amsart} 

\usepackage{graphicx}
\usepackage{amssymb}
\usepackage{epstopdf}
\usepackage{mathrsfs}
\usepackage{amsmath}
\usepackage{amsthm}
\usepackage{enumerate}
\usepackage[all,cmtip]{xy}
\usepackage{mathtools}
\usepackage{stmaryrd}
\usepackage{wasysym}
\usepackage{youngtab}
\usepackage{ytableau}
\usepackage{enumitem}
\usepackage{multicol}
 \usepackage{hyperref}
\usepackage{scalerel}[2016/12/29]
\usepackage[makeroom]{cancel}

\usepackage{tikz}
\usepackage{tikz-cd}

\usetikzlibrary{matrix, arrows}
\makeatletter
\tikzset{
  edge node/.code={%
      \expandafter\def\expandafter\tikz@tonodes\expandafter{\tikz@tonodes #1}}}
\makeatother
\tikzset{
  succ/.style={
    draw=none,
    edge node={node [sloped, allow upside down, auto=false]{$\succ$}}},
  Succ/.style={
    draw=none,
    every to/.append style={
      edge node={node [sloped, allow upside down, auto=false]{$\succ$}}}
  }
}

\makeatletter
\tikzset{
  edge node/.code={%
      \expandafter\def\expandafter\tikz@tonodes\expandafter{\tikz@tonodes #1}}}
\makeatother
\tikzset{
  Cong/.style={
    draw=none,
    edge node={node [sloped, allow upside down, auto=false]{$\cong$}}},
  Cong/.style={
    draw=none,
    every to/.append style={
      edge node={node [sloped, allow upside down, auto=false]{$\cong$}}}
  }
}

\newtheoremstyle{plainx}
  {3pt} 
  {3pt} 
  {\itshape} 
  {} 
  {\bfseries} 
  {.} 
  {.5em} 
  {} 

\newtheoremstyle{definitionx}
  {3pt} 
  {3pt} 
  {} 
  {} 
  {\bfseries} 
  {.} 
  {.5em} 
  {} 

\theoremstyle{plainx}
\newtheorem{thm}{Theorem}[section]
\newtheorem{lem}[thm]{Lemma}
\newtheorem{cor}[thm]{Corollary}
\newtheorem{prop}[thm]{Proposition}

\newtheorem{thmm}{Theorem}[section]

\theoremstyle{definitionx}
\newtheorem{defi}[thm]{Definition}
\newtheorem{exa}[thm]{Example}

\newtheorem{rem}[thm]{Remark}

\makeatletter
\DeclareRobustCommand\widecheck[1]{{\mathpalette\@widecheck{#1}}}
\def\@widecheck#1#2{%
    \setbox\z@\hbox{\m@th$#1#2$}%
    \setbox\tw@\hbox{\m@th$#1%
       \widehat{%
          \vrule\@width\z@\@height\ht\z@
          \vrule\@height\z@\@width\wd\z@}$}%
    \dp\tw@-\ht\z@
    \@tempdima\ht\z@ \advance\@tempdima2\ht\tw@ \divide\@tempdima\thr@@
    \setbox\tw@\hbox{%
       \raise\@tempdima\hbox{\scalebox{1}[-1]{\lower\@tempdima\box
\tw@}}}%
    {\ooalign{\box\tw@ \cr \box\z@}}}
\makeatother

\DeclareSymbolFont{TXlettersA}{U}{txmia}{m}{it}
\SetSymbolFont{TXlettersA}{bold}{U}{txmia}{bx}{it}
\DeclareFontSubstitution{U}{txmia}{m}{it}
\DeclareMathSymbol{\varpitx}{\mathord}{TXlettersA}{36}

\def \C {\mathbb{C}}

\def \Z {\mathbb{Z}}

\def \g {\mathfrak{g}}

\def \O {\mathcal{O}}

\DeclareMathOperator{\Ima}{Im}

\DeclareMathOperator{\End}{End}

\DeclareMathOperator{\gr}{gr}

\DeclareMathOperator{\Ind}{Ind}
\DeclareMathOperator{\Res}{Res}

\DeclareMathOperator{\mmod}{mod}

\DeclareMathOperator{\Hom}{Hom}

\DeclareMathOperator{\id}{id}

\DeclareMathOperator{\Stab}{Stab}
\DeclareMathOperator{\Ext}{Ext}

\DeclareMathOperator{\Com}{Com}
\DeclareMathOperator{\Bil}{Bil}
\DeclareMathOperator{\Fun}{Fun}
\DeclareMathOperator{\Morph}{Mor}
\DeclareMathOperator{\Inv}{Inv}
\DeclareMathOperator{\pmmod}{pmod}
\DeclareMathOperator{\Sym}{Sym}

\def \eu {\mathtt{eu}}

\makeatletter
\DeclareRobustCommand{\cev}[1]{%
  \mathpalette\do@cev{#1}%
}
\newcommand{\do@cev}[2]{%
  \fix@cev{#1}{+}%
  \reflectbox{$\m@th#1\vec{\reflectbox{$\fix@cev{#1}{-}\m@th#1#2\fix@cev{#1}{+}$}}$}%
  \fix@cev{#1}{-}%
}
\newcommand{\fix@cev}[2]{%
  \ifx#1\displaystyle
    \mkern#23mu
  \else
    \ifx#1\textstyle
      \mkern#23mu
    \else
      \ifx#1\scriptstyle
        \mkern#22mu
      \else
        \mkern#22mu
      \fi
    \fi
  \fi
}

\makeatletter
\newcommand{\leqnomode}{\tagsleft@true\let\veqno\@@leqno}
\newcommand{\reqnomode}{\tagsleft@false\let\veqno\@@eqno}
\makeatother

\newcommand{\ux}[1]{\underline{\mathbf{#1}}}
\newcommand{\bx}[1]{\mathbf{#1}}
\newcommand{\Sx}[2]{\mathfrak{#1}_{\mathbf{#2}}}
\newcommand{\Sxx}[2]{\mathfrak{#1}_{\underline{\mathbf{#2}}}}
\newcommand{\Sxxy}[3]{\mathfrak{#1}_{\underline{\mathbf{#2}},\underline{\mathbf{#3}}}}

\newcommand{\Gx}[2]{\mathsf{#1}_{\mathbf{#2}}}
\newcommand{\Gxx}[2]{\mathsf{#1}_{\underline{\mathbf{#2}}}}
\newcommand{\Ax}[2]{\mathcal{#1}_{\mathbf{#2}}}
\newcommand{\Axx}[2]{\mathcal{#1}_{\underline{\mathbf{#2}}}}
\newcommand{\Axxy}[3]{\mathcal{#1}_{\underline{\mathbf{#2}},\underline{\mathbf{#3}}}}

\def \ld {\ell_{\ux{d}}}
\newcommand{\shf}{\boldsymbol{\pitchfork}}
\newcommand{\dcc}[3]{{}_{\ux{#1}}\overset{\ux{#3}}{\mathsf{D}}_{\ux{#2}}}
\newcommand{\dc}[3]{{}_{\ux{#1}}\overset{\mathbf{#3}}{\mathsf{D}}_{\ux{#2}}}

\newcommand{\dcb}[2]{\mathsf{D}_{\ux{#1}}^{\mathbf{#2}}}

\newcommand{\dccb}[2]{\mathsf{D}^{\ux{#2}}_{\ux{#1}}}
\newcommand{\mer}[2]{\bigcurlywedge_{\ux{#1}}^{\ux{#2}}}
\newcommand{\spl}[2]{\bigcurlyvee_{\ux{#1}}^{\ux{#2}}}
\newcommand{\lamxx}[1]{\Lambda_{\ux{#1}}}
\newcommand{\lamx}[1]{\Lambda_{\mathbf{#1}}}
\newcommand{\mul}[2]{\mathsf{m}_{\ux{#1}}^{\ux{#2}}}
\newcommand{\com}[2]{\mathsf{com}_{\ux{#1}}^{\ux{#2}}}

\makeatletter
\DeclareRobustCommand{\cross}{\mathbin{\mathpalette\cross@@\relax}}
\newcommand{\cross@@}[2]{%
  \vbox{\offinterlineskip
    \sbox\z@{$\m@th#1\curlywedge$}%
    \ialign{%
      \hfil##\hfil\cr
      $\m@th#1\curlyvee\kern+.0\ht\z@$\cr
      \noalign{\kern-.3\ht\z@}
      \box\z@\cr
    }%
  }%
}
\makeatother

\newcommand{\arxiv}[1]{\href{http://arxiv.org/abs/#1}{\tt arXiv:\nolinkurl{#1}}}

\def \Rsl {\mathrlap{\rightslice}-} 

\newcommand{\tGx}[2]{{}^\theta\mathsf{#1}_{\mathbf{#2}}}
\newcommand{\tGxx}[2]{{}^\theta\mathsf{#1}_{\underline{\mathbf{#2}}}}
\newcommand{\tSx}[2]{{}^\theta\mathfrak{#1}_{\mathbf{#2}}}
\newcommand{\tSxx}[2]{{}^\theta\mathfrak{#1}_{\underline{\mathbf{#2}}}}
\newcommand{\tSxxy}[3]{{}^\theta\mathfrak{#1}_{\underline{\mathbf{#2}},\underline{\mathbf{#3}}}}

\newcommand{\tAx}[2]{{}^\theta\mathcal{#1}_{\mathbf{#2}}}
\newcommand{\tAxx}[2]{{}^\theta\mathcal{#1}_{\underline{\mathbf{#2}}}}
\newcommand{\tAxxy}[3]{{}^\theta\mathcal{#1}_{\underline{\mathbf{#2}},\underline{\mathbf{#3}}}}

\begin{document}

\title{Quiver Schur algebras and cohomological Hall algebras}
\author{Tomasz Przezdziecki}
\date{} 

\begin{abstract} 
We establish a connection between a generalization of KLR algebras, called quiver Schur algebras, and the cohomological Hall algebras of  Kontsevich and Soibelman. More specifically, we realize quiver Schur algebras as algebras of multiplication and comultiplication operators on the CoHA, and reinterpret the shuffle description of the CoHA in terms of Demazure operators. 
We introduce ``mixed quiver Schur algebras" associated to quivers with a contravariant involution, and show that they are related, in an analogous way, to the cohomological Hall modules defined by Young. 
We also obtain a geometric realization of the modified quiver Schur algebra, which appeared in a version of the Brundan-Kleshchev-Rouquier isomorphism for the affine $q$-Schur algebra due to Miemietz and Stroppel. 

\end{abstract}

\maketitle
\tableofcontents

\section{Introduction}

The main goal of this paper is to establish a connection between two algebras, which, historically, appeared in very different mathematical  contexts and were introduced with rather different motivations in mind, namely: quiver Schur algebras and cohomological Hall algebras. 

Quiver Schur algebras are a generalization of Khovanov and Lauda's \cite{KL1} and Rouquier's \cite{Rou} quiver Hecke algebras, 
nowadays also known as KLR algebras. 
The latter can be described algebraically by generators and relations, or in terms of a certain diagrammatic calculus. However, the passage from KLR algebras to quiver Schur algebras is easiest to understand from a geometric point of view. Varagnolo and Vasserot \cite{VV} (and later Kang, Kashiwara and Park \cite{KKP}, in a somewhat more general setting) constructed KLR algebras as extension algebras of a certain semisimple complex of constructible sheaves on the moduli stack of representations of a quiver. These extension algebras can also be described as convolution algebras in the equivariant Borel-Moore homology of a certain variety of triples, reminiscent of the classical  Steinberg variety. The triples consist of a pair of full  flags together with a compatible quiver representation. By incorporating partial flags into this construction, Stroppel and Webster \cite{SW} arrived at the definition of a quiver Schur algebra. Later, these algebras were studied from a more algebraic point of view in  \cite{MS}. 

One of the main motivations for introducing KLR algebras was to construct a categorification of quantum groups and their canonical bases. For results in this direction, we refer the reader to, e.g., \cite{KK, KL1, Rou2, VV}. Quiver Schur algebras also play an important role in this context. For example, quiver Schur algebras associated to the cyclic quiver provide a categorification of the generic nilpotent Hall algebra \cite[Proposition 2.12]{SW}, and their higher level versions categorify a higher level $q$-Fock space \cite[Theorem C]{SW}. 

The second protagonist of our story, the cohomological Hall algebra (CoHA), was introduced by Kontsevich and Soibelman \cite{KS} as a 
categorification of Donaldson-Thomas invariants of three dimensional Calabi-Yau categories. One of the primary original motivations for studying the CoHA was to provide a rigorous mathematical definition of the algebra of BPS states from string theory. 
CoHAs and their generalizations have found numerous applications in representation theory, including a new proof of the Kac positivity conjecture \cite{Dav}, as well as new realizations of the elliptic Hall algebra \cite{SV} and Yangians \cite{DM, SV2, YZ2}. 

In this paper we restrict ourselves to the relatively simple case of CoHAs associated to quivers with the trivial potential. For more information about this special case, including explicit examples, we refer the reader to, e.g., \cite{Efi, Fra, Rim}. 
One of our main results, described in more detail below, says that the relations between algebra and coalgebra structures on the CoHA  can be understood in terms of actions of quiver Schur algebras. 
It would be interesting to know whether one can associate KLR-type algebras to more general categories than those of quiver representations, and whether the connection between quiver Schur algebras and the CoHA described in this paper could be extended to such categories. 

We also remark that similar connections arise in other settings. For example, Nakajima's original proposal \cite[\S 7]{Nak} for the mathematical definition of Coulomb branches in terms of the vanishing cycle associated to the Chern-Simons functional was inspired by Donaldson-Thomas theory. On the other hand, the ultimate definition of Coulomb branches from \cite{BFN} involves a convolution algebra which can be viewed as an infinite dimensional example of Sauter's generalized quiver Hecke algebras from \cite{Sau1} (see \cite[Remark 3.9.4]{BFN}). 

\subsection{Main results.} 

We will now describe our results in more detail. Given a quiver $Q$ and a dimension vector $\bx{c}$, we consider the space $\Sx{Q}{c}$ of flagged representations of $Q$ with dimension vector~$\bx{c}$, together with the forgetful map onto the space $\Sx{R}{c}$ of  unflagged representations. In contrast to KLR algebras, we allow arbitrary partial flags instead of full flags only. 
The quiver Schur algebra $\Ax{Z}{c}$ is the equivariant Borel-Moore homology of the corresponding Steinberg-type variety 
\[ \Ax{Z}{c} = H_*^{\Gx{G}{c}}(\Sx{Q}{c} \times_{\Sx{R}{c}} \Sx{Q}{c}), \]
equipped with the convolution product as in \cite{CG}. We remark that our construction differs slightly from the construction of Stroppel and Webster \cite{SW} - they impose the additional condition on the space $\Sx{Q}{c}$ that each flagged quiver representation is nilpotent and its associated graded must be semisimple. To distinguish the two constructions, we refer to their convolution algebra $\Ax{Z}{c}^{SW}$ as the Stroppel-Webster quiver Schur algebra, and reserve the simpler name ``quiver Schur algebra" for $\Ax{Z}{c}$.

Our first result deals with the basic structural properties of quiver Schur algebras. It is well known that KLR algebras are generated by certain distinguished elements, called idempotents, polynomials and crossings, and that they admit a PBW-type basis. We prove an analogous result for quiver Schur algebras, with crossings replaced by fundamental classes called (elementary) merges and splits (see Definition \ref{defi:mergesandsplits}).

\newcounter{tmp}
\begingroup
\setcounter{tmp}{\value{thmm}}
\setcounter{thmm}{0} 
\renewcommand\thethmm{\Alph{thmm}}

\begin{thmm}[Theorem \ref{thm:basis}, Corollary \ref{cor:el spl mer}] 
The following hold: 
\begin{enumerate}[label=\alph*), font=\textnormal,noitemsep,topsep=3pt,leftmargin=1cm]
\item The quiver Schur algebra $\Ax{Z}{c}$ has a ``Bott-Samelson" basis consisting of pushforwards of fundamental classes of certain vector bundles on diagonal Bott-Samelson varieties. 
\item Elementary merges, elementary splits and polynomials generate $\Ax{Z}{c}$ as an algebra. 
\end{enumerate}
\end{thmm} 

The quiver Schur algebra $\Ax{Z}{c}$ has a natural faithful representation $\Ax{Q}{c}$, called the ``polynomial representation", on the direct sum of rings of partial invariants. 
We give an explicit description of this representation (Theorem~\ref{thm:polrep}) and interpret it in terms of Demazure operators (Proposition~\ref{lem: merge = Demazure}). In the special cases of the $A_1$ quiver (i.e., the quiver with one vertex and no arrows) and the Jordan quiver, we give a complete list of defining relations for the associated reduced quiver Schur algebra (see Theorems \ref{thm: rel A1} and \ref{thm: Jordan q rels}, as well as \cite{Sei}), which is defined as the subalgebra of $\Ax{Z}{c}$ generated by merges and splits, without the polynomials. The reduced quiver Schur algebra of the $A_1$ quiver turns out to be related to the green web category from \cite{CKM,TVW} (see Corollary \ref{cor: green web A1}), which arises naturally in the context of skew Howe duality. 

Our next result establishes a connection between quiver Schur algebras and the CoHA associated to the same quiver $Q$. 
We first need to introduce some notation. 
Let $\mathcal{Z} = \bigoplus_{\bx{c}} \Ax{Z}{c}$ be the direct sum of all the quiver Schur algebras associated to $Q$ (summing over all dimension vectors $\bx{c}$) and let $\mathcal{Q} = \bigoplus_{\bx{c}} \Ax{Q}{c}$ be the direct sum of their polynomial representations. We call $\mathcal{Z}$ the \emph{total} quiver Schur algebra. 
Let us now briefly recall a few facts about the CoHA. 
It is defined as the direct sum of equivariant cohomology groups 
\[ \mathcal{H} = \bigoplus_{\bx{c}}  H_{\mathsf{G}_{\mathbf{c}}}^\bullet (\mathfrak{R}_{\mathbf{c}}), \] 
equipped with multiplication via a certain pullback-pushforward construction. The CoHA can also be endowed with a coalgebra structure. However, the natural coproduct on the CoHA (see \cite[\S 2.9]{KS}) is not compatible with the multiplication in the sense that $\mathcal{H}$ is not a bialgebra. This problem can be remedied at the cost of passing to a localization of $\mathcal{H}$ and working with a localized version of the natural coproduct (see \cite{Dav2}). We do not pursue this approach here. Instead, we are interested in gaining a better understanding of the relations between the natural coalgebra and algebra structures on $\mathcal{H}$. 
The following theorem shows that these relations are controlled by the total quiver Schur algebra~$\mathcal{Z}$. 

\begin{thmm}[Theorem \ref{thm: TH vs Pol}] 
The faithful polynomial representation $\mathcal{Q}$ of the total quiver Schur algebra $\mathcal{Z}$ can be naturally identified with the tensor algebra $T(\mathcal{H}_+)$ on the augmentation ideal $\mathcal{H}_+$ of the CoHA. This identification induces an injective algebra homomorphism \[ \mathcal{Z} \hookrightarrow \End(T(\mathcal{H}_+)) \]
sending elementary merges in $\mathcal{Z}$ to CoHA multiplication operators and elementary splits in $\mathcal{Z}$ to CoHA comultiplication operators. 
\end{thmm}

The CoHA admits a description as a shuffle algebra \cite[Theorem 2]{KS} in the sense of Feigin and Odesskii \cite{FO}. We interpret this description in terms of Demazure operators (Proposition \ref{pro: mult Dem op}), connecting it to our description of the polynomial representation $\mathcal{Q}$ of the quiver Schur algebra~$\mathcal{Z}$. We expect that the relationship between shuffle algebras and Demazure operators carries over to more general settings. For example, we expect that multiplication in the formal version of the CoHA, defined by Yang and Zhao \cite{YZ} for any equivariant oriented Borel-Moore homology theory, can be rephrased in terms of the formal Demazure operators from \cite{HMSZ}. 

\subsection{Geometric realization of the modified quiver Schur algebra.} 

One of the exciting features of KLR algebras (associated to finite and affine type A quivers) is that, 
after passing to suitable completions or cyclotomic quotients, they are isomorphic to affine Hecke algebras \cite{BK, Rou}, and endow the latter with interesting gradings. This isomorphism, known in the literature as the Brundan-Kleshchev-Rouquier isomorphism, was later generalized to Schur algebras in \cite{MakS,MS,SW} (see also \cite{Web}). 

The main result of \cite{MS} says that the convolution algebra $\Ax{Z}{c}^{SW}$ from \cite{SW} is, after completion, isomorphic to the affine $q$-Schur algebra appearing naturally in the representation theory of $p$-adic general linear groups. The proof of this result relies on the fact that both of these algebras are isomorphic to a certain intermediate algebra $\Ax{Z}{c}^{MS}$, called the modified quiver Schur algebra, which is defined in purely algebraic terms. We show that the modified quiver Schur algebra also admits a geometric realization as a convolution algebra. 

\begin{thmm}[Theorem \ref{cor:QSvMQS}] 
There is a natural algebra isomorphism $\Ax{Z}{c} \cong \Ax{Z}{c}^{MS}$ between our quiver Schur algebra $\Ax{Z}{c}$ and the modified quiver Schur algebra $\Ax{Z}{c}^{MS}$. 
\end{thmm}

As an application, we deduce that our quiver Schur algebra $\Ax{Z}{c}$ is also isomorphic to the Stroppel-Webster quiver Schur algebra $\Ax{Z}{c}^{SW}$ (Theorem \ref{thm: our vs SW}). 

\subsection{Mixed quiver Schur algebras.} 

As we have already mentioned, KLR and quiver Schur algebras can be realized as convolution algebras, or, equivalently, extension algebras of a certain semisimple complex of sheaves on the moduli stack of representations of a quiver. If the quiver admits a contravariant involution $\theta$, this construction can be generalized by replacing the stack of representations of the quiver with the stack of its self-dual representations. 

This idea was pursued by Varagnolo and Vasserot in \cite{VV2}. They obtained generalized KLR algebras which are Morita equivalent to affine Hecke algebras of type B, and provide a categorification of highest weight modules over $B_\theta(\g_Q)$, the algebra introduced by Enomoto and Kashiwara \cite{EK1, EK2} in the context of symmetric crystals. 
The type D case is considered in  \cite{KM,SVV}. 

Sauter \cite{Sau1, Sau2, Sau3}  
took the idea of generalizing KLR algebras further, and replaced the stack of self-dual representations of a quiver with the stack of generalized quiver representations in the sense of Derksen and Weyman \cite{DW}. In this generalization, the gauge group acting on the space of quiver representations is no longer a classical group, but an arbitrary reductive group. 

We define a generalization of quiver Schur algebras which is close in spirit to the above-mentioned generalizations of KLR algebras. Given a quiver $Q$ with a contravariant involution $\theta$ and an extra datum, called a duality structure (see Definition \ref{defi: duality str}), we consider the stack of a certain type of self-dual representations of $Q$, introduced by Zubkov \cite{Zub} under the name of supermixed quiver representations. 
We refer to the resulting Ext-algebra as the \emph{mixed quiver Schur algebra} and denote it by $\tAx{Z}{c}$. 
The mixed quiver Schur algebra has similar structural properties to the ordinary quiver Schur algebra: it has a Bott-Samelson basis (Theorem \ref{thm:basis theta}) and is generated by elementary merges, elementary splits and polynomials (Corollary \ref{cor:el spl mer theta}). 

The idea of replacing ordinary quiver representations by self-dual representations has also been exploited  in the representation theory of Hall algebras (in the finite field setting) \cite{You1} and cohomological Hall algebras  \cite{You} by Young. 
In the finite field case, Young defined a ``Hall module" over the Hall algebra of $Q$, and showed that it carries a natural action of the aforementioned Enomoto-Kashiwara algebra $B_\theta(\g_Q)$. In the cohomological case, he introduced a ``cohomological Hall module"  ${}^\theta\mathcal{M}$ over the cohomological Hall algebra $\mathcal{H}$ associated to the same quiver $Q$ without the involution $\theta$. The module ${}^\theta\mathcal{M}$ is defined as the direct sum of equivariant cohomology groups 
\[ {}^\theta\mathcal{M} = \bigoplus_{\bx{c}} H_{{}^\theta\mathsf{G}_{\mathbf{c}}}^\bullet ({}^\theta\mathfrak{R}_{\mathbf{c}}) \]
of the spaces $\tSx{R}{c}$ of self-dual quiver representations, equipped with an $\mathcal{H}$-module structure via certain geometric correspondences. The module ${}^\theta\mathcal{M}$ also carries a natural $\mathcal{H}$-comodule structure, but it fails to be a Hopf module. 
Our next result shows that the relations between multiplication and comultiplication in the CoHA, as well as its action and coaction on the cohomological Hall module, are controlled by the total mixed quiver Schur algebra ${}^\theta\mathcal{Z} = \bigoplus_{\bx{c}} \tAx{Z}{c}$.

\begin{thmm}[Theorem \ref{thm: TH vs Pol theta}] \label{thm D intro}
There is an injective algebra homomorphism \[{}^\theta\mathcal{Z} \hookrightarrow \End(T(\mathcal{H}_+) \otimes {}^\theta\mathcal{M})\]
sending elementary merges in ${}^\theta\mathcal{Z}$ to CoHA multiplication and action operators and elementary splits in ${}^\theta\mathcal{Z}$ to CoHA comultiplication and coaction operators. 
\end{thmm}

As an application of Theorem \ref{thm D intro}, we obtain an explicit description of the faithful polynomial representation of a mixed quiver Schur algebra  (Theorem \ref{thm: pol rep theta}). Moreover, we reinterpret the description of the CoHM as a shuffle module \cite[Theorem 3.3]{You} in terms of Demazure operators of types A-D (Corollary \ref{cor: action Demazure theta}). 

Mixed quiver Schur algebras are also related to the Hall modules defined in the finite field setting. 
The direct sum ${}^\theta\mathcal{Z}\mbox{-}\pmmod$ of the categories of finitely generated graded projective modules over all the mixed quiver Schur algebras carries a natural action of the monoidal category $\mathcal{Z}\mbox{-}\pmmod$, 
and its Grothendieck group $K_0({}^\theta\mathcal{Z})$ is a module as well as a comodule over $K_0(\mathcal{Z})$ (Proposition~\ref{pro: monoidal action}). 
We expect that, via the standard technique of sending the class of a semisimple perverse sheaf to the function given by the super-trace of the Frobenius on its stalks (see, e.g., \cite{Schiff2}), $K_0(\mathcal{Z})$ can be identified with a subalgebra of the Hall algebra of $Q$. 
For example, in the special case of a Dynkin or cyclic quiver, $K_0(\mathcal{Z})^{op}$ is naturally isomorphic to the generic nilpotent Hall algebra (Proposition~\ref{pro: K0 gen nilp Hall}). 
We also expect that $K_0({}^\theta\mathcal{Z})$ can be identified with a subspace of the Hall module associated to the category of self-dual representations of $Q$, and that $K_0({}^\theta\mathcal{Z})$ is a semisimple module over the Enomoto-Kashiwara algebra $B_\theta(\g_Q)$.

\subsection*{Acknowledgements.} 
This work was supported by the College of Science \& Engineering at the University of Glasgow and the Max Planck Institute for Mathematics in Bonn. 
I would like to thank C. Stroppel for drawing my attention to the similarity between the shuffle formula from \cite{KS} and formulas appearing in the description of the quiver Schur algebra, which was the starting point for this project, as well as M.B. Young for discussions about (cohomological) Hall modules. I would also like to thank G. Bellamy, C. Stroppel and M.B. Young for numerous comments and suggestions on draft versions of this paper. 

\section{Preliminaries}

In this section we introduce notation and basic definitions which will be used throughout the paper. We begin by setting up the notation for quivers and associated combinatorial objects such as dimension vectors and their compositions. We then recall the definitions of some geometric objects associated to quivers, such as quiver flag varieties and the corresponding Steinberg-type varieties. We finish by recalling a few facts about equivariant cohomology and convolution algebras.

\subsection{Quivers and associated combinatorics.} 
Let us fix for the rest of this section a quiver $Q$ with a finite set of vertices $Q_0$ and a finite set of arrows $Q_1$. In particular, we allow multiple edges and edge loops. 

If $a \in Q_1$ is an arrow, let $s(a)$ be its source and $t(a)$ its target. Let $a_{ij}$ denote the number of arrows from vertex $i$ to $j$.  Let $\Gamma := \Z_{\geq 0}Q_0$ denote the free commutative monoid of dimension vectors for $Q$ and let $\Gamma_+ := \Gamma \backslash \{0\}$. 
If $\mathbf{c} = \sum_{i \in Q_0} \mathbf{c}(i)\cdot i \in \Gamma$, write $|\mathbf{c}| = \sum_{i \in Q_0} \mathbf{c}(i) \in \Z$. Given a $Q_0$-graded vector space $V$, let $\dim_{Q_0} V \in \Gamma$ denote its $Q_0$-graded dimension. 

Let $n$ be a positive integer. We say that $\beta = (\beta_1, \hdots, \beta_{\ell_\beta}) \in (\Z_{\geq 1})^{\ell_\beta}$ is a \emph{composition} of~$n$ if $\sum_j \beta_j = n$. Let $\Com(n)$ denote the set of compositions of $n$. 
Given $\beta \in \Com(n)$, let $\mathring{\beta}_j = \beta_1 + \hdots + \beta_j$ for $1 \leq j \leq \ell_{\beta}$, with $\mathring{\beta}_0 = 0$. 

\begin{defi}
Let $\mathbf{c} \in \Gamma_+$. 
We say that $\underline{\mathbf{d}} = (\mathbf{d}_1, \hdots, \mathbf{d}_{\ell_{\underline{\mathbf{d}}}}) \in \Gamma^{\ell_{\underline{\mathbf{d}}}}_+$ 
is a \emph{vector composition} of~$\mathbf{c}$, denoted $\underline{\mathbf{d}} \rightslice \mathbf{c}$, if $\langle \underline{\mathbf{d}}\rangle := \sum_{j=1}^{\ld} \mathbf{d}_j = \mathbf{c}$. 
We call $\ld$ the length of $\ux{d}$. 
Let $\mathbf{Com}_{\mathbf{c}}$ denote the set of vector compositions of $\mathbf{c}$ and let $\mathbf{Com}_{\mathbf{c}}^n$ denote the subset of vector compositions of length $n$. 
The symmetric group $\mathsf{Sym}_n$ acts naturally on $\mathbf{Com}_{\mathbf{c}}^n$ from the right by permutations. 
For each $i \in Q_0$, we have a map 
\[ \mathbf{Com}_{\mathbf{c}} \to \Com(\mathbf{c}(i)), \quad \ux{d} \mapsto \underline{\mathbf{d}}(i) := (\mathbf{d}_1(i), \hdots, \mathbf{d}_{\ell_{\underline{\mathbf{d}}}}(i)).\]
Given two vector compositions $\ux{d} \rightslice \mathbf{a}$ and $\ux{e} \rightslice \mathbf{b}$, let $\ux{d} \cup \ux{e} = (\mathbf{d}_1, \hdots, \mathbf{d}_{\ell_{\underline{\mathbf{d}}}}, \mathbf{e}_1, \hdots, \mathbf{e}_{\ell_{\ux{e}}}) \rightslice \mathbf{a} + \mathbf{b}$ be their concatenation.  
\end{defi}

\begin{defi} \label{defi: wedge comp}
Suppose that $\beta \in \Com(\ell_{\underline{\mathbf{d}}})$. 
Define 
\[ \vee_{\beta}^j( \underline{\mathbf{d}}) := (\mathbf{d}_{\mathring{\beta}_{j-1}+1}, \hdots, \mathbf{d}_{\mathring{\beta}_j}), \quad \wedge_\beta(\underline{\mathbf{d}}) := ( \langle \vee_{\beta}^{1}(\underline{\mathbf{d}})\rangle, \hdots, \langle \vee_{\beta}^{\ell_{\beta}}(\underline{\mathbf{d}}) \rangle).\] 
In particular, if $\beta = (1^{k-1},2,1^{\ell_{\underline{\mathbf{d}}}-k-1})$ for some $1 \leq k \leq \ell_{\underline{\mathbf{d}}}-1$, then we abbreviate $\wedge_k(\underline{\mathbf{d}}) := \wedge_\beta(\underline{\mathbf{d}})$. We define a partial order on $\mathbf{Com}_{\mathbf{c}}$ by setting
\[ \underline{\mathbf{d}} \succcurlyeq \underline{\mathbf{e}} \iff \underline{\mathbf{e}} = \wedge_\beta(\underline{\mathbf{d}}) \]
for some $\beta \in \Com(\ell_{\underline{\mathbf{d}}})$. 
If $ \underline{\mathbf{d}} \succcurlyeq \underline{\mathbf{e}}$, we call $\ux{d}$ a \emph{refinement} of $\ux{e}$, and $\ux{e}$ a \emph{coarsening} of $\ux{d}$. 
\end{defi} 

\begin{exa} \label{first example}
Consider the $A_3$ quiver
\[
\begin{tikzcd}
\underset{i_1}{\bullet} \arrow[r] &\underset{i_2}{\bullet} \arrow[r] &\underset{i_3}{\bullet}
\end{tikzcd}
\]
Let $\bx{c} = 4 i_1 + 3  i_2 + 3  i_3$ and $\ux{d} = (i_1 + i_3, 2 i_1+i_2, 2  i_3, i_1 + i_2, i_2) \rightslice \bx{c}$ so that $\ld = 5$. 
We have $\wedge_{(5)}(\ux{d}) = \bx{c}$ and $\wedge_{(1,1,1,1,1)}(\ux{d}) = \ux{d}$. Moreover, 
\begin{alignat*}{8}
\wedge_{(2,3)}(\ux{d}) &=  (3i_1+i_2+i_3, i_1+2i_2+2i_3)  \\
\wedge_{(2,1,2)}(\ux{d}) &=  (3i_1+i_2+i_3, 2i_3, i_1 + 2i_2)  \\
\wedge_{1}(\ux{d}) = \wedge_{(2,1,1,1)}(\ux{d}) &= (3i_1 + i_2 +i_3, 2 i_1+i_2, 2  i_3, i_1 + i_2, i_2),  \\ 
\wedge_{3}(\ux{d}) = \wedge_{(1,1,2,1)}(\ux{d}) &= (i_1 + i_3, 2 i_1+i_2, i_1 + i_2 + 2i_3, i_2).
\end{alignat*}
\end{exa}

Next we assign some products of symmetric groups to the combinatorics developed above. 
Given a positive integer $n$ and $\alpha \in \Com(n)$, let $\mathsf{Sym}_{\alpha} = \prod_{j=1}^{\ell_\alpha} \mathsf{Sym}_{\alpha_j}$. 
Furthermore, set 
\[ \Gx{W}{c} := \prod_{i \in Q_0} \mathsf{Sym}_{\mathbf{c}(i)}, \quad \Gxx{W}{d} := \prod_{i \in Q_0} \mathsf{Sym}_{\ux{d}(i)} \subseteq \mathsf{Sym}_{\mathbf{c}}.\]  
We consider the groups $\Gx{W}{c}$ and $\Gxx{W}{d}$ as Coxeter groups in the usual way. In particular, they are endowed with a length function $\ell$ and a Bruhat order. Let $s_j(i)$ $(i \in Q_0, 1 \leq j \leq \mathbf{c}(i)-1)$ be the standard generators of $\Gx{W}{c}$. 
Given $\ux{e}, \ux{f} \succ \ux{d} \rightslice \mathbf{c}$, let 
\[ \dcc{e}{f}{d} := [\Gxx{W}{e} \backslash \Gxx{W}{d} / \Gxx{W}{f}]^{\mathsf{min}} \]
denote the set of minimal length double coset representatives. If $\ux{d} = (\mathbf{c})$, we write $\dc{e}{f}{c} := \dcc{e}{f}{d}$. When $\Gxx{W}{e}$ is trivial,  
we abbreviate $ \mathsf{D}^{\ux{d}}_{\ux{f}} := \dcc{e}{f}{d}$.

\subsection{Quiver representations and flag varieties.} 

Let $\mathbf{c} \in \Gamma_+$. 
Fix a $Q_0$-graded $\C$-vector space $\mathbf{V}_{\mathbf{c}} = \bigoplus_{i \in Q_0} \mathbf{V}_{\mathbf{c}}(i)$ with $\dim \mathbf{V}_{\mathbf{c}}(i) = \mathbf{c}(i)$. 
Let us fix a basis $\{v_k(i) \mid 1\leq k \leq \mathbf{c}(i)\}$ of $\mathbf{V}_{\mathbf{c}}(i)$ for each $i \in Q_0$. 
Let 
\[ \Sx{R}{c} := \bigoplus_{a \in Q_1} \Hom_{\C}(\mathbf{V}_{\mathbf{c}}(s(a)), \mathbf{V}_{\mathbf{c}}(t(a))), \quad \Gx{G}{c} := \prod_{i \in Q_0} \mathsf{GL}(\mathbf{V}_{\mathbf{c}}(i)).  \]
The group $\Gx{G}{c}$ acts naturally on $\Sx{R}{c}$ by conjugation. 
Let $\Gx{T}{c}$ be the standard maximal torus in $\Gx{G}{c}$, with fundamental weights $\omega_j(i)$ (for $i \in Q_0$, $1 \leq j \leq \bx{c}(i)$), and let $\Gx{B}{c}$ be the standard Borel subgroup. Let $R_{\bx{c}}^+ \subset R_{\bx{c}}$ be the associated (positive) root system. We identify the associated Weyl group with $\Gx{W}{c}$. Given $w \in \Gx{W}{c}$, let $R_{\bx{c}}(w) = \{ \alpha \in R_{\bx{c}}^+ \mid w(\alpha) \in -R_{\bx{c}}^+\}$. If $\alpha \in R_{\bx{c}}$, let $\mathsf{U}_\alpha$ be the corresponding unipotent subgroup of $\Gx{G}{c}$, and set $\mathsf{U}_w =  \prod_{\alpha \in R_{\bx{c}}(w)}\mathsf{U}_\alpha$.

We call a sequence $V_\bullet$ of $Q_0$-graded subspaces 
\[ \{0\} = V_0 \subset V_1 \subset V_2 \subset \hdots \subset V_{\ld} = \mathbf{V}_{\mathbf{c}} \] 
a \emph{flag of type} $\underline{\mathbf{d}} \in \mathbf{Com}_{\mathbf{c}}$ if $\dim_{Q_0} V_j/V_{j-1} = \mathbf{d}_j$. 
We refer to the flag $\mathbf{V}_{\underline{\mathbf{d}}} = (\mathbf{V}_{\underline{\mathbf{d}}}^j)_{j=0}^{\ld}$, where  $\mathbf{V}_{\underline{\mathbf{d}}}^j :=\langle v_k(i) \mid 1 \leq k \leq \mathbf{d}_1(i) + \hdots + \mathbf{d}_j(i), i \in Q_0 \rangle$, as the \emph{standard flag} of type $\underline{\mathbf{d}}$. 
Let 
\begin{equation*}  \textstyle\Gxx{P}{d} := \Stab_{\mathsf{G}_{\mathbf{c}}}(\mathbf{V}_{\underline{\mathbf{d}}} ), \quad \Gxx{L}{d} := \prod_{j=1}^{\ell_{\underline{\mathbf{d}}}} \mathsf{G}_{\mathbf{d}_j}. \end{equation*}
be the parabolic and Levi subgroups, respectively, associated to $\ux{d}$, and let $R_{\ux{d}}^+ \subseteq R^+_{\bx{c}}$ be the corresponding subset of positive roots. Let $\Sxx{F}{d} \cong \mathsf{G}_{\mathbf{c}} / \mathsf{P}_{\underline{\mathbf{d}}}$ be the projective variety parametrizing flags of type $\underline{\mathbf{d}}$. 

\begin{defi} \label{def: stable vs str stable}
Let $\rho = (\rho_a) \in \Sx{R}{c}$. 
We say that a flag $V_\bullet$ is:
\begin{itemize}[itemsep=5pt]
\item $\rho$\emph{-stable} if $\rho_a(V_{j}(s(a))) \subseteq V_{j}(t(a))$, 
\item \emph{strictly} $\rho$\emph{-stable} if $\rho_a(V_{j}(s(a))) \subseteq V_{j-1}(t(a))$, 
\end{itemize}
for all $a \in Q_1$ and $1 \leq j \leq \ell_{\underline{\mathbf{d}}}$. 
\end{defi} 

Given $\underline{\mathbf{d}} \succcurlyeq \underline{\mathbf{e}}$ and $V_\bullet$, a flag of type $\underline{\mathbf{d}}$, let $V_\bullet|_{\underline{\mathbf{e}}}$ denote its coarsening to a flag of type $\underline{\mathbf{e}}$. Let
\begin{equation} \label{Qde} \Sxxy{Q}{d}{e} := \{ (V_\bullet, \rho) \in \Sxx{F}{d} \times \Sx{R}{c} \mid V_\bullet|_{\underline{\mathbf{e}}} \mbox{ is } \rho\mbox{-stable} \}, \quad \Sxx{Q}{d}:= \Sxxy{Q}{d}{d}. \end{equation}
be the space of flags of type $\ux{d}$ together with suitably compatible quiver representations. 
There is a canonical $\mathsf{G}_{\mathbf{c}}$-equivariant isomorphism 
\begin{equation} \label{Qd as quotient} \Gx{G}{c} \times^{\Gxx{P}{d}} \Sxx{R}{e} \cong \Sxxy{Q}{d}{e}, \quad (g,\rho) \mapsto (g \cdot \mathbf{V}_{\underline{\mathbf{d}}}, g \cdot \rho), \end{equation}  
where $ \Sxx{R}{d} := \{ \rho \in \Sx{R}{c} \mid \mathbf{V}_{\underline{\mathbf{d}}} \mbox{ is } \rho\mbox{-stable}\}$. 
Let 
\[ \Sxx{F}{d} \overset{\tau_{\ux{d}}}{\longleftarrow} \Sxx{Q}{d} \overset{\pi_{\ux{d}}}{\longrightarrow} \Sx{R}{c} \]
be the canonical projections. The first one, $\tau_{\ux{d}}$, is a vector bundle while the second one, $\pi_{\ux{d}}$, is a proper map.  We abbreviate 
\[ \Sx{F}{c} :=  \bigsqcup_{\underline{\mathbf{d}} \rightslice \mathbf{c}} \Sxx{F}{d}, \quad \Sx{Q}{c} :=  \bigsqcup_{\underline{\mathbf{d}} \rightslice \mathbf{c}} \Sxx{Q}{d}, \quad \pi_{\mathbf{c}} := \sqcup \pi_{\ux{d}} \colon  \Sx{Q}{c} \to \Sx{R}{c}.\] 

\begin{defi}
Let $\Sxx{Q}{d}^s$ and $\Sxx{R}{d}^s$ be the varieties obtained by replacing ``$\rho$-stable'' with ``strictly $\rho$-stable'' in the definitions of $\Sxx{Q}{d}$ and $\Sxx{R}{d}$, respectively. 
Set $\Sx{Q}{c}^s = \bigsqcup_{\underline{\mathbf{d}} \rightslice \mathbf{c}}\Sxx{Q}{d}^s$. 
\end{defi}

\begin{rem} \label{rem: str st vs st}
We make a few remarks about the relationship between $\Sxx{Q}{d}$ and $\Sxx{Q}{d}^s$. 
\begin{enumerate}[label=(\roman*),topsep=2pt,itemsep=1pt]
\item The variety $\Sxx{Q}{d}$ is isomorphic to the variety of representations $\phi$ of $Q$ with dimension vector~$\bx{c}$ together with a filtration by subrepresentations $\phi_1 \subset \hdots \subset \phi_{\ld}$ such that the dimension vectors of the subquotients form the vector composition $\ux{d}$. 
If we impose the additional condition that the subquotients in a filtration are nilpotent and semisimple, we obtain $\Sxx{Q}{d}^s$.
\item If the quiver $Q$ has no edge loops and each dimension vector $\mathbf{d}_1, \hdots, \mathbf{d}_{\ell_{\underline{\mathbf{d}}}}$ in the vector composition $\ux{d}$ is supported only at one vertex, then $\Sxx{Q}{d} = \Sxx{Q}{d}^s$. 
\item The variety $\Sxx{Q}{d}$ is a generalization of the universal quiver Grassmannian defined in \cite{Sch}. 
\end{enumerate}
\end{rem}

\subsection{The Steinberg variety.} 
Given $\ux{d}, \ux{e} \rightslice \mathbf{c}$, set
\[ \mathfrak{Z}_{\underline{\mathbf{e}},\underline{\mathbf{d}}} := \mathfrak{Q}_{\underline{\mathbf{e}}} \times_{\mathfrak{R}_{\mathbf{c}}} \mathfrak{Q}_{\underline{\mathbf{d}}}, \quad \mathfrak{Z}_{\mathbf{c}} := \mathfrak{Q}_{\mathbf{c}} \times_{\mathfrak{R}_{\mathbf{c}}} \mathfrak{Q}_{\mathbf{c}} = \bigsqcup_{\ux{d},\ux{e} \rightslice \mathbf{c}} \Sxxy{Z}{e}{d},\] 
where the fibred product is taken with respect to  
$\pi_{\mathbf{c}}.$ We call $\Sx{Z}{c}$ the \emph{quiver Steinberg variety}. 
Let $\mathfrak{Z}_{\underline{\mathbf{e}},\underline{\mathbf{d}}}^s := \mathfrak{Q}_{\underline{\mathbf{e}}}^s \times_{\mathfrak{R}_{\mathbf{c}}} \mathfrak{Q}_{\underline{\mathbf{d}}}^s$ and $\Sx{Z}{c}^s := \Sx{Q}{c}^s \times_{\Sx{R}{c}} \Sx{Q}{c}^s$ be the corresponding strictly stable versions. 

\begin{exa}
Let $Q$ be the Jordan quiver, $\bx{c}=n$ and $\ux{d} = (1^n)$. Then $\Sx{R}{c}^s = \mathcal{N}$ is the nilpotent cone, $\Sxx{Q}{d}^s = \widetilde{\mathcal{N}} = T^*(GL_n/B)$ is the cotangent bundle to the flag variety, $\pi_{\ux{d}} \colon  \widetilde{\mathcal{N}} \to \mathcal{N}$ is the Springer resolution and $\Sxxy{Z}{d}{d}^s$ is the usual Steinberg variety. On the other hand, $\Sx{R}{c} = \g = \mathfrak{gl}_n$, $\Sxx{Q}{d} = \widetilde{\g} = GL_n \times^B \mathfrak{b}$ and $\pi_{\ux{d}} \colon \widetilde{\g} \to \g$ is the Grothendieck-Springer resolution. 
\end{exa}

We next define a relative position stratification on $\Ax{Z}{c}$. 
Consider the projection
\[ \tau_{\ux{e},\ux{d}} \colon \Sxxy{Z}{e}{d} \to \Sxx{F}{e} \times \Sxx{F}{d} \]
remembering the flags and forgetting the quiver representation. Given
$w \in \dc{e}{d}{c}$, let $\O_w^{\Delta} := \Gx{G}{c}\cdot(e\Gxx{P}{e}, w\Gxx{P}{d}) \subset \Sxx{F}{e} \times \Sxx{F}{d}$ be the diagonal $\Gx{G}{c}$-orbit corresponding to $w$, and set 
\begin{equation} \label{Zw lesseq} \Sxxy{Z}{e}{d}^w := \tau_{\ux{e},\ux{d}}^{-1}(\O_w^{\Delta}), \quad \Sxxy{Z}{e}{d}^{\leqslant w} := \bigsqcup_{\dc{e}{d}{c} \ni u \leqslant w} \Sxxy{Z}{e}{d}^u, \quad \Sxxy{Z}{e}{d}^{< w} = \Sxxy{Z}{e}{d}^{\leqslant w} \backslash \Sxxy{Z}{e}{d}^w,\end{equation}
where $u \leqslant w$ stands for the Bruhat order. 

\begin{lem} \label{lem: Z closed}
For each $w \in \dc{e}{d}{c}$, the subvariety $\Sxxy{Z}{e}{d}^{\leqslant w}$ is closed in $\Sxxy{Z}{e}{d}$, and the inclusion $\Sxxy{Z}{e}{d}^w \hookrightarrow \Sxxy{Z}{e}{d}^{\leqslant w}$ is an open immersion. 
\end{lem}

\begin{proof}
By the usual Bruhat decomposition, we have $\overline{\O_w^{\Delta}} = \bigsqcup_{\dc{e}{d}{c} \ni u \leqslant w} \O_u^{\Delta}$. Hence $\Sxxy{Z}{e}{d}^{\leqslant w} = \tau_{\ux{e},\ux{d}}^{-1}(\overline{\O_w^{\Delta}})$ is closed in $\Sxxy{Z}{e}{d}$. Since $\O_w^\Delta$ is open in $\overline{\O_w^{\Delta}}$, the preimage $\Sxxy{Z}{e}{d}^w= \tau_{\ux{e},\ux{d}}^{-1}(\O_w^{\Delta})$ is also open in $\Sxxy{Z}{e}{d}^{\leqslant w}$. 
\end{proof}

\subsection{Cohomology.} \label{section: cohomology}

Below we will always deal with complex algebraic varieties which are also smooth manifolds or admit closed embeddings into smooth manifolds. 
Let $X$ be a complex algebraic variety with an action of a complex linear algebraic group $G$. We denote by $EG$ the universal bundle and by $BG$ the classifying space associated to $G$. The quotient $X_G := EG \times^G X = (EG \times X)/G$ by the diagonal $G$-action 
is called the homotopy quotient of $X$ by $G$. 
Let $H^\bullet_G(X) := H^\bullet(X_G)$ denote the $G$-equivariant cohomology ring and $H_\bullet^G(X) := H_\bullet(X_G)$ the $G$-equivariant Borel-Moore homology of $X$, with coefficients in $\C$. 
If $Y \subseteq X$ is a closed $G$-stable subvariety, let $[Y] \in H_\bullet^G(X)$ denote its $G$-equivariant fundamental class. 
Given a $G$-equivariant complex vector bundle $V$ on $X$, let $\mathtt{eu}_{G}(V) = \mathtt{eu}(EG \times^G V) \in H^\bullet_G(X)$ denote its top $G$-equivariant Chern class, i.e., the equivariant Euler class of the underlying real vector bundle. 
More information about equivariant homology and cohomology may be found in, e.g., \cite{Bri, Bri2}. 

We will now introduce notation for various equivariant cohomology groups. 
Define 
\[ 
\Ax{P}{c} := H^\bullet( B\Gx{T}{c}) = \bigotimes_{i \in Q_0} \C[x_1(i), \hdots, x_{\mathbf{c}(i)}(i)], 
\]
where 
$x_j(i) := \mathtt{eu}(\mathfrak{V}_j(i))$ is the first Chern class of the line bundle 
\begin{equation} \label{line bundle} \mathfrak{V}_j(i) := E\Gx{T}{c} \times^{\omega_j(i)} \C. \end{equation}
For each $\ux{d} \rightslice \mathbf{c}$, 
set 
\[ \lamxx{d} := \Ax{P}{c}^{\Gxx{W}{d}}, \quad \lamx{c} := \bigoplus_{\underline{\mathbf{d}} \rightslice \mathbf{c}} \lamxx{d}. \] 
The canonical map $B\Gx{T}{c} \twoheadrightarrow B \Gxx{P}{d}$ induces an injective algebra homomorphism
\[ H^\bullet(B\Gxx{P}{d}) \hookrightarrow  H^\bullet(B\Gx{T}{c}) 
\] 
whose image is $\lamxx{d}$. 
Given any $\underline{\mathbf{d}} \succcurlyeq \underline{\mathbf{e}}$, we use the homotopy equivalence $\Sxxy{Q}{d}{e} \simeq \Sxx{F}{d}$ and the fact that $(\Sxx{F}{d})_{\Gx{G}{c}} = B\Gxx{P}{d}$ to 
identify
\[ H^\bullet_{\Gx{G}{c}}(\Sxxy{Q}{d}{e}) \cong H^\bullet_{\Gx{G}{c}}(\Sxx{F}{d}) = \Lambda_{\ux{d}}. \] 

We next introduce notation for various equivariant Borel-Moore homology groups. Set
\begin{equation} \label{BM homology groups Q Z} \Axx{Q}{d} := H_\bullet^{\Gx{G}{c}}(\Sxx{Q}{d}), \quad \mathcal{Q}_{\mathbf{c}} := H^{\mathsf{G}_{\mathbf{c}}}_\bullet( \mathfrak{Q}_{\mathbf{c}}), \quad \quad \Axxy{Z}{e}{d} := H_\bullet^{\Gx{G}{c}}(\Sxxy{Z}{e}{d}), \quad \mathcal{Z}_{\mathbf{c}} := H^{\mathsf{G}_{\mathbf{c}}}_\bullet( \mathfrak{Z}_{\mathbf{c}}). \end{equation}
Since the varieties $\Sxx{Q}{d}$ and $\Sx{Q}{c}$ are smooth, Poincar\'{e} duality yields isomorphisms
\begin{equation} \label{Poincare} \mathcal{Q}_{\underline{\mathbf{d}}} \cong H_{\mathsf{G}_{\mathbf{c}}}^\bullet( \mathfrak{Q}_{\underline{\mathbf{d}}}) \cong \Lambda_{\ux{d}}, \quad \Ax{Q}{c} \cong H_{\mathsf{G}_{\mathbf{c}}}^\bullet( \Sx{Q}{c}) \cong \Lambda_{\mathbf{c}}. \end{equation} 
Moreover, set 
\begin{equation} \label{Zw Ze} \Axxy{Z}{e}{d}^w := H_\bullet^{\Gx{G}{c}}(\Sxxy{Z}{e}{d}^w), \quad \Ax{Z}{c}^e := \bigoplus_{\underline{\mathbf{d}} \rightslice \mathbf{c}} \Axxy{Z}{d}{d}^e. \end{equation} 

\subsection{Convolution.} \label{subsec:convolution}

We recall the definition of the convolution product from \cite{CG}. Let $G$ be a complex Lie group, $X_1, X_2, X_3$ be smooth complex $G$-manifolds, and let $Z_{12} \subset X_1 \times X_2$ and $Z_{23} \subset X_2 \times X_3$ be closed $G$-stable subsets. Let $p_{ij} \colon X_1 \times X_2 \times X_3 \to X_i \times X_j$ be the projection onto the $i$-th and $j$-th factors. Assume that the restriction of $p_{13}$ to $Z_{12} \times_{X_2} Z_{23}$ is proper. Set $Z_{13} = Z_{12} \circ Z_{23} = p_{13}(Z_{12} \times_{X_2} Z_{23})$. Given $c_{12} \in H_\bullet^{G}(Z_{12})$ and $c_{23} \in H_\bullet^{G}(Z_{23})$, their \emph{convolution} is defined as 
\[ c_{12} \star c_{23} := (p_{13})_*((p_{12}^*c_{12}) \cap (p_{23}^*c_{23})) \in H_\bullet^G(Z_{13}),\]
where $\cap$ denotes the intersection pairing. We will often need to compute the convolution of fundamental classes in the following special case. 

\begin{lem} \label{lem:conv of classes}
Assume that $Z_{12} \subset X_1 \times X_2$ and $Z_{23} \subset X_2 \times X_3$ are complex submanifolds. Further, suppose that either of the canonical projections $Z_{12} \to X_2 \leftarrow Z_{23}$ is a submersion, and that the map $p_{13} \colon Z_{12} \times_{X_2} Z_{23} \to Z_{13}$ is an isomorphism. 
Then $[Z_{12}] \star [Z_{23}] = [Z_{13}]$.  
\end{lem} 

\begin{proof}
The submersion assumption implies that the intersection of $p_{12}^{-1}(Z_{12})$ and $p_{23}^{-1}(Z_{23})$ is transverse (see \cite[Remark 2.7.27.(ii)]{CG}). Hence, by \cite[Proposition 2.6.47]{CG}, we have  $p_{12}^*[Z_{12}] \cap p_{23}^*[Z_{23}] = [Z_{12} \times_{X_2} Z_{23}]$. Since $p_{13}$, restricted to  $Z_{12} \times_{X_2} Z_{23}$, is an isomorphism onto $Z_{13}$, we get $(p_{13})_*[Z_{12} \times_{X_2} Z_{23}] =[Z_{13}]$. 
\end{proof}

Let $X$ be a smooth complex $G$-manifold, let $Y$ be a possibly singular complex $G$-variety and let $\pi \colon X \to Y$ be a $G$-equivariant proper map. Set $X_1 = X_2 = X_3 = X$, $Z = Z_{12} = Z_{23} = X \times_Y X$. 
Convolution yields a product $H_\bullet^G(Z) \times H_\bullet^G(Z) \to H_\bullet^G(Z)$, which, by \cite[Corollary 2.7.41]{CG}, makes $H_\bullet^G(Z)$ into a unital associative $H^\bullet(BG)$-algebra. The unit is given by $[X_{\Delta}]$, the $G$-equivariant fundamental class of $X$ diagonally embedded into $Z$.  
Next, let $X_1 = X_2 = X$ and $X_3 = \{ pt \}$. Then convolution yields an action $H_\bullet^G(Z) \times H_\bullet^G(X) \to H_\bullet^G(X)$, which makes $H_\bullet^G(X)$ into a left $H_\bullet^G(Z)$-module.

\section{Quiver Schur algebras}

In this section we define the quiver Schur algebra $\Ax{Z}{c}$ and construct a ``Bott-Samelson basis'' for  $\Ax{Z}{c}$. We deduce that  $\Ax{Z}{c}$ is generated by certain special elements called merges, splits and polynomials. 

\subsection{The quiver Schur algebra} 
Fix $\bx{c} \in \Gamma$. 
We apply the framework of \S \ref{subsec:convolution} to the vector bundle $X = \Sx{Q}{c}$ on the quiver flag variety $\Sx{F}{c}$, the space of quiver representations $Y = \Sx{R}{c}$ and the projection $\pi = \pi_{\mathbf{c}}$. Then $Z = \Sx{Z}{c}$ is the quiver Steinberg variety, and we obtain a convolution algebra structure on its Borel-Moore homology $\Ax{Z}{c} = H^{\mathsf{G}_{\mathbf{c}}}_\bullet( \mathfrak{Z}_{\mathbf{c}})$ and a $\Ax{Z}{c}$-module structure on $\Ax{Q}{c} = H_\bullet^{\Gx{G}{c}}(\Sx{Q}{c})$. By~\eqref{Poincare}, $\Ax{Q}{c}$ can be identified with the direct sum $\lamx{c}$ of rings of invariant polynomials.

\begin{defi}
We call $\Ax{Z}{c}$ the \emph{quiver Schur algebra} associated to $(Q,\mathbf{c})$, and $\Ax{Q}{c}$ its \emph{polynomial representation}. 
\end{defi}

\begin{rem} \label{rem: SW vs our QS}
Our quiver Schur algebra can be seen as a modification of the quiver Schur algebra introduced by Stroppel and Webster in \cite[\S 2.2]{SW}. There are two differences between our construction and theirs. Firstly, Stroppel and Webster only consider cyclic quivers with at least two vertices, while we work with arbitrary finite quivers. Secondly, we use the quiver Steinberg variety $\Sx{Z}{c}$ while they use its strictly stable version $\Sx{Z}{c}^s$. We will refer to the algebra from \cite{SW} as the ``\emph{Stroppel-Webster quiver Schur algebra}'' and denote it by $\Ax{Z}{c}^{SW}$. 
\end{rem}

The following standard result follows from the general theory of convolution algebras (see, e.g., \cite[Proposition 8.6.35]{CG}). 

\begin{prop}
There are canonical isomorphisms  
\begin{equation} \label{Ext alg iso} \mathcal{Z}_{\mathbf{c}} \cong \Ext^\bullet_{\mathsf{G}_{\mathbf{c}}}((\pi_{\mathbf{c}})_*\C_{\mathfrak{Q}_{\mathbf{c}}}, (\pi_{\mathbf{c}})_*\C_{\mathfrak{Q}_{\mathbf{c}}}), \quad \mathcal{Q}_{\mathbf{c}} \cong \Ext^\bullet_{\mathsf{G}_{\mathbf{c}}}(\C_{\mathfrak{R}_{\mathbf{c}}}, (\pi_{\mathbf{c}})_*\C_{\mathfrak{Q}_{\mathbf{c}}}) \end{equation} 
intertwining the convolution product with the Yoneda product, and the convolution action with the Yoneda action, respectively. 
\end{prop}

\subsection{Merges, splits and polynomials} 

We will now introduce notation and a diagrammatic calculus for certain special fundamental classes in $\Ax{Z}{c}$. 
We begin by observing that $\Ax{Z}{c}^e \subset \Ax{Z}{c}$ is a subalgebra and that there is an algebra isomorphism
\begin{equation} \label{Zepols}
\Ax{Z}{c}^e \cong H_{\mathsf{G}_{\mathbf{c}}}^\bullet(\mathfrak{Q}_{\mathbf{c}}) \cong \Lambda_{\mathbf{c}}. 
\end{equation} 
For this reason, we refer to $\Ax{Z}{c}^e$ as ``the polynomials'' in $\Ax{Z}{c}$. Next, observe that the  fundamental classes $\mathsf{e}_{\ux{d}} := [\Sxxy{Z}{d}{d}^e]$ form a complete set of mutually orthogonal idempotents in $\Ax{Z}{c}$. The definition below introduces two other kinds of fundamental classes, which, following \cite{SW}, we call ``merges'' and ``splits''.

\begin{defi} \label{defi:mergesandsplits} 
Given $\underline{\mathbf{d}} \succ \underline{\mathbf{e}} \rightslice \mathbf{c}$, we call 
\begin{itemize}[itemsep=5pt]
\item $\bigcurlywedge_{\ux{d}}^{\ux{e}}:= [\Sxxy{Z}{e}{d}^e] \in \Axxy{Z}{e}{d}$ a \emph{merge},  
\item $\bigcurlyvee_{\ux{e}}^{\ux{d}}:= [\Sxxy{Z}{d}{e}^e] \in \Axxy{Z}{d}{e}$ a \emph{split}. 
\end{itemize}

We say that a merge or split is \emph{elementary} if $\ux{e} = \wedge_k(\ux{d})$ (see Definition \ref{defi: wedge comp}) for some $1 \leq k \leq \ell_{\underline{\mathbf{d}}}-1$. 
We will depict elementary merges and splits diagrammatically in the following way. To the elementary merge  $\bigcurlywedge_{\ux{d}}^{\wedge^k(\ux{d})}$ we associate the diagram 
\[
\tikz[thick,xscale=2,yscale=1.2]{ \small 
\draw (0,0) node[below] {$\bx{d}_k$} to [out=90,in=-90](.3,.5)
(.6,0) node[below] {$\bx{d}_{k+1}$} to [out=90,in=-90] (.3,.5)
(.3,.5) -- (.3,.8) node[above] {$\bx{d}_k+\bx{d}_{k+1}$};
\node at (1.2,0.4) {$:=$};
\begin{scope}[xshift=3cm]
\draw (-1.2,0) -- (-1.2,.8) node[below,at start]{$\bx{d}_1$} node[above]{$\bx{d}_1$};
\node at (-0.9,.4) {$\cdots$}; 
\draw (-0.6,0) -- (-0.6,.8) node[below,at start]{$\bx{d}_{k-1}$} node[above]{$\bx{d}_{k-1}$};
\draw (0,0) node[below] {$\bx{d}_k$} to [out=90,in=-90](.3,.5)
(.6,0) node[below] {$\bx{d}_{k+1}$} to [out=90,in=-90] (.3,.5)
(.3,.5) -- (.3,.8) node[above] {$\bx{d}_k+\bx{d}_{k+1}$};
\draw (1.2,0) -- (1.2,.8) node[below,at start]{$\bx{d}_{k+2}$} node[above]{$\bx{d}_{k+2}$};
\node at (1.5,.4) {$\cdots$}; 
\draw (1.8,0) -- (1.8,.8) node[below,at start]{$\bx{d}_{\ld}$} node[above]{$\bx{d}_{\ld}$};
\end{scope} 
}\] 
and to the elementary split $\bigcurlyvee^{\ux{d}}_{\wedge^k(\ux{d})}$ the diagram 
\[
\tikz[thick,xscale=2,yscale=-1.2]{ \small
\draw (0,0) node[above] {$\bx{d}_k$} to [out=90,in=-90](.3,.5)
(.6,0) node[above] {$\bx{d}_{k+1}$} to [out=90,in=-90] (.3,.5)
(.3,.5) -- (.3,.8) node[below] {$\bx{d}_k+\bx{d}_{k+1}$};
\node at (1.2,0.4) {$:=$};
\begin{scope}[xshift=3cm]
\draw (-1.2,0) -- (-1.2,.8) node[above,at start]{$\bx{d}_1$} node[below]{$\bx{d}_1$};
\node at (-0.9,.4) {$\cdots$}; 
\draw (-0.6,0) -- (-0.6,.8) node[above,at start]{$\bx{d}_{k-1}$} node[below]{$\bx{d}_{k-1}$};
\draw (0,0) node[above] {$\bx{d}_k$} to [out=90,in=-90](.3,.5)
(.6,0) node[above] {$\bx{d}_{k+1}$} to [out=90,in=-90] (.3,.5)
(.3,.5) -- (.3,.8) node[below] {$\bx{d}_k+\bx{d}_{k+1}$};
\draw (1.2,0) -- (1.2,.8) node[above,at start]{$\bx{d}_{k+2}$} node[below]{$\bx{d}_{k+2}$};
\node at (1.5,.4) {$\cdots$}; 
\draw (1.8,0) -- (1.8,.8) node[above,at start]{$\bx{d}_{\ld}$} node[below]{$\bx{d}_{\ld}$};
\end{scope} 
}\]
The diagram on the LHS should be understood as shorthand notation for the full diagram on the RHS. 
Multiplication of elementary merges and splits is depicted through a vertical composition of diagrams. We always read diagrams from the bottom to the top. 

We call 
\[
\tikz[thick,xscale=2,yscale=1.2]{ 
\node at (-1,0.6) {$\cross_{\ux{d}}^k$}; 
\node at (-0.5,0.6) {$:=$}; 
\small 
\draw (0,0) node[below] {$\bx{d}_k$} to [out=90,in=-90](.3,.5)
(.6,0) node[below] {$\bx{d}_{k+1}$} to [out=90,in=-90] (.3,.5)
(.3,.5) -- (.3,.8);
\begin{scope}[yscale=-1, yshift=-1.25cm]
\draw (0,0) node[above] {$\bx{d}_{k+1}$} to [out=90,in=-90](.3,.5)
(.6,0) node[above] {$\bx{d}_{k}$} to [out=90,in=-90] (.3,.5)
(.3,.5) -- (.3,.8);
\end{scope}
}\] 
a \emph{crossing}. 
\end{defi}

\begin{prop} \label{pro:transitivity} 
We list several basic relations which hold in $\Ax{Z}{c}$. 
\begin{enumerate}[label=\alph*), font=\textnormal,noitemsep,topsep=3pt,leftmargin=1cm]
\item Let $\ux{d} \succ \ux{e} \succ \ux{f} \rightslice \bx{c}$. Merges and splits satisfy the following \emph{transitivity} relations: 
\[  \textstyle \mer{e}{f} \star \mer{d}{e} = \mer{d}{f}, \quad \quad \spl{e}{d} \star \spl{f}{e} = \spl{f}{d}. \] 
\item Let $\ux{d} \rightslice \bx{c}$ and $1 \leq k \leq \ell_{\underline{\mathbf{d}}}-2$. Elementary merges satisfy the following \emph{associativity} relation: 
\begin{equation} \tag{R1} \label{R1 relation}
\tikz[thick,xscale=2,yscale=1, baseline=0.8cm]{ \small
\draw (0,0) node[below] {$\bx{d}_k$} to [out=90,in=-90](.3,.5)
(.6,0) node[below] {$\bx{d}_{k+1}$} to [out=90,in=-90] (.3,.5)
(.3,.5) -- (.3,.8) 
(1.2,0) -- (1.2,.8) node[below,at start]{$\bx{d}_{k+2}$} 
(.3,.8) to [out=90,in=-90](.75,1.3)
(1.2,.8) to [out=90,in=-90](.75,1.3) 
(.75,1.3) -- (.75,1.6) node[above] {$\bx{d}_k+\bx{d}_{k+1}+\bx{d}_{k+2}$}; 
\node at (2,0.8) {$=$};
\begin{scope}[xshift=2.8cm]
\draw 
(0,0) -- (0,.8) node[below,at start]{$\bx{d}_{k}$} 
(0.6,0) node[below] {$\bx{d}_{k+1}$} to [out=90,in=-90](.9,.5)
(1.2,0) node[below] {$\bx{d}_{k+2}$} to [out=90,in=-90] (.9,.5)
(.9,.5) -- (.9,.8) 
(0,.8) to [out=90,in=-90](.45,1.3)
(0.9,.8) to [out=90,in=-90](.45,1.3) 
(.45,1.3) -- (.45,1.6) node[above] {$\bx{d}_k+\bx{d}_{k+1}+\bx{d}_{k+2}$}; 
\end{scope}
}\end{equation}
Elementary splits satisfy the following \emph{coassociativity} relation: 
\begin{equation} \tag{R2} \label{R2 relation}
\tikz[thick,xscale=2,yscale=-1, baseline=-0.8cm]{ \small
\draw (0,0) node[above] {$\bx{d}_k$} to [out=90,in=-90](.3,.5)
(.6,0) node[above] {$\bx{d}_{k+1}$} to [out=90,in=-90] (.3,.5)
(.3,.5) -- (.3,.8) 
(1.2,0) -- (1.2,.8) node[above,at start]{$\bx{d}_{k+2}$} 
(.3,.8) to [out=90,in=-90](.75,1.3)
(1.2,.8) to [out=90,in=-90](.75,1.3) 
(.75,1.3) -- (.75,1.6) node[below] {$\bx{d}_k+\bx{d}_{k+1}+\bx{d}_{k+2}$}; 
\node at (2,0.8) {$=$};
\begin{scope}[xshift=2.8cm]
\draw 
(0,0) -- (0,.8) node[above,at start]{$\bx{d}_{k}$} 
(0.6,0) node[above] {$\bx{d}_{k+1}$} to [out=90,in=-90](.9,.5)
(1.2,0) node[above] {$\bx{d}_{k+2}$} to [out=90,in=-90] (.9,.5)
(.9,.5) -- (.9,.8) 
(0,.8) to [out=90,in=-90](.45,1.3)
(0.9,.8) to [out=90,in=-90](.45,1.3) 
(.45,1.3) -- (.45,1.6) node[below] {$\bx{d}_k+\bx{d}_{k+1}+\bx{d}_{k+2}$}; 
\end{scope}
}\end{equation}
\end{enumerate} 
\end{prop}

\begin{proof}
Part a) follows via an easy calculation from Lemma \ref{lem:conv of classes}. Part b) follows immediately from part a). 
\end{proof} 

\subsection{Relation to KLR algebras} 
We would like to connect the quiver Schur algebra to the well known quiver Hecke (or KLR) algebra (associated to the same quiver $Q$), defined diagrammatically by Khovanov and Lauda \cite{KL1}, and algebraically by Rouquier \cite{Rou}. Let us recall its geometric construction and some generalizations. 
Define  
\begin{equation} \label{KLR} \mathfrak{Q}_{\mathbf{c}}^{\mathsf{KLR}} := \bigsqcup_{\underline{\mathbf{d}} \in \Com_{\mathbf{c}}^{|\mathbf{c}|}} \mathfrak{Q}_{\underline{\mathbf{d}}}, \quad \mathfrak{Z}_{\mathbf{c}}^{\mathsf{KLR}} := \mathfrak{Q}_{\mathbf{c}}^{\mathsf{KLR}} \times_{\mathfrak{R}_{\mathbf{c}}} \mathfrak{Q}_{\mathbf{c}}^{\mathsf{KLR}}, \quad \mathcal{Z}_{\mathbf{c}}^{\mathsf{KLR}} := H^{\mathsf{G}_{\mathbf{c}}}_\bullet( \mathfrak{Z}_{\mathbf{c}}^{\mathsf{KLR}}).  \end{equation}
Note that the set $\Com_{\mathbf{c}}^{|\mathbf{c}|}$ contains precisely the vector compositions of $\bx{c}$ of maximal length, i.e., those parametrizing the types of complete quiver flags. 
Let $\mathcal{Z}_{\mathbf{c}}^{\mathsf{KLR},s}$ be the strictly stable version of $\mathcal{Z}_{\mathbf{c}}^{\mathsf{KLR}}$,  obtained by replacing each $\mathfrak{Q}_{\underline{\mathbf{d}}}$ with $\mathfrak{Q}_{\underline{\mathbf{d}}}^s$ in \eqref{KLR}. 

In the case when the quiver $Q$ has no edge loops, it was proven by Varagnolo and Vasserot \cite{VV} that 
the convolution algebra $\mathcal{Z}_{\mathbf{c}}^{\mathsf{KLR}}$ gives a geometric realization of the KLR algebra associated to $(Q, \mathbf{c})$. 
Note that in this case Remark \ref{rem: str st vs st}(ii) implies that $\mathcal{Z}_{\mathbf{c}}^{\mathsf{KLR},s}$ coincides with $\mathcal{Z}_{\mathbf{c}}^{\mathsf{KLR}}$. 

The case when $Q$ may contain edge loops was studied by Kang, Kashiwara and Park \cite{KKP}. They showed that the convolution algebra $\mathcal{Z}_{\mathbf{c}}^{\mathsf{KLR},s}$ gives a geometric realization of the generalized KLR algebra from \cite{KOP}  associated to a symmetrizable Borcherds-Cartan datum. 

Let $Q$ be an arbitrary quiver. 
The following proposition describes $\mathcal{Z}_{\mathbf{c}}^{\mathsf{KLR}}$ as a subalgebra of $\Ax{Z}{c}$. 

\begin{prop}
The convolution algebra $\mathcal{Z}_{\mathbf{c}}^{\mathsf{KLR}}$ is a subalgebra of $\mathcal{Z}_{\mathbf{c}}$. It is generated by the polynomials $\Axxy{Z}{d}{d}^e$ and the crossings $\cross_{\ux{d}}^k$ (for $\underline{\mathbf{d}} \in \Com_{\mathbf{c}}^{|\mathbf{c}|}$ and $1 \leq k \leq |\mathbf{c}|-1$). 
\end{prop}

\begin{proof}
The natural inclusion $\mathfrak{Z}_{\mathbf{c}}^{\mathsf{KLR}} \hookrightarrow \Sx{Z}{c}$ is an inclusion of connected components, and hence induces an inclusion of the corresponding convolution algebras. This proves the first claim. 
Given $\underline{\mathbf{d}} \in \Com_{\mathbf{c}}^{|\mathbf{c}|}$, 
let us identify $\Gx{W}{c}$ with the subgroup of $\mathsf{Sym}_{|\mathbf{c}|}$ preserving $\ux{d}$.
It follows easily from Lemma \ref{lem:conv of classes} that $\cross_{\ux{d}}^k = [\mathfrak{Z}_{s_k(\ux{d}),\ux{d}}^e]$ if $s_k(\ux{d}) \neq \ux{d}$ and $\cross_{\ux{d}}^k = [\overline{\Sxxy{Z}{d}{d}^{s_k}}]$ if $s_k(\ux{d}) = \ux{d}$. But these fundamental classes, together with the polynomials $\Axxy{Z}{d}{d}^e \cong \lamxx{d}$, generate $\mathcal{Z}_{\mathbf{c}}^{\mathsf{KLR}}$ by (a straightforward generalization of)  \cite[Theorem 3.6]{VV}. 
\end{proof}

\subsection{The combinatorics of refinements} \label{subsec: refinements}

Our next goal is to construct a basis for the quiver Schur algebra which is natural both from algebraic and geometric points of view. Algebraically, the basis elements are certain products of merges, splits and polynomials. Geometrically, they will be realized as pushforwards of vector bundles on diagonal Bott-Samelson varieties. In \S \ref{subsec: refinements} we develop the combinatorial tools  needed to define the basis. We state the basis theorem in \S \ref{sec: basis and gens}, and prove it in \S \ref{subsec: proof of basis}. 

\begin{defi}
Given $i \in Q_0$, let $\mathbf{N}_{\bx{c}}(i):=\{ (1,i), \hdots, (\bx{c}(i),i)\}$, and set $\mathbf{N}_{\bx{c}} := \bigsqcup_{i \in Q_0}\mathbf{N}_{\bx{c}}(i)$. 
By a \emph{partitioning} of $\bx{c}$ of length $n = \ell_\lambda$ we mean a function $\lambda \colon \mathbf{N}_{\bx{c}} \to \Z_{\geq 1}$ such that $\Ima \lambda = [1,n] := \{1, \hdots, n\}$. 
Let $\lambda_i$ be the restriction of $\lambda$ to $\mathbf{N}_{\bx{c}}(i)$. 
Let $\mathbf{Par}_{\bx{c}}^n$ denote the set of all partitionings of $\bx{c}$ of length $n$ and let $\mathbf{Par}_{\bx{c}} := \bigsqcup_{1 \leq n \leq |\mathbf{c}|}\mathbf{Par}_{\bx{c}}^n$. The following lemma follows directly from the definitions. 
\end{defi}

\begin{lem} \label{lem: partitionings and flags}
There is a bijection between $\mathbf{Par}_{\bx{c}}$ and the set of coordinate flags in $\mathbf{V}_{\bx{c}}$, sending $\lambda$ to the flag $V_\bullet$ with $V_r = \langle v_k(i) \mid (k,i) \in \lambda^{-1}([1,r])\rangle$. 
\end{lem}

Partitionings are related to vector compositions of $\bx{c}$ through functions
\[ \begin{tikzcd}
\mathbf{Par}_{\bx{c}} \arrow[twoheadrightarrow, bend left=15,r, "C"] & \mathbf{Com}_{\bx{c}} \arrow[bend left=15, l, "P"]
\end{tikzcd} \] 
defined in the following way. If $\lambda \in \mathbf{Par}_{\bx{c}}$, then $C(\lambda) = (\bx{d}_1, \hdots, \bx{d}_{|\lambda|})$ is given by $\bx{d}_k(i) = |\lambda_i^{-1}(k)|$. If $\ux{e} \in \mathbf{Com}_{\bx{c}}$, then $P(\ux{e}) = \lambda$ is given by $\lambda_i^{-1}(k) = \{ (\mathring{\bx{e}}_{k-1}(i)+1,i), \hdots, (\mathring{\bx{e}}_{k}(i),i)\}$. 

\begin{exa}
Consider $Q, \bx{c}$ and $\ux{d} \in \mathbf{Com}_{\bx{c}}^5$ from Example \ref{first example}. Let $\lambda = P(\ux{d})$. Then 
\[ \lambda^{-1}(1) = \{ (1,i_1), (1,i_3) \}, \quad \lambda^{-1}(2) = \{ (2,i_1), (3,i_1), (1,i_2) \}, \quad \lambda^{-1}(3) = \{ (2,i_3), (3,i_3) \},\]
\[  \lambda^{-1}(4) = \{ (4,i_1), (2,i_2) \}, \quad \lambda^{-1}(5) = \{ (3,i_2) \}.\]
Next, let $\mu \in \mathbf{Par}_{\bx{c}}^4$ be given by 
\begin{alignat*}{4}
\mu^{-1}(1) \ &&=&\ \{ (2,i_2), (3,i_2), (2, i_3) \}, \quad && \mu^{-1}(2) \ &=& \ \{ (4, i_1) \},\\
\mu^{-1}(3) \ &&=& \ \{ (2, i_1), (1,i_2), (1,i_3), (3,i_3) \}, \quad && \mu^{-1}(4) \ &=& \ \{ (1, i_1), (3,i_1) \}. 
\end{alignat*}
Then $C(\mu) = (2i_2 + i_3, i_1, i_1+i_2+2i_3, 2i_1)$. 
\end{exa}

We let $\Gx{W}{c}$ act on $\mathbf{N}_{\bx{c}}$ from the left and $\mathsf{Sym}_{n}$ act on $[1,n]$ from the right. This means that $\mathsf{Sym}_{n}$ acts by permuting places rather than numbers. 
We get induced actions on $\Fun(\mathbf{N}_{\bx{c}}, [1,n])$, which preserve $\mathbf{Par}_{\bx{c}}^n$, viewed as the subset of $\Fun(\mathbf{N}_{\bx{c}}, [1,n])$ consisting of surjective functions.  
Note that the resulting $\mathsf{Sym}_{n}$-action on $\mathbf{Par}_{\bx{c}}^n$ is free. 
The following lemma follows directly from the definitions. 

\begin{lem} \label{lem: C P functions}
The functions $C$ and $P$ have the following properties:
\begin{enumerate}[label=\alph*), font=\textnormal,noitemsep,topsep=3pt,leftmargin=1cm]
\item The function $C$ is surjective and $C \circ P = \id$. 
\item We have $\ell_{C(\lambda)} = \ell_{\lambda}$ and $\ell_{P(\ux{d})} = \ell_{\ux{d}}$, for all $\lambda \in \mathbf{Par}_{\bx{c}}$ and $\ux{d} \in \mathbf{Com}_{\bx{c}}$. 
\item The fibres of $C$ are precisely the $\Gx{W}{c}$-orbits in $\mathbf{Par}_{\bx{c}}$. Hence $C$ induces a set isomorphism $$\mathbf{Par}_{\bx{c}}/\Gx{W}{c} \cong \mathbf{Com}_{\bx{c}}.$$
\item We have $\Stab_{\Gx{W}{c}}(P(\ux{d})) = \Gxx{W}{d}$ for any $\ux{d} \in \mathbf{Com}_{\bx{c}}$. 
\item The map $C|_{\mathbf{Par}_{\bx{c}}^{n}}$ is $\mathsf{Sym}_n$-equivariant,  for each $1 \leq n \leq |\bx{c}|$. 
\item We have $s_j \cdot P(\ux{d}) = \widetilde{s_j} \cdot P(s_j \cdot \ux{d})$ with $\widetilde{s_j}$ being the longest element in $\mathsf{D}_{s_j\cdot\ux{d}}^{\wedge_j(\ux{d})}$, for any $\ux{d} \in \mathbf{Com}_{\bx{c}}$ and $1 \leq j \leq \ell_{\ux{d}} -1$ $($note that $s_j \in \mathsf{Sym}_{\ld}$ while $\widetilde{s_j} \in \Gx{W}{c})$. 
\end{enumerate}
\end{lem} 

Next, we define a binary operation~$\leo$ on $\mathbf{Par}_{\bx{c}}$. 

\begin{defi}
Given $\lambda, \mu \in \mathbf{Par}_{\bx{c}}$, 
let $R_{k,l} = \lambda^{-1}(k) \cap \mu^{-1}(l)$ and ${S}_{\lambda,\mu} = \{ R_{k,l} \neq \varnothing \mid 1 \leq k \leq |\lambda|, \ 1 \leq l \leq |\mu| \}$. 
We define a total order on the set ${S}_{\lambda,\mu}$ by declaring that $R_{k,l} < R_{r,s}$ if and only if $r>k$ or $r=k$ and $l < s$. 
We then define the \emph{ordered intersection} $\lambda \mathrel{\leo} \mu = \nu$ of partitionings $\lambda$ and $\mu$ by setting 
\[
\nu^{-1}(m) = \left\{ \begin{array}{ll} 
\mbox{$m$-th element of $S_{\lambda,\mu}$ in the total order defined above}, & \mbox{if $1 \leq m \leq |S_{\lambda,\mu}|$}, \\
\varnothing, & \mbox{if $m > |S_{\lambda,\mu}|$.}
\end{array} \right. 
\]
\end{defi} 

One can immediately see that $\nu$ is in fact a partitioning of $\bx{c}$. The operation $\leo$ is not symmetric. However, the following holds. 

\begin{lem} \label{lem: comb cap}
Let $\lambda, \mu \in \mathbf{Par}_{\bx{c}}$. Then: 
\begin{enumerate}[label=\alph*), font=\textnormal,noitemsep,topsep=3pt,leftmargin=1cm]
\item $\lambda \mathrel{\textnormal{\leo}} \mu$ and $\mu \mathrel{\textnormal{\leo}} \lambda$ are of the same length and lie in the same orbit of $\mathsf{Sym}_{\ell_{\lambda \mathrel{\textnormal{\leo}} \mu}}$. 
\item $C(\lambda \mathrel{\textnormal{\leo}} \mu) \succcurlyeq C(\lambda)$ $($but in general $C(\lambda \mathrel{\textnormal{\leo}} \mu) \not\succcurlyeq C(\mu))$. 
\item $\Stab_{\Gx{W}{c}}(\lambda \mathrel{\textnormal{\leo}} \mu) = \Stab_{\Gx{W}{c}}(\lambda) \cap \Stab_{\Gx{W}{c}}(\mu)$. 
\item $w \cdot (\lambda \mathrel{\textnormal{\leo}} \mu) = (w \cdot \lambda) \mathrel{\textnormal{\leo}} (w\cdot \mu)$ for all $w \in \Gx{W}{c}$. 
\end{enumerate} 
\end{lem}

\begin{proof}
Part a) follows immediately from the fact that the sets $S_{\lambda,\mu}$ and $S_{\mu,\lambda}$ are the same if we forget their orderings. Part b) is obvious. For part c), observe that if $\nu \in \mathbf{Par}_{\bx{c}}$ then $\Stab_{\Gx{W}{c}}(\nu) = \bigcap_{1 \leq r \leq \ell_\nu} \Stab_{\Gx{W}{c}}(\nu^{-1}(r))$, where $\Stab_{\Gx{W}{c}}(\nu^{-1}(r))$ is the subgroup of $\Gx{W}{c}$ fixing $\nu^{-1}(r)$ setwise. 
If $x \in \Gx{W}{c}$ stabilizes all the subsets $\lambda^{-1}(k)$ and $\mu^{-1}(l)$ then $x$ also stabilizes all their intersections $R_{k,l}$. Hence it stabilizes $\lambda \mathrel{\textnormal{\leo}} \mu$. Conversely, note that $S_{\lambda,\mu}$ is a partition of the set $\mathbf{N}_{\bx{c}}$, which refines the partitions $S_{\lambda,\lambda}$ and $S_{\mu,\mu}$. Hence, if $x$ stabilizes all the sets $R_{k,l}$ in $S_{\lambda,\mu}$, then it must also stabilize all the preimages $\lambda^{-1}(k)$ and $\mu^{-1}(l)$. 
Part d) is clear. 
\end{proof}

\begin{exa} \label{exa: int}
Suppose that $Q_0$ is a singleton. 
We can then identify $\Gamma$ with $\Z_{\geq 0}$. Let $\bx{c} = 8$ and take $\lambda = [1,2,3][4,5][6,7,8]$ and $\mu = [1,2,4,5][6,7][3,8]$. This notation means that, e.g., $\mu^{-1}(1) = \{1,2,4,5\}$ and $\mu^{-1}(2) = \{6,7\}$ and $\mu^{-1}(3) = \{3,8\}$.
Then 
\[ \lambda \mathrel{\textnormal{\leo}} \mu = [1,2][3][4,5][6,7][8], \quad \mu \mathrel{\textnormal{\leo}} \lambda = [1,2][4,5][6,7][3][8].\] 

Again suppose that $Q_0$ is a singleton, and take $\bx{c} = 10$, $\lambda = [3,5,6][1,4][2,8][7,9,10]$ and $\mu = [1, 5, 9, 10] [ 2,3,4,6,7,8]$. Then 
\[ \lambda \mathrel{\textnormal{\leo}} \mu = [5][3,6][1][4][2,8][9,10][7], \quad \mu \mathrel{\textnormal{\leo}} \lambda = [5][1][9,10][3,6][4][2,8][7]. \]
We see that $\lambda \mathrel{\textnormal{\leo}} \mu$ and $\mu \mathrel{\textnormal{\leo}} \lambda$ are of the same length and differ only by a permutation. 
\end{exa}

\begin{exa} \label{exa: int two}
Suppose that $Q$ is the $A_3$ quiver from Example \ref{first example} and $\bx{c} = 5i_1 + 4i_2+3i_3$. Let $\lambda$ be given by 
$\lambda^{-1}(1) = \{ (1,i_1), (2,i_1), (3,i_1), (1,i_3), (2,i_3)\}$, $\lambda^{-1}(2) = \{ (1,i_2), (2,i_2), (3,i_3) \}$ and  $\lambda^{-1}(3) = \{ (4,i_1), (5,i_1), (3,i_2), (4,i_2)\}$. 
We assign the colour black to vertex $i_1$,  blue to $i_2$ and red to $i_3$ and rewrite $\lambda$ in the following more readable notation: $\lambda = [1, 2, 3,{\color{red} 1}, {\color{red} 2}][{\color{blue} 1}, {\color{blue} 2}, {\color{red} 3}][4, 5,{\color{blue} 3}, {\color{blue} 4}]$. 
Let $\mu$ be a second partitioning given by $\mu = [1,4, {\color{blue} 1},  {\color{blue} 3},  {\color{blue} 4}, {\color{red}{3}}][2,3,5, {\color{blue} 2}, {\color{red} 1}, {\color{red} 2}]$. 
Then
\[ 
\lambda \mathrel{\textnormal{\leo}} \mu =  [1][2,3, {\color{red} 1}, {\color{red} 2}][{\color{blue} 1}, {\color{red} 3}][{\color{blue} 2}][4, {\color{blue} 3},  {\color{blue} 4}][5], \quad 
\mu \mathrel{\textnormal{\leo}} \lambda = [1][{\color{blue} 1}, {\color{red} 3}][4, {\color{blue} 3},  {\color{blue} 4}][2,3, {\color{red} 1}, {\color{red} 2}][{\color{blue} 2}][5].
\]
\end{exa}

\begin{defi} \label{defi: orbit datum}
We call a triple $(\ux{e}, \ux{d}, w)$, consisting of $\ux{e},\ux{d} \in \mathbf{Com}_{\bx{c}}$ and $w \in \dc{e}{d}{c}$, an \emph{orbit datum}. This name is motivated by the fact that orbit data naturally label the $\Gx{G}{c}$-orbits in $\Sx{F}{c} \times \Sx{F}{c}$. 
Abbreviating $\lambda = P(\ux{e})$ and $\mu = w \cdot  P(\ux{d})$, we also define 
\[ \widehat{\ux{e}} := C(\lambda \mathrel{\leo} \mu), \quad \widehat{\ux{d}} := C(\mu \mathrel{\leo} \lambda). \] 
By Lemma \ref{lem: comb cap}.b), $\widehat{\ux{e}}$ is a refinement of $\ux{e}$ and $\widehat{\ux{d}}$ is a refinement of $\ux{d}$.
By Lemma \ref{lem: comb cap}.a), the partitionings $\mu \mathrel{\leo} \lambda$ and $\lambda \mathrel{\leo} \mu$ are of the same length $n$ and lie in the same orbit of $\mathsf{Sym}_n$. Since the $\mathsf{Sym}_n$-action on $\mathbf{Par}_{\bx{c}}^n$ is free, there exists a unique permutation $u \in \mathsf{Sym}_n$ which sends $\lambda \mathrel{\leo} \mu$ to $\mu \mathrel{\leo} \lambda$. 
We call the triple $(\widehat{\ux{e}}, \widehat{\ux{d}}, u)$ the \emph{refinement datum} corresponding to $(\ux{e}, \ux{d}, w)$.
\end{defi} 
\begin{exa} \label{exa: e hat}
Let $Q_0$ be a singleton, $\bx{c}=8$, $\ux{e} = (3,2,3)$, $\ux{d} = (4,2,2)$ and $w = s_3s_4s_5s_6$. Then 
$\lambda = P(\ux{e}) =  [1,2,3][4,5][6,7,8]$ and $\mu = w\cdot P(\ux{d}) =  [1,2,4,5][6,7][3,8]$, 
i.e., $\lambda$ and $\mu$ are as in the first case considered in Example \ref{exa: int}. 
Hence 
$\widehat{\ux{e}} = (2,1,2,2,1)$, $\widehat{\ux{d}}  = (2,2,2,1,1)$ and $u =  s_3s_2 \in \mathsf{Sym}_5$.  
\end{exa}

\begin{exa} \label{exa: e hat two}
Let $Q$ be the $A_3$ quiver, $\bx{c} = 5i_1 + 4i_2+3i_3$, $\ux{e} = (3i_1+2i_3, 2i_2+i_3, 2i_1,2i_2)$, $\ux{d} = (2i_1,3i_2+i_3, 3i_1+i_2+2i_3)$ and $w = (s_3s_2)({\color{blue}s_2s_3})({\color{red}s_2s_1}) \in \Gx{W}{c} = \mathsf{Sym}_5 \times {\color{blue} \mathsf{Sym}_4} \times {\color{red} \mathsf{Sym}_3}$. Then $\lambda = P(\ux{e}) = [1, 2, 3,{\color{red} 1}, {\color{red} 2}][{\color{blue} 1}, {\color{blue} 2}, {\color{red} 3}][4, 5,{\color{blue} 3}, {\color{blue} 4}]$ and $\mu = w\cdot P(\ux{d}) = [1,4, {\color{blue} 2},  {\color{blue} 3},  {\color{blue} 4}, {\color{red}{3}}][2,3,5, {\color{blue} 1}, {\color{red} 1}, {\color{red} 2}]$, i.e., $\lambda$ and $\mu$ are as in Example \ref{exa: int two}. Hence $\widehat{\ux{e}} = (i_1, 2i_1+2i_3, i_2+i_3,i_2,i_1+2i_2,i_1)$,  $\widehat{\ux{d}} = (i_1, i_2+i_3,i_1+2i_2,2i_1+2i_3,i_2,i_1)$ and $u = s_3s_2s_4 \in \mathsf{Sym}_6$. We rewrite the vector compositions $\widehat{\ux{e}}$ and $\widehat{\ux{d}}$ in the following more readable notation  
\[  \widehat{\ux{e}} = (1, 2+{\color{red}2},{\color{blue}1}+ {\color{red}1},{\color{blue}1},1+{\color{blue}2},1), \quad  \widehat{\ux{d}} = (1, {\color{blue}1}+ {\color{red}1},1+{\color{blue}2},2+{\color{red}2},{\color{blue}1},1).\]
Again, we see that $\widehat{\ux{e}}$ and $\widehat{\ux{d}}$ are of the same length and differ only by a permutation. 
\end{exa}

\begin{lem} \label{lem: e hat in P}
Let $\lambda = P(\ux{e})$ and $\mu = w \cdot  P(\ux{d})$ as in Definition \ref{defi: orbit datum}. 
Then $P(\widehat{\ux{e}}) = \lambda \mathrel{\textnormal{\leo}} \mu$. 
\end{lem}

\begin{proof}
A partitioning $\nu$ is in the image of $P$ if and only if for all $1 \leq r \leq \ell_\nu-1$, $i \in Q_0$ and $(x, i) \in \nu_i^{-1}(r)$, $(y, i) \in \nu^{-1}(r+1)$, we have $x < y$. Since $\lambda$ is in the image of $P$, and $w$ is a shortest coset representative, $\lambda \mathrel{\textnormal{\leo}} \mu$ satisfies this condition. Hence $\lambda \mathrel{\textnormal{\leo}} \mu = P(\ux{f})$ for some $\ux{f} \in \mathbf{Com}_{\bx{c}}$. But then $ P(\widehat{\ux{e}}) = P \circ C(\lambda \mathrel{\textnormal{\leo}} \mu) = P \circ C \circ P(\ux{f}) = P(\ux{f}) = \lambda \mathrel{\textnormal{\leo}} \mu$. 
\end{proof}

\begin{defi} \label{defi: crossing datum}
Given a refinement datum $(\widehat{\ux{e}}, \widehat{\ux{d}}, u)$, let us choose a reduced expression $u = s_{j_k} \cdot \hdots \cdot s_{j_1}$, where $k=\ell(u)$. Set $u_l = s_{j_l} \cdot \hdots \cdot s_{j_1}$ and let $\ux{e}^{2l} = u_l(\widehat{\ux{e}})$, $\ux{e}^{2l-1} = \wedge_{j_l}(\ux{e}^{2l-2})$, for $1 \leq l \leq k$, with $\ux{e}^0 =\widehat{\ux{e}}$. Observe that $\ux{e}^{2k} = \widehat{\ux{d}}$. We call $(\ux{e}^{0}, \hdots, \ux{e}^{2k})$ a \emph{crossing datum} associated to $(\widehat{\ux{e}}, \widehat{\ux{d}}, u)$. 
\end{defi} 

The following diagram illustrates the relationships between the different vector compositions in a crossing datum. Vector compositions in the same row (possibly except for $\ux{e}$ and $\ux{d}$) are of the same length. 
\[
\begin{tikzpicture}[baseline= (a).base]
\node[scale=.8] (a) at (0,0){
\begin{tikzcd}[row sep=tiny, column sep=small]
\ux{e}           & &                                           \ux{e}^{1} & & \ux{e}^{3} & \hdots & \ux{e}^{2k-1} &&  \ux{d} \\
            & \ux{e}^{0} \arrow[ul, Succ] \arrow[ru, Succ] & &  \ux{e}^{2} \arrow[lu, Succ]  \arrow[ru, Succ] & & & & \ux{e}^{2k}  \arrow[lu, Succ]  \arrow[ru, Succ]
\end{tikzcd}
};
\end{tikzpicture}\]

\begin{exa}
Consider the first case from Example \ref{exa: e hat}. Given the reduced expression $u =  s_3s_2 \in \mathsf{Sym}_5$, we have $\ux{e}^1 = (2,3,2,1)$, $\ux{e}^2 = (2,2,1,2,1)$, $\ux{e}^3 = (2,2,3,1)$. 
Next, consider Example \ref{exa: e hat two}. Given the reduced expression $u = s_3s_2s_4 \in \mathsf{Sym}_6$, we have 
$\ux{e}^1 = (1, 2+{\color{red}2},{\color{blue}1}+ {\color{red}1}, 1+{\color{blue}3},1)$, 
$\ux{e}^2 = (1, 2+{\color{red}2},{\color{blue}1}+ {\color{red}1}, 1+{\color{blue}2},{\color{blue}1},1)$, 
$\ux{e}^3 = (1, 2+{\color{blue}1}+ {\color{red}3},1+{\color{blue}2},{\color{blue}1},1)$, 
$\ux{e}^4 = (1, {\color{blue}1}+ {\color{red}1},2+{\color{red}2},1+{\color{blue}2},{\color{blue}1},1)$ and 
$\ux{e}^5 = (1, {\color{blue}1}+ {\color{red}1},3+{\color{blue}2}+{\color{red}2},{\color{blue}1},1)$. 
\end{exa}

We will now explain the connection between $w \in \Gx{W}{c}$ and $u = s_{j_k} \cdot \hdots \cdot s_{j_1} \in \mathsf{Sym}_n$. Let $w_l$ denote the longest element in $\mathsf{D}_{\ux{e}^{2l}}^{\ux{e}^{2l-1}}$ and let $\widetilde{u} = w_1 \cdot \hdots \cdot w_k$.

\begin{prop} \label{lem: p vs w} \label{lem: Sym refined}
The following hold: 
\begin{enumerate}[label=\alph*), font=\textnormal,noitemsep,topsep=3pt,leftmargin=1cm]
\item $\ell(\widetilde{u}) = \ell(w_1) + \hdots + \ell(w_k)$ and $w_1 \cdot \hdots \cdot w_l \in \mathsf{D}_{\ux{e}^{2l}}^\mathbf{c}$, for all $ 1 \leq l \leq k$. 
\item $\widetilde{u} = w$. 
\item $\mathsf{W}_{\ux{e}^{0}} = \Gxx{W}{e} \cap (w \Gxx{W}{d} w^{-1})$.
\end{enumerate}
\end{prop}

\begin{proof}
We start by proving the first statement in part a). To simplify notation, we assume (without loss of generality) that $Q_0$ is a singleton. We divide the interval $[1,\bx{c}]$ into blocks (i.e.\ subintervals) of size $\widehat{\bx{e}}_1, \hdots, \widehat{\bx{e}}_n$. Let $B_1, \hdots, B_n$ be the blocks. The permutation $u$ acts by permuting these blocks. Let $\Inv = \{ (i,j) \in [1,\bx{c}]^2 \mid i < j, \ \widetilde{u}(i) > \widetilde{u}(j)\}$ and $\Inv_B = \{ (i,j) \in [1,n]^2 \mid B_i < B_j, \ u(B_i) > u(B_j)\}$. 
The length of $\widetilde{u}$ equals the number of inversions, i.e., $\ell(\widetilde{u}) = \lvert\Inv\rvert$. Since $\widetilde{u}$ permutes blocks but does not change the order inside blocks, we have $\lvert\Inv\rvert = \sum_{(i,j) \in \Inv_B} \widehat{\bx{e}}_i \cdot \widehat{\bx{e}}_j$. We can identify each inversion $(i,j)$ in $\Inv_B$ with some simple transposition $s_{j_l}$ in the reduced expression $u = s_{j_k} \cdot \hdots \cdot s_{j_1}$. But then $\widehat{\bx{e}}_i \cdot \widehat{\bx{e}}_j = \ell(w_l)$, proving the statement. For the second statement in part a), note that $w_1 \cdot \hdots \cdot w_l$ also permutes blocks but does not change the order inside blocks. But any $x \in \mathsf{W}_{\ux{e}^{2l}}$ permutes numbers within blocks, increasing the number of inversions. Hence $\ell(w_1 \cdot \hdots \cdot w_l \cdot x) > \ell(w_1 \cdot \hdots \cdot w_l)$, which implies that $w_1 \cdot \hdots \cdot w_l \in \mathsf{D}_{\ux{e}^{2l}}^\mathbf{c}$.

Let us prove part b). 
We have the following chain of equalities
\begin{equation} \label{length lemma} w \cdot P(\widehat{\ux{d}}) = w \cdot P(\ux{d}) \mathrel{\textnormal{\leo}} P(\ux{e}) = u \cdot P(\widehat{\ux{e}}) = \widetilde{u} \cdot P(u \cdot \widehat{\ux{e}}) = \widetilde{u} \cdot P(\widehat{\ux{d}}).\end{equation}
For the first equality, note that $\widehat{\ux{d}} = C(w \cdot P(\ux{d}) \mathrel{\textnormal{\leo}} P(\ux{e})) = C(P(\ux{d}) \mathrel{\textnormal{\leo}} w^{-1}\cdot P(\ux{e}))$. Hence, switching the roles of $\ux{e}$ and $\ux{d}$ in Lemma \ref{lem: e hat in P}, we get $P(\widehat{\ux{d}}) =  P(\ux{d}) \mathrel{\textnormal{\leo}} w^{-1}\cdot P(\ux{e})$. After acting on both sides by $w$ we get the first equality. The second equality follows directly from the definition of $u$, while the third equality follows from a repeated application of Lemma \ref{lem: C P functions}.f). The final equality holds since 
$u \cdot \widehat{\ux{e}} = u \cdot C(P(\ux{e}) \mathrel{\textnormal{\leo}} w \cdot P(\ux{d})) = C(u (\cdot P(\ux{e}) \mathrel{\textnormal{\leo}} w \cdot P(\ux{d}))) = C(w \cdot P(\ux{d}) \mathrel{\textnormal{\leo}} P(\ux{e})) = \widehat{\ux{d}}$. 
The claim that $\widetilde{u} = w$ now follows from \eqref{length lemma}, the fact that $\Stab_{\Gx{W}{c}}(P(\widehat{\ux{d}})) = \mathsf{W}_{\widehat{\ux{d}}}$ and that both $w$ and $\widetilde{u}$ are in  $\mathsf{D}_{\widehat{\ux{d}}}^{\mathbf{c}}$. 

Part c) follows from the following chain of equalities
\begin{align*}
\mathsf{W}_{\ux{e}^{0}} = \Stab_{\Gx{W}{c}}(P \circ C(\lambda \mathrel{\textnormal{\leo}} \mu)) =  \Stab_{\Gx{W}{c}}(\lambda \mathrel{\textnormal{\leo}} \mu) = \Gxx{W}{e} \cap (w \Gxx{W}{d} w^{-1}), 
\end{align*}
where $\lambda = P(\ux{e})$ and $\mu = w \cdot P(\ux{d})$. The first equality follows from the fact that, by definition, $\ux{e}^{0} = \widehat{\ux{e}} = C(\lambda \mathrel{\textnormal{\leo}} \mu)$, and Lemma \ref{lem: C P functions}.d). 
The second equality follows from Lemma \ref{lem: e hat in P} while the third equality follows from Lemma \ref{lem: comb cap}.c) and Lemma \ref{lem: C P functions}.d).  
\end{proof}

We have so far treated the ordered intersection operation $\mathrel{\textnormal{\leo}}$ on partitionings as a combinatorial device. 
However, as indicated in Lemma \ref{lem: partitionings and flags}, partitionings also have a geometric meaning since they correspond to coordinate flags in $\mathbf{V}_{\bx{c}}$. Based on this observation, we will extend the operation $\mathrel{\textnormal{\leo}}$ to all flags in $\Sx{F}{c}$. 

\begin{defi}
Let $F = (F_i)_{i=0}^{\ell_{\ux{e}}} \in \Sxx{F}{e}$ and $F' = (F_j)_{j=0}^{\ell_{\ux{d}}} \in \Sxx{F}{d}$ be two flags. Let $R_{i,j} =  F_i \cap F'_j$ and $S_{F,F'} = \{ R_{i,j} \mid 0 \leq i \leq \ell_{\ux{e}}, \ 0 \leq j \leq \ell_{\ux{d}} \}$. We put a total order on the set $S_{F,F'}$ by declaring that $R_{i,j} < R_{r,s}$ if and only if $r>i$ or $r=i$ and $j < s$. We then define the \emph{ordered intersection}  $F\mathrel{\textnormal{\leo}}F'$ of flags $F$ and $F'$ by setting $(F\mathrel{\textnormal{\leo}}F')_m$ to be the $m$-th element of $S_{F,F'}$ with respect to the total order defined above, and deleting all the repeated occurrences of subspaces. 
\end{defi}

It is clear that $F\mathrel{\textnormal{\leo}}F'$ is a flag in $\Sx{F}{c}$. Moreover, if $(F,F') \in \O_w^{\Delta}$ then $F\mathrel{\textnormal{\leo}}F' \in \mathfrak{F}_{\widehat{\ux{e}}}$. 

\begin{lem} \label{lem: intersection and stability}
Let $\rho \in \Sx{R}{c}$. 
If $F \in \Sxx{F}{e}$ and $F' \in \Sxx{F}{d}$ are $\rho$-stable, then so are $F\mathrel{\textnormal{\leo}}F'$ and $F'\mathrel{\textnormal{\leo}}F$. 
\end{lem}

\begin{proof}
Since $F$ and $F'$ are $\rho$-stable, each intersection $R_{i,j} =  F_i \cap F'_j$ is preserved by $\rho$, which implies that $F\mathrel{\textnormal{\leo}}F'$ and $F'\mathrel{\textnormal{\leo}}F$ are also $\rho$-stable. 
\end{proof}

\subsection{Basis and generators} \label{sec: basis and gens}

In this section we state a basis theorem for the quiver Schur algebra and use it to find a convenient generating set for our algebra.

\begin{defi} \label{defi: basis elements for QS}
Let $(\ux{e}, \ux{d}, w)$ be an orbit datum with the corresponding refinement datum $(\widehat{\ux{e}}, \widehat{\ux{d}}, u)$. Let us fix a reduced decomposition $u= s_{j_k} \cdot \hdots \cdot s_{j_1}$, which determines the corresponding crossing datum $(\ux{e}^{0}, \hdots, \ux{e}^{2k})$. 
Define 
\[ \textstyle
 \ux{d} \underset{w}{\Longrightarrow} \ux{e} \quad := \quad \bigcurlywedge_{\ux{e}^{0}}^{\ux{e}} \star \cross_{\ux{e}^{2}}^{j_1} \star \hdots \star \cross_{\ux{e}^{2k}}^{j_k} \star  \bigcurlyvee_{\ux{d}}^{\ux{e}^{2k}} \in \Axxy{Z}{e}{d}. 
\] 
Given $c \in \mathcal{Z}_{\ux{e}^{2k},\ux{e}^{2k}}^e \cong \Lambda_{\ux{e}^{2k}}$, also define 
\begin{equation} \label{basis el c} \textstyle
 \ux{d} \underset{w}{\overset{c}{\Longrightarrow}} \ux{e} \quad := \quad \bigcurlywedge_{\ux{e}^{0}}^{\ux{e}} \star \cross_{\ux{e}^{2}}^{j_1} \star \hdots \star \cross_{\ux{e}^{2k}}^{j_k} \star c \star \bigcurlyvee_{\ux{d}}^{\ux{e}^{2k}} \in \Axxy{Z}{e}{d}. 
\end{equation}
\end{defi}

We now state the basis theorem for the quiver Schur algebra.

\begin{thm} \label{thm:basis} 
If we let $(\ux{e}, \ux{d}, w)$ range over all orbit data and $c$ range over a basis of $\Lambda_{\ux{e}^{2k}}$, then the elements $\ux{d} \underset{w}{\overset{c}{\Longrightarrow}} \ux{e}$ form a basis of $\Ax{Z}{c}$. 
\end{thm}

\begin{rem} 
We make a few remarks regarding the theorem. 
\begin{enumerate}[label=(\roman*),topsep=2pt,itemsep=1pt]
\item The basis in Theorem \ref{thm:basis} depends on the choice of a crossing datum, i.e., the choice of a reduced expression for $u$. One would also obtain other bases by letting $c$ range over a basis of $\mathcal{Z}_{\ux{e}^{2l},\ux{e}^{2l}}^e$, for any $1 \leq l \leq k$, and redefining the elements \eqref{basis el c} appropriately. 
\item The basis in Theorem \ref{thm:basis} is natural from a geometric point of view. As we will see in \S \ref{subsec: proof of basis}, it is related to certain generalizations of Bott-Samelson varieties. For this reason, we call it a \emph{Bott-Samelson basis} of the quiver Schur algebra. Our basis also admits a natural interpretation in terms of cohomological Hall algebras (see Theorem \ref{thm: TH vs Pol}). 
\item Theorem \ref{thm:basis} is analogous to the basis theorem \cite[Theorem 3.13]{SW} for the Stroppel-Webster quiver Schur algebra.  However, the combinatorics of residue sequences developed in \cite{SW} is not sufficient to correctly characterize the refined vector compositions  used in the definition of the basis. Our combinatorics of partitionings (\S \ref{subsec: refinements}) fixes this problem.
\end{enumerate}
\end{rem}

Using Theorem \ref{thm:basis}, we can find a generating set for the quiver Schur algebra.

\begin{cor} \label{cor:el spl mer}
Elementary merges, elementary splits and the polynomials $\Ax{Z}{c}^e$ generate $\Ax{Z}{c}$ as an algebra. 
\end{cor}

\begin{proof}
This follows directly from Theorem \ref{thm:basis} and Proposition \ref{pro:transitivity}. 
\end{proof}

\subsection{Proof of the basis theorem.} \label{subsec: proof of basis}
We begin by proving four technical lemmas. 
Set $\O_w := \Gxx{P}{e}w\Gxx{P}{d}/\Gxx{P}{d}$. 
Consider the following parabolic analogue of the Bott-Samelson variety: 
\[ \mathfrak{BS}_{\ux{e},\ux{d},w} := \mathsf{P}_{\ux{e}} \times_{\mathsf{P}_{\ux{e}^{0}}} \mathsf{P}_{\ux{e}^{1}} \times_{\mathsf{P}_{\ux{e}^{2}}} \hdots \times_{\mathsf{P}_{\ux{e}^{2k-2}}} \mathsf{P}_{\ux{e}^{2k-1}} \times_{\mathsf{P}_{\ux{e}^{2k}}} \Gxx{P}{d} / \Gxx{P}{d}. \] 
We have a commutative diagram
\begin{equation} \label{Bott-S diag}
\begin{tikzcd}
\mathsf{P}_{\ux{e}} \times \mathsf{P}_{\ux{e}^{1}} \times \mathsf{P}_{\ux{e}^{3}} \times \hdots \times \mathsf{P}_{\ux{e}^{2k-1}} \times \Gxx{P}{d} \arrow[r,"m"] \arrow[d,swap, "can_1"] & \Gx{G}{c} \arrow[d, "can_2"] \\
\mathfrak{BS}_{\ux{e},\ux{d},w} \arrow[r,"\phi"] & \Sxx{F}{d}
\end{tikzcd}
\end{equation}
where $m$ is the multiplication map and $\phi$ is the induced map.

\begin{lem} \label{lem: basis proof 1}
The map $\phi$ is proper, its image equals $\overline{\O_w}$, and it restricts to an isomorphism over~$\O_w$. 
\end{lem}

\begin{proof}
The proof is similar to the proof for the usual Bott-Samelson resolution, where the parabolics $\Gxx{P}{e}$ and $\Gxx{P}{d}$ are replaced with a Borel subgroup. Since we could not find an explicit reference, and since the proof relies on the combinatorics from \S \ref{subsec: refinements}, we sketch it below.

It is clear that $can_1$ and $can_2$ are proper. The multiplication map $(\Gx{G}{c})^{k+2} \to \Gx{G}{c}$ is proper and hence its restriction to the closed submanifold $\mathsf{P}_{\ux{e}} \times \mathsf{P}_{\ux{e}^{1}} \times \hdots \times \mathsf{P}_{\ux{e}^{2k-1}} \times \Gxx{P}{d}$ is proper as well. Since $can_1$ is a locally trivial fibration, and properness is a local property, it follows that $\phi$ is proper, as required. This proves the first statement in the lemma. 

By the Bruhat decomposition, we have 
\[ \mathsf{P}_{\ux{e}^{2l-1}} = \bigsqcup_{x \in \mathsf{D}^{\ux{e}^{2l-1}}_{\ux{e}^{2l}}} \mathsf{B}_{\bx{c}} x  \mathsf{P}_{\ux{e}^{2l}}. \]
Proposition \ref{lem: p vs w}.a) implies that 
\[ \mathsf{P}_{\ux{e}^{2l-1}} \cdot \mathsf{P}_{\ux{e}^{2l+1}} = \bigcup_{x \in \mathsf{D}^{\ux{e}^{2l-1}}_{\ux{e}^{2l}}} \mathsf{B}_{\bx{c}} x  \mathsf{P}_{\ux{e}^{2l+1}} = \bigsqcup_{x \leq w_{l}w_{l+1} \in \mathsf{D}^{\bx{c}}_{\ux{e}^{2l+2}}} \mathsf{B}_{\bx{c}} x  \mathsf{P}_{\ux{e}^{2l+2}}. \]
By induction and Proposition \ref{lem: p vs w}.b), we get 
\begin{equation} \label{bas pr 1} \mathsf{P}_{\ux{e}} \cdot \mathsf{P}_{\ux{e}^{1}} \cdot \hdots \cdot \mathsf{P}_{\ux{e}^{2k-1}} \cdot \mathsf{P}_{\ux{d}} = \bigsqcup_{x \leq w \in \dc{e}{d}{c}} \mathsf{P}_{\ux{e}} x \mathsf{P}_{\ux{d}}. \end{equation}
This implies that the image of $can_2 \circ m$ is $\overline{\O_w}$. Since $can_1$ is surjective, this is also the image of $\phi$. 
This proves the second statement of the lemma. 

Next, we claim that 
\begin{equation} \label{bas pr 1}  \mathsf{P}_{\ux{e}} w \mathsf{P}_{\ux{d}} = \bigcup_{y \in \Gxx{W}{e}} \mathsf{B}_{\bx{c}} yw \mathsf{P}_{\ux{d}} =  \bigsqcup_{y \in \mathsf{D}^{\ux{e}}_{\ux{e}^{0}}} \mathsf{B}_{\bx{c}} yw \mathsf{P}_{\ux{d}}. \end{equation} 
Let us first show that the union on the RHS of \eqref{bas pr 1} is disjoint. 
Suppose that $y_1 w = y_2 w x$ for some $y_1, y_2 \in \mathsf{D}^{\ux{e}}_{\ux{e}^{0}}$ and $x \in \Gxx{W}{d}$. Then $y_2^{-1}y_1 = w x w^{-1}$, which, by Proposition \ref{lem: Sym refined}.c), implies that $y_2^{-1}y_1 \in \Gxx{W}{e} \cap w \Gxx{W}{d} w^{-1} = \mathsf{W}_{\ux{e}^{0}}$. Hence $y_1 = y_2$, so the union is indeed disjoint. The first equality in \eqref{bas pr 1} follows from the Bruhat decomposition of $\Gxx{P}{e}$ and the fact that $w \in \dc{e}{d}{c}$. 
Next, if $y \in \Gxx{W}{e}$, we can write it as $y = ab$ with $a \in \mathsf{D}^{\ux{e}}_{\ux{e}^{0}}$ and $b \in \mathsf{W}_{\ux{e}^{0}}$. But then $yw\Gxx{P}{d} = aw\Gxx{P}{d}$, which yields the second equality in \eqref{bas pr 1}. 

Given $y \in \mathsf{D}^{\ux{e}}_{\ux{e}^{0}}$, let $\widetilde{U}_y := \mathsf{U}_{y^{-1}} y \times \mathsf{U}_{w_{1}^{-1}} w_{1} \times \hdots \times \mathsf{U}_{w_{k}^{-1}} w_{k}$. 
Set $\widetilde{U} = \bigsqcup_{y \in \mathsf{D}^{\ux{e}}_{\ux{e}^{0}}} \widetilde{U}_y$ and $U = can_1(\widetilde{U})$. 
It is easy to check that $m$ maps $\widetilde{U}_y$ isomorphically onto $\mathsf{U}_{(yw)^{-1}} yw$. 
It is also well known that the map 
\[ \mathsf{U}_{(yw)^{-1}} \times \Gxx{P}{d} \to \Gx{B}{c} yw\Gxx{P}{d}, \quad (v,p) \mapsto v(yw)p, \]
is an isomorphism. Hence $can_2 \circ m$ maps $\widetilde{U}$ isomorphically onto $\O_w$. 
By the commutativity of the diagram \eqref{Bott-S diag}, $\phi|_U$ is also an isomorphism onto $\O_w$. 
An easy argument again based on the Bruhat decomposition also shows that $U = \phi^{-1}(\O_w)$. This completes the proof of the lemma. 
\end{proof}

Let $\O_w^{\Delta} := \Gx{G}{c}\cdot(e\Gxx{P}{e}, w\Gxx{P}{d})$ be the diagonal $\Gx{G}{c}$-orbit corresponding to $w$ and define 
\[ \mathfrak{BS}_{\ux{e},\ux{d},w}^{\Delta} := \mathfrak{F}_{\ux{e}^{0}} \times_{\mathfrak{F}_{\ux{e}^{1}}} \mathfrak{F}_{\ux{e}^{2}} \times_{\mathfrak{F}_{\ux{e}^{3}}} \hdots \times_{\mathfrak{F}_{\ux{e}^{2k-1}}} \mathfrak{F}_{\ux{e}^{2k}}. \]
Let $p \colon \mathfrak{BS}_{\ux{e},\ux{d},w}^{\Delta} \to \mathfrak{F}_{\ux{e}^{0}} \to \Sxx{F}{e}$ and $q \colon \mathfrak{BS}_{\ux{e},\ux{d},w}^{\Delta} \to \mathfrak{F}_{\ux{e}^{2k}} \to \Sxx{F}{d}$ be the canonical maps. Set $\psi:=p \times q$. Consider the following diagram of $\Gx{G}{c}$-equivariant maps. 
\[
\begin{tikzcd}[column sep = tiny]
\mathfrak{BS}_{\ux{e},\ux{d},w}^{\Delta} \arrow[rr, "\psi"] \arrow[rd, "p", swap] & &  \Sxx{F}{e} \times \Sxx{F}{d} \arrow[ld, "can"] \\
 & \Sxx{F}{e}
\end{tikzcd}
\]
\begin{lem} \label{lem: psi proper}
The map $\psi$ is proper, its image equals $\overline{\O_w^{\Delta}}$, and it restricts to an isomorphism over~$\O_w^{\Delta}$. 
\end{lem}

\begin{proof}
It is easy to check that $p$ is a locally trivial fibration with fibre $\mathfrak{BS}_{\ux{e},\ux{d},w}$. Let $F_e = p^{-1}(e\Gxx{P}{e}/\Gxx{P}{e})$. 
Lemma \ref{lem: basis proof 1} implies that $\psi|_{F_e}$ is proper, $\psi(F_e) = \{ (e\Gxx{P}{e}/\Gxx{P}{e}, x) \mid x \in \overline{\O_w} \}$ and that the restriction of $\psi$ to $F_e\cap\psi^{-1}(\O_w^{\Delta})$ is an isomorphism onto $\{ (e\Gxx{P}{e}/\Gxx{P}{e}, x) \mid x \in \O_w \}$. The lemma now follows from the fact that $\psi$ is $\Gx{G}{c}$-equivariant. 
\end{proof}

Next consider the following iterated fibre product 
\begin{align*} 
\widetilde{\mathfrak{BS}}_{\ux{e},\ux{d},w}^{\Delta} := \mathfrak{Q}_{\ux{e}^{0}} \times_{\mathfrak{Q}_{\ux{e}^{1}}} \mathfrak{Q}_{\ux{e}^{2}} \times_{\mathfrak{Q}_{\ux{e}^{3}}} \hdots \times_{\mathfrak{Q}_{\ux{e}^{2k-1}}} \mathfrak{Q}_{\ux{e}^{2k}}.  
\end{align*}
The closed points of $\widetilde{\mathfrak{BS}}_{\ux{e},\ux{d},w}^{\Delta}$ correspond to sequences of flags $F^0, F^2, \hdots, F^{2k}$ (satisfying appropriate conditions) together with a quiver representation $\rho$ with respect to which each flag is stable. This implies that $\psi$ lifts to a map $\widetilde{\psi} \colon \widetilde{\mathfrak{BS}}_{\ux{e},\ux{d},w}^{\Delta} \to \Sxxy{Z}{e}{d}$ (sending $(F^0, \hdots, F^{2k}, \rho) \mapsto (F^0|_{\ux{e}},F^{2k}|_{\ux{d}},\rho))$, i.e., there is a commutative diagram 
\begin{equation} \label{pullback d}
\begin{tikzcd}
\widetilde{\mathfrak{BS}}_{\ux{e},\ux{d},w}^{\Delta} \arrow[r,"\widetilde{\psi}"] \arrow[d, swap, "\widetilde{\pi}_{\ux{e},\ux{d}}"] & \Sxxy{Z}{e}{d} \arrow[d,"\pi_{\ux{e},\ux{d}}"] \\
\mathfrak{BS}_{\ux{e},\ux{d},w}^{\Delta} \arrow[r, "\psi"] & \Sxx{F}{e} \times \Sxx{F}{d}
\end{tikzcd}
\end{equation}
where $\widetilde{\pi}_{\ux{e},\ux{d}}$ is the map forgetting the quiver representation. 

\begin{lem} 
The map $\widetilde{\psi}$ is proper, its image is contained in $\Sxxy{Z}{e}{d}^{\leqslant w}$, and it restricts to an isomorphism over $\Sxxy{Z}{e}{d}^{w}$.
\end{lem}

\begin{proof}
It is easy to see that $\widetilde{\mathfrak{BS}}_{\ux{e},\ux{d},w}^{\Delta}$ is a closed subset inside the fibre product of $\mathfrak{BS}_{\ux{e},\ux{d},w}^{\Delta}$ and $\Sxxy{Z}{e}{d}$ over $\Sxx{F}{e} \times \Sxx{F}{d}$. So $\widetilde{\psi}$ is the restriction of a base change of a proper map to a closed subset, and is therefore proper. By Lemma \ref{lem: psi proper}, $\Ima \psi = \overline{\O_w^{\Delta}}$, so $\Ima \widetilde{\psi} \subseteq \pi_{\ux{e},\ux{d}}^{-1}(\overline{\O_w^{\Delta}}) = \Sxxy{Z}{e}{d}^{\leqslant w}$. 

It remains to prove the third statement. Let $(F,F',\rho) \in \Sxxy{Z}{e}{d}^{w}$. By Lemma \ref{lem: psi proper}, we know that there exists a unique sequence of flags $(F^0, F^2, \hdots, F^{2k}) \in \mathfrak{BS}_{\ux{e},\ux{d},w}^{\Delta}$ such that $F = F^0|_{\ux{e}}$ and $F' = F^{2k}|_{\ux{d}}$. Since $F^0 = F\mathrel{\textnormal{\leo}}F'$ and $F^{2k} = F'\mathrel{\textnormal{\leo}}F$, Lemma \ref{lem: intersection and stability} implies that $F^0$ and $F^{2k}$ are $\rho$-stable. Let $1 \leq l \leq k-1$. Since $F^{2l}$ is in relative position $w_1\cdot \hdots \cdot w_l$ to $F^{0}$, $F^{2k}$ is in relative position $w_{l+1}\cdot \hdots \cdot w_k$ to $F^{2l}$ and, by Proposition \ref{lem: p vs w}, $\ell(w) = \ell(w_1\cdot \hdots \cdot w_l) + \ell(w_{l+1}\cdot \hdots \cdot w_k)$, $F^{2l}$ is also $\rho$-stable. Hence $(F^0, F^2, \hdots, F^{2k}, \rho) \in \widetilde{\mathfrak{BS}}_{\ux{e},\ux{d},w}^{\Delta}$ and so $\widetilde{\psi}$ is surjective. This clearly implies that $\widetilde{\psi}$ is an isomorphism. 
\end{proof}

We have the following Cartesian diagram. 
\[
\begin{tikzcd}
\Sxxy{Z}{e}{d}^w \arrow[hookrightarrow,r, "\iota_1"] & \Sxxy{Z}{e}{d}^{\leqslant w} \\
\widetilde{\psi}^{-1}(\Sxxy{Z}{e}{d}^{w}) \arrow[hookrightarrow,r, "\iota_2"] \arrow[u, "\wr"] & \widetilde{\mathfrak{BS}}_{\ux{e},\ux{d},w}^{\Delta} \arrow[u, swap, "\widetilde{\psi}"] 
\end{tikzcd}
\]
Let $q_0 \colon \mathfrak{Q}_{\ux{e}^{0}} \times \mathfrak{Q}_{\ux{e}^{2}} \times \hdots \times \mathfrak{Q}_{\ux{e}^{2k}} \to \mathfrak{Q}_{\ux{e}^{2k}}$ be the canonical map.

\begin{lem} \label{lem: basis proof 4}
Let $B(\Lambda_{\ux{d}^0})$ be a basis of $\Lambda_{\ux{d}^0} \cong H_\bullet^{\Gx{G}{c}}(\mathfrak{Q}_{\ux{d}^0}) $ and $c \in B(\Lambda_{\ux{d}^0})$. Then: 
\begin{enumerate}[label=\alph*), font=\textnormal,noitemsep,topsep=3pt,leftmargin=1cm]
\item $\widetilde{\psi}_*(q_0^*c \cap [\widetilde{\mathfrak{BS}}_{\ux{e},\ux{d},w}^{\Delta}]) \ = \ \ux{d} \overset{c}{\underset{w}{\Longrightarrow}} \ux{e}$.  
\item $\{ \iota_1^*(\ux{d} \overset{c}{\underset{w}{\Longrightarrow}} \ux{e}) \mid c \in B(\Lambda_{\ux{d}^0})\}$ is a basis of $\Axxy{Z}{e}{d}^w$. 
\end{enumerate}
\end{lem}

\begin{proof}
The first part of the lemma follows directly from Lemma \ref{lem:conv of classes} and the fact that
\[ \widetilde{\mathfrak{BS}}_{\ux{e},\ux{d},w}^{\Delta} \cong \mathfrak{Z}_{\ux{e},\ux{e}^{0}}^e \times_{\mathfrak{Q}_{\ux{e}^{0}}} \mathfrak{Z}_{\ux{e}^{0},\ux{e}^{1}}^e \times_{\mathfrak{Q}_{\ux{e}^{1}}} \mathfrak{Z}_{\ux{e}^{1},\ux{e}^{2}}^e \times_{\mathfrak{Q}_{\ux{e}^{2}}} \hdots \times_{\mathfrak{Q}_{\ux{e}^{2k-1}}} \mathfrak{Z}_{\ux{e}^{2k-1},\ux{e}^{2k}}^e \times_{\mathfrak{Q}_{\ux{e}^{2k}}} \mathfrak{Z}_{\ux{e}^{2k},\ux{d}}^e.\]
For the second part, we have
\[
\iota_1^*(\ux{d} \overset{c}{\underset{w}{\Longrightarrow}} \ux{e}) = \iota_1^*\widetilde{\psi}_*(q_0^*c \cap [\widetilde{\mathfrak{BS}}_{\ux{e},\ux{d},w}^{\Delta}]) = \widetilde{\psi}_*\iota_2^*(q_0^*c \cap [\widetilde{\mathfrak{BS}}_{\ux{e},\ux{d},w}^{\Delta}]) = \widetilde{\psi}_*(q_0^*c \cap [\widetilde{\psi}^{-1}(\Sxxy{Z}{e}{d}^{w})]), 
\]
where the second equality follows by proper base change. 
It now suffices to check that $q_0^*c \cap [\widetilde{\psi}^{-1}(\Sxxy{Z}{e}{d}^{w})]$ form a basis of $H^\bullet_{\Gx{G}{c}}(\widetilde{\psi}^{-1}(\Sxxy{Z}{e}{d}^{w}))$. 
Since the variety $\widetilde{\psi}^{-1}(\Sxxy{Z}{e}{d}^{w})$ is an affine bundle over $\O_w^{\Delta} \cong \Gx{G}{c}/(w^{-1}\Gxx{P}{e}w \cap \Gxx{P}{d})$, which itself is an affine bundle over $\mathfrak{F}_{\ux{d}^0}$, the restriction of $q_0$ to $\widetilde{\psi}^{-1}(\Sxxy{Z}{e}{d}^{w})$ is a homotopy equivalence. Hence the map $\Lambda_{\ux{d}^0} \to H^\bullet_{\Gx{G}{c}}(\widetilde{\psi}^{-1}(\Sxxy{Z}{e}{d}^{w}))$ sending $c \mapsto q_0^*c \cap [\widetilde{\psi}^{-1}(\Sxxy{Z}{e}{d}^{w})]$ is an isomorphism, and sends the basis $B(\Lambda_{\ux{d}^0})$ to a basis of $H^\bullet_{\Gx{G}{c}}(\widetilde{\psi}^{-1}(\Sxxy{Z}{e}{d}^{w}))$, as required. 
\end{proof} 

We are now ready to prove Theorem \ref{thm:basis}. 

\begin{proof}[Proof of Theorem \ref{thm:basis}]
Refine the Bruhat order on $\dc{e}{d}{c}$ to a linear order, which we denote by $\trianglelefteq$.
Given $w \in \dc{e}{d}{c}$, let $\Sxxy{Z}{e}{d}^{\trianglelefteq w} := \bigsqcup_{\dc{e}{d}{c} \ni u \trianglelefteq w} \Sxxy{Z}{e}{d}^u$ and $\Sxxy{Z}{e}{d}^{\triangleleft w} = \Sxxy{Z}{e}{d}^{\trianglelefteq w} \backslash \Sxxy{Z}{e}{d}^w$. Consider the inclusions
\begin{equation} \label{Z inclusions} \Sxxy{Z}{e}{d}^w \overset{i}{\hookrightarrow} \Sxxy{Z}{e}{d}^{\trianglelefteq w} \overset{j}{\hookleftarrow} \Sxxy{Z}{e}{d}^{\triangleleft w}. \end{equation}
It follows from Lemma \ref{lem: Z closed} that $i$ is an open embedding and $j$ is a closed embedding. Since the odd cohomology of $\Sxxy{Z}{e}{d}^w$ and $\Sxxy{Z}{e}{d}^{\triangleleft w}$ vanishes, the long exact sequence in $\Gx{G}{c}$-equivariant Borel-Moore homology associated to \eqref{Z inclusions} becomes a short exact sequence of $H_{\Gx{G}{c}}^\bullet(pt)$-modules 
\[ 0 \to H_\bullet^{\Gx{G}{c}}(\Sxxy{Z}{e}{d}^{\triangleleft w}) \to H_\bullet^{\Gx{G}{c}}(\Sxxy{Z}{e}{d}^{\trianglelefteq w}) \to H_\bullet^{\Gx{G}{c}}(\Sxxy{Z}{e}{d}^w) \to 0. \]
Since, by equivariant formality, the $H_{\Gx{G}{c}}^\bullet(pt)$-module $H_\bullet^{\Gx{G}{c}}(\Sxxy{Z}{e}{d}^w)$ is free, the short exact sequence splits. Arguing by induction on the refined Bruhat order, we conclude that 
\begin{equation} \label{Zed sum} \Axxy{Z}{e}{d} = \bigoplus_{w \in \dc{e}{d}{c}} \Axxy{Z}{e}{d}^w.\end{equation} 
By Lemma \ref{lem: basis proof 4}, the elements $\ux{d} \overset{c}{\underset{w}{\Longrightarrow}} \ux{e} \in \Axxy{Z}{e}{d}$ $(c \in B(\Lambda_{\ux{d}^0}))$ pull back to a basis of $\Axxy{Z}{e}{d}^w$. Therefore, \eqref{Zed sum} implies that, if we let $(\ux{e}, \ux{d}, w)$ range over all orbit data, the elements $\ux{d} \overset{c}{\underset{w}{\Longrightarrow}} \ux{e}$ indeed form a basis of $\Ax{Z}{c}$. 
\end{proof}

\section{The polynomial representation}

In this section we compute the polynomial representation of the quiver Schur algebra $\Ax{Z}{c}$, and use it to show that $\Ax{Z}{c}$ gives a geometric realization of the ``modified quiver Schur algebra'' from~\cite{MS}, thereby connecting $\Ax{Z}{c}$ to the affine $q$-Schur algebra. We also give a complete list of relations for the quiver Schur algebra associated to the $A_1$ and Jordan quivers. 

\subsection{$\Gx{T}{c}$-equivariant cohomology and localization}

We first recall some facts about $\Gx{T}{c}$-equivariant cohomology of flag varieties. 
By \cite[Theorem 3]{Tu}, there is a ring isomorphism 
\[ \Phi_{\ux{d}} \colon H^\bullet(B\Gx{T}{c}) \otimes_{H^\bullet(B\Gx{G}{c})} H^\bullet_{\Gx{G}{c}}(\Sxx{F}{d}) \cong H^\bullet_{\Gx{T}{c}}(\Sxx{F}{d}), \quad a \otimes b \mapsto (\gamma_1^*a)(\gamma_2^*b), \]
where $\gamma_1$ is the projection of $\Sxx{F}{d}$ onto a point and $\gamma_2$ is the canonical map $(\Sxx{F}{d})_{\Gx{T}{c}}  \twoheadrightarrow (\Sxx{F}{d})_{\Gx{G}{c}}$. 
Let us write a polynomial $f$ in variables $x_j(i)$ as $f(\vec{x})$. 
We abbreviate $x_j(i) := \Phi_{\ux{d}}(x_j(i) \otimes 1)$ and $f(\vec{y}) := \Phi_{\ux{d}}(1 \otimes f(\vec{x}))$ (we substitute variables $y_j(i)$ for $x_j(i)$). 

\begin{defi}
For $\underline{\mathbf{d}} \succ \underline{\mathbf{e}}$, we define the following polynomials in $\Lambda_{\ux{d}}$:
\begin{equation} \label{Sd classes}
\mathtt{S}_{\ux{d}} := \prod_{i \in Q_0} \prod_{r=1}^{\ell_{\ux{d}-1}} \prod_{k=\overset{\circ}{\mathbf{d}}_{r-1}(i)+1}^{\overset{\circ}{\mathbf{d}}_{r}(i)} \prod_{l=\overset{\circ}{\mathbf{d}}_r(i)+1}^{\mathbf{c}(i)} (x_l(i) - x_k(i)), \quad  \quad
\mathtt{S}_{\ux{d}}^{\ux{e}} := \frac{\mathtt{S}_{\ux{d}}}{\mathtt{S}_{\ux{e}}}.  
\end{equation}
\end{defi}

Note that $\mathtt{S}_{\ux{d}}$ is indeed $\Gxx{W}{d}$-invariant and $\mathtt{S}_{\ux{d}}^{\ux{e}}$ is a polynomial. Explicit examples of these polynomials for specific quivers and dimension vectors can be found in \cite[\S 8]{MS}.

It is well known that the fixed points $\Sxx{F}{d}^{\Gx{T}{c}}$ are parametrized by $\dcb{d}{c}$. Given $w \in \dcb{d}{c}$, let $i_w \colon \{w\} \hookrightarrow~\Sxx{F}{d}$ be the inclusion of the corresponding fixed point, and let $\zeta_w = [w] \in H_\bullet^{\Gx{T}{c}}(\Sxx{F}{d})$ denote the $\Gx{T}{c}$-equivariant fundamental class of this fixed point.

The theorem below summarizes the key facts about the $\Gx{T}{c}$-equivariant cohomology of flag varieties.

\begin{thm} \label{Tu theorem}
The following hold: 
\begin{enumerate}[label=\alph*), font=\textnormal,noitemsep,topsep=3pt,leftmargin=1cm]
\item The $\Gx{T}{c}$-equivariant cohomology of $\Sxx{F}{d}$ is equal to the quotient 
\[ H^\bullet_{\Gx{T}{c}}(\Sxx{F}{d}) = (\Ax{P}{c} \otimes \Phi_{\ux{d}}(\Lambda_{\ux{d}}) )/ I, \] 
where $I$ is the ideal generated by $p(\vec{y}) - p(\vec{x})$ as p ranges over all $\Gx{W}{c}$-invariant polynomials of positive degree. 
\item The pullback $i_w^* \colon H^*_{\Gx{T}{c}}(\Sxx{F}{d}) \to H^*(B\Gx{T}{c})$ is given by 
\[ i_w^*(x_j(i)) = x_j(i), \quad i_w^*(f(\vec{y})) = f(\vec{x}). \]
\item The $\Gx{T}{c}$-equivariant Euler class of the normal bundle to the fixed point $w$ is given by 
\[ \eu_{\Gx{T}{c}}(T_{\{w\}}\Sxx{F}{d}) = w \cdot \mathtt{S}_{\ux{d}}. \]
\end{enumerate} 
\end{thm}

\begin{proof}
See, e.g., \cite[Theorem 11]{Tu}. 
\end{proof}

Next, we recall the localization theorem for equivariant cohomology (see, e.g., \cite{Bri}). Let $\mathcal{K}_{\bx{c}}$ denote the fraction field of $\mathcal{P}_{\bx{c}} = H^\bullet( B\Gx{T}{c})$. 
\begin{thm} \label{thm: loc thm}
Let $X$ be a smooth quasi-projective $\Gx{T}{c}$-variety and let $Y$ be the set of connected components of the fixed point set $X^{\Gx{T}{c}}$. Suppose that $Y$ is finite. Then the maps 
\[ \Ax{K}{c}\otimes_{\Ax{P}{c}} H^\bullet_{\Gx{T}{c}}(X) \xrightarrow{i^*} \Ax{K}{c}\otimes_{\Ax{P}{c}} H^\bullet_{\Gx{T}{c}}(X^{\Gx{T}{c}}) \xrightarrow{i_*} \Ax{K}{c}\otimes_{\Ax{P}{c}} H^\bullet_{\Gx{T}{c}}(X) \] are isomorphisms and 
\[ u = \sum_{y \in Y} \frac{(i_{y})_* i_y^*(u)}{\mathtt{eu}_{\Gx{T}{c}}(T_{y}X)}\]
for all $u \in H^\bullet_{\Gx{T}{c}}(X)$. Here $i \colon X^{\Gx{T}{c}} \hookrightarrow X$ and $i_y \colon y \hookrightarrow X$ are the natural inclusions. 
\end{thm}

\subsection{The polynomial representation.}

In this subsection we describe the polynomial representation $\Ax{Q}{c}$ of $\Ax{Z}{c}$. The following result is standard. 

\begin{prop}  \label{pro: faithful}
The $\mathcal{Z}_{\mathbf{c}}$-module $\mathcal{Q}_{\mathbf{c}}$ is faithful. 
\end{prop}

\begin{proof}
See, e.g., the proofs of \cite[Lemma 1.8(a)]{VV} and \cite[Proposition 3.1]{VV2}. 
\end{proof}

We will now calculate the action of the generators of the quiver Schur algebra on its polynomial representation. 
As preperation, we first compute the Euler classes of certain normal bundles. 
\begin{defi}
For $\underline{\mathbf{d}} \succ \underline{\mathbf{e}} \rightslice \bx{c}$, we define the following polynomials in $\Lambda_{\ux{d}}$: 
\begin{equation} \label{Ed classes}
\mathtt{E}_{\ux{d}} := \prod_{i,j \in Q_0} \prod_{r=1}^{\ell_{\ux{d}-1}} \prod_{k=\overset{\circ}{\mathbf{d}}_{r-1}(i)+1}^{\overset{\circ}{\mathbf{d}}_{r}(i)} \prod_{l=\overset{\circ}{\mathbf{d}}_r(j)+1}^{\mathbf{c}(j)} (x_l(j) - x_k(i))^{a_{ij}}, \quad  \quad
\mathtt{E}_{\ux{d}}^{\ux{e}} := \frac{\mathtt{E}_{\ux{d}}}{\mathtt{E}_{\ux{e}}}, 
\end{equation}
where $a_{ij}$ is the number of arrows from vertex $i$ to $j$. 
\end{defi}

It is easy to see that $\mathtt{E}_{\ux{d}}$ is $\Gxx{W}{d}$-invariant and $\mathtt{E}_{\ux{d}}^{\ux{e}}$ is a polynomial. Explicit examples of these polynomials can be found in \cite[\S 8]{MS}. 

\begin{lem} \label{lem: eu TQ}
We have 
$\eu_{\Gx{G}{c}}(T_{\Sxx{Q}{d}}\mathfrak{Q}_{\ux{d},(\bx{c})}) = \mathtt{E}_{\ux{d}}$. 
\end{lem}

\begin{proof}
We identify $T_{\Sxx{Q}{d}}\mathfrak{Q}_{\ux{d},(\bx{c})} \cong \Gx{G}{c} \times^{\Gxx{P}{d}} (\Sx{R}{c}/\Sxx{R}{d})$ and $(\Gx{G}{c} \times^{\Gxx{P}{d}} (\Sx{R}{c}/\Sxx{R}{e}))_{\Gx{G}{c}} = (\Sx{R}{c}/\Sxx{R}{e})_{\Gxx{P}{d}}$. The pullback of the latter vector bundle on $B \Gxx{P}{d}$ along the canonical map $B\Gx{T}{c} \twoheadrightarrow B \Gxx{P}{d}$ equals $(\Sx{R}{c}/\Sxx{R}{e})_{\Gx{T}{c}}$. Observe that 
\[  (\Sx{R}{c}/\Sxx{R}{e})_{\Gx{T}{c}}  \cong \bigoplus_{i,j \in Q_0} \bigoplus_{r=1}^{\ell_{\ux{d}-1}} \bigoplus_{k=\overset{\circ}{\mathbf{d}}_{r-1}(i)+1}^{\overset{\circ}{\mathbf{d}}_{r}(i)} \ \bigoplus_{l=\overset{\circ}{\mathbf{d}}_r(j)+1}^{\mathbf{c}(j)}  (\mathfrak{V}_l(j) \otimes \mathfrak{V}_k(i)^*)^{\oplus a_{ij}},
\]
where $\mathfrak{V}_l(j)$ is the line bundle from \eqref{line bundle}. By definition, the $\Gx{T}{c}$-equivariant Euler class of the bundle on the RHS equals $\mathtt{E}_{\ux{d}}$. Since Euler classes commute with pullbacks, 
it follows that $\eu_{\Gx{G}{c}}(T_{\Sxx{Q}{d}}\mathfrak{Q}_{\ux{d},(\bx{c})}) = \eu_{\Gx{T}{c}}((\Sx{R}{c}/\Sxx{R}{e})_{\Gx{T}{c}}) =~\mathtt{E}_{\ux{d}}$, as desired. 
\end{proof}

We also need the following ``shuffle operator'' 
\[ 
\scalebox{1.5}{$\shf$}_{\ux{d}}^{\ux{e}} := \sum_{w \in \dccb{d}{e}} w \in \Gxx{W}{e}. 
\]

\begin{thm} \label{thm:polrep} 
Let $\underline{\mathbf{d}} \succ \underline{\mathbf{e}} \rightslice \bx{c}$. 
\begin{enumerate}[label=\alph*), font=\textnormal,noitemsep,topsep=3pt,leftmargin=1cm]
\item The action of $\bigcurlywedge_{\ux{d}}^{\ux{e}}$ on $\Ax{Q}{c}\cong \Lambda_{\mathbf{c}}$ is given by 
\[ \textstyle \bigcurlywedge_{\ux{d}}^{\ux{e}} \displaystyle 
\colon \Lambda_{\ux{d}} \to \Lambda_{\ux{e}}, \quad 
f \mapsto \scalebox{1.5}{$\shf$}_{\ux{d}}^{\ux{e}} \left( \frac{\mathtt{E}_{\ux{d}}^{\ux{e}}}{\mathtt{S}_{\ux{d}}^{\ux{e}}} f \right), \quad \quad \quad \textstyle \bigcurlywedge_{\ux{d}}^{\ux{e}}|_{\Lambda_{\ux{b}}} = 0 \quad \mbox{if} \quad \ux{b} \neq \ux{d}. 
\]
\item The action of $\bigcurlyvee_{\ux{e}}^{\ux{d}}$ is given by the inclusion
\[ \textstyle \bigcurlyvee_{\ux{e}}^{\ux{d}} \displaystyle 
\colon \Lambda_{\ux{e}} \hookrightarrow \Lambda_{\ux{d}}, \quad 
f \mapsto f, \quad \quad \quad \textstyle \bigcurlyvee_{\ux{e}}^{\ux{d}}|_{\Lambda_{\ux{b}}} = 0 \quad \mbox{if} \quad \ux{b} \neq \ux{e}.
\] 
\item The action of $\Axxy{Z}{d}{d}^e$ on $\lamxx{b}$ is trivial unless $\ux{b} = \ux{d}$. In the latter case, if we identify $\Axxy{Z}{d}{d}^e\cong \lamxx{d}$ as in \eqref{Zepols}, then $\Axxy{Z}{d}{d}^e$ acts on $\lamxx{d}$ by usual multiplication. 
\end{enumerate}
\end{thm}

\begin{proof}

It is obvious that $\bigcurlywedge_{\ux{d}}^{\ux{e}}|_{\Lambda_{\ux{b}}} = 0$ unless $\ux{b} = \ux{d}$. In the latter case, we observe that $\Sxxy{Z}{e}{d}^e\times_{\Sxx{Q}{d}} \Sxx{Q}{d} = \Sxxy{Z}{e}{d}^e \cong \Sxx{Q}{d}$ and that, under this identification, $p_{12}^*(\bigcurlywedge_{\ux{d}}^{\ux{e}}) \cap p_{23}^*f = f$ (with $p_{ij}$ as in \S \ref{subsec:convolution}). Hence $\bigcurlywedge_{\ux{d}}^{\ux{e}} \star f$ equals the pushforward of $f$ along the canonical map $p_{13} \colon \Sxx{Q}{d} \twoheadrightarrow \Sxx{Q}{e}$, which factors as follows 
\begin{equation} \label{iotap}
p_{13} \colon \Sxx{Q}{d} \overset{\iota}{\hookrightarrow} \Sxxy{Q}{d}{e} \overset{q}{\twoheadrightarrow} \Sxx{Q}{e}.  
\end{equation}
We first compute $\iota_*$. 
By \cite[Corollary 2.6.44]{CG}, we have
$\iota_*f = \eu_{\Gx{G}{c}}(T_{\Sxx{Q}{d}}\Sxxy{Q}{d}{e}) f.$ 
The short exact sequence
$0 \to T_{\Sxx{Q}{e}}\mathfrak{Q}_{\ux{e},(\bx{c})} \to T_{\Sxx{Q}{d}}\mathfrak{Q}_{\ux{d},(\bx{c})} \to T_{\Sxx{Q}{d}}\Sxxy{Q}{d}{e} \to 0$
implies that 
\begin{equation*} \label{eu frac}
\eu_{\Gx{G}{c}}(T_{\Sxx{Q}{d}}\Sxxy{Q}{d}{e}) = \eu_{\Gx{G}{c}}(T_{\Sxx{Q}{d}}\mathfrak{Q}_{\ux{d},(\bx{c})})/\eu_{\Gx{G}{c}}(T_{\Sxx{Q}{e}}\mathfrak{Q}_{\ux{e},(\bx{c})}).
\end{equation*}
It now follows from Lemma \ref{lem: eu TQ} that 
\begin{equation} \label{iotaEed}
\iota_*f = \mathtt{E}_{\ux{d}}^{\ux{e}}f. 
\end{equation}

We will next compute $q_*h$, where $h:= \mathtt{E}_{\ux{d}}^{\ux{e}}f$. Since $\Sxxy{Q}{d}{e} = \Sxx{F}{d} \times_{\Sxx{F}{e}} \Sxx{Q}{e}$, calculating the pushforward $q_*$ reduces to calculating the pushforward along $\bar{q} \colon \Sxx{F}{d} \to \Sxx{F}{e}$. 
By Theorems \ref{Tu theorem} and \ref{thm: loc thm}, we have
\[ 
h(\vec{y}) = \sum_{w \in \dcb{d}{c}} \frac{(i_w)_* i_w^*(h(\vec{y}))}{\eu_{\Gx{T}{c}}(T_w\Sxx{F}{d})} = \sum_{w \in \dcb{d}{c}} \frac{w\cdot h(\vec{x})}{w \cdot \mathtt{S}_{\ux{d}}} \zeta_w
= \sum_{u \in \dcb{e}{c}} \frac{1}{u \cdot \mathtt{S}_{\ux{e}}} u  \sum_{v \in \dccb{d}{e}} \frac{v\cdot h(\vec{x})}{v \cdot \mathtt{S}_{\ux{d}}^{\ux{e}}} \zeta_{uv}
\] 
since $v \cdot \mathtt{S}_{\ux{e}} = \mathtt{S}_{\ux{e}}$. 
Let $g(\vec{y}) := \bar{q}_*(h(\vec{y}))$. 
Since $\bar{q}_*\zeta_{uv} = \zeta_u$, we have 
\[ 
g(\vec{y}) = \sum_{u \in \dcb{e}{c}} \frac{1}{u \cdot \mathtt{S}_{\ux{e}}} u \sum_{v \in \dccb{d}{e}} \frac{v\cdot h(\vec{x})}{v \cdot \mathtt{S}_{\ux{d}}^{\ux{e}}} \zeta_u. 
\] 
On the other hand, Theorem \ref{Tu theorem} implies that 
\[ 
g(\vec{y}) = \sum_{u \in \dcb{e}{c}} \frac{(i_u)_* i_u^*(g(\vec{y}))}{\eu_{\Gx{T}{c}}(T_u\Sxx{F}{e})} = \sum_{u \in \dcb{e}{c}} \frac{u\cdot g(\vec{x})}{u \cdot \mathtt{S}_{\ux{e}}} \zeta_u.  
\] 
Hence
\begin{equation} \label{gxvec}
g(\vec{x}) = \sum_{v \in \dccb{d}{e}} v \cdot \frac{h(\vec{x})}{\mathtt{S}_{\ux{d}}^{\ux{e}}}. 
\end{equation}
Combining \eqref{iotaEed} with \eqref{gxvec} yields the first part of the theorem. 

An argument analogous to the one at the beginning of the proof  shows that $\bigcurlyvee_{\ux{e}}^{\ux{d}}|_{\Lambda_{\ux{b}}} = 0$ unless $\ux{b} = \ux{e}$, and that convolving $\bigcurlyvee_{\ux{e}}^{\ux{d}}$ with a function $f \in \Lambda_{\ux{e}}$ is the same as taking the pullback with respect to \eqref{iotap}. A calculation using the localization theorem, similar to the one above, shows that $q^*$ is given by the inclusion of the invariants $\lamxx{e} \hookrightarrow \lamxx{d}$, while the second pullback $\iota^*$ is just an isomorphism. This yields the second part of the theorem. The third part is standard - see, e.g., \cite[Example 2.7.10(i)]{CG}. 
\end{proof} 

We will now relate the action of the merges to Demazure operators. 

\begin{defi} 
Given $s_j(i) \in \Gx{W}{c}$, let $\Delta_j(i) = \frac{1-s_j(i)}{x_j(i) - x_{j+1}(i)}$ be the corresponding Demazure operator. 
Given $\underline{\mathbf{d}} \succ \underline{\mathbf{e}} \rightslice \bx{c}$, let $w^{\ux{e}}_{\ux{d}} $ be the longest element in $\dccb{d}{e}$. Choose a reduced expression $w^{\ux{e}}_{\ux{d}} = s_{j_1}(i_1) \cdot \hdots \cdot s_{j_k}(i_k)$ and define $ \Delta_{\ux{d}}^{\ux{e}} = \Delta_{j_1}(i_1) \circ \hdots \circ \Delta_{j_k}(i_k)$. It is well known that $ \Delta_{\ux{d}}^{\ux{e}}$ does not depend on the choice of reduced expression for $w^{\ux{e}}_{\ux{d}}$. Let $r_{\ux{d}}^{\ux{e}} = |R_{\ux{e}}^+ - R_{\ux{d}}^+|$ and $r_{\ux{d}} = |R_{\bx{c}}^+ - R_{\ux{d}}^+|$. 
\end{defi}

\begin{prop} \label{lem: merge = Demazure}
We have an equality of operators
\[ \scalebox{1}{$\shf$}_{\ux{d}}^{\ux{e}} (\mathtt{S}_{\ux{d}}^{\ux{e}})^{-1} = (-1)^{r_{\ux{d}}^{\ux{e}}} \Delta_{\ux{d}}^{\ux{e}}.\] 
\end{prop}

\begin{proof}
See \cite[Proposition 8.13]{MS}. 
\end{proof}

\subsection{Application: geometric realization of the modified quiver Schur algebra.}

We now deduce some consequences from Theorem \ref{thm:polrep} in the special case when $Q$ is the cyclic quiver with at least two  vertices or the infinite (in both directions) linear quiver $A_\infty$, connecting our quiver Schur algebra $\Ax{Z}{c}$ to exisiting constructions. 

Miemietz and Stroppel introduced in \cite[Definition 8.4]{MS} a \emph{modified quiver Schur algebra}. Let us denote it by $\Ax{Z}{c}^{MS}$ (in \cite{MS} the notation $\mathbf{C}_{\mathbf{i}}$ is used). 
It is defined, purely algebraically, as the subalgebra of $\End_{\C}(\Lambda_{\mathbf{c}})$ generated by certain linear operators, called idempotents, polynomials, splits and merges. These operators are defined by explicit formulas. 
We will refer to them as ``algebraic'', in order to distinguish them from the fundamental classes in Definition \ref{defi:mergesandsplits}. 

We must first deal with a minor technical issue. The algebraic merges are defined using ``reversed Euler classes'', denoted in \cite{MS} by $\mathtt{E}_{\mathbf{u}_J}$, and  ``symmetrisers'', denoted by $\mathtt{S}_{\mathbf{u}_J}$ (see \cite[(8.1-2)]{MS}). Both of them are given by certain product formulas. We define sign-corrected algebraic merges to be the operators obtained by multiplying $\mathtt{E}_{\mathbf{u}_J}$ and $\mathtt{S}_{\mathbf{u}_J}$ by  $-1$ if number of factors in the corresponding product is odd. 

The main result of \cite{MS} says that the geometrically defined Stroppel-Webster quiver Schur algebra $\Ax{Z}{c}^{SW}$ is, after completion, isomorphic to the affine $q$-Schur algebra \cite{Gre}, which naturally appears in the representation theory of $p$-adic general linear groups. The proof of this result relies on the fact that both of these algebras are isomorphic to the modified quiver Schur algebra $\Ax{Z}{c}^{MS}$. The following theorem shows that $\Ax{Z}{c}^{MS}$ also admits a geometric realization as a convolution algebra, and that this realization is afforded by our quiver Schur algebra $\Ax{Z}{c}$.

\begin{thm} \label{cor:QSvMQS} 
There is an algebra isomorphism $\Ax{Z}{c} \cong \Ax{Z}{c}^{MS}$. Explicitly, this isomorphism sends $\spl{e}{d}$, $\mer{d}{e}$, $\mathsf{e}_{\ux{d}}$ and a polynomial in $\Axxy{Z}{d}{d}^e$ $($where $\underline{\mathbf{d}} \succ \underline{\mathbf{e}} \rightslice \bx{c})$ to the corresponding algebraic split, sign-corrected merge, idempotent and polynomial, respectively. 
\end{thm} 

\begin{proof}
Corollary \ref{cor:el spl mer} says that $\spl{e}{d}$, $\mer{d}{e}$ and $\Axxy{Z}{d}{d}^e$ generate $\Ax{Z}{c}$. In fact, the corollary makes the stronger statement that it is enough to take polynomials together with \emph{elementary} splits and merges to get a generating set. However, since the definition of $\Ax{Z}{c}^{MS}$ involves arbitrary algebraic splits and merges, we only need the weaker form of Corollary \ref{cor:el spl mer} here. By Proposition \ref{pro: faithful}, the polynomial representation $\Ax{Q}{c} \cong \lamx{c}$ of $\Ax{Z}{c}$ is faithful. Hence $\Ax{Z}{c}$ is isomorphic to the subalgebra of $\End_{\C}(\Lambda_{\mathbf{c}})$ generated by the linear operators representing splits, merges and polynomials. To complete the proof, one only has to compare the description of these operators from Theorem \ref{thm:polrep} with the definition of their algebraic counterparts in \cite[Definition 8.4]{MS}.
\end{proof}

Corollary \ref{cor:el spl mer} and Theorem \ref{cor:QSvMQS} directly imply the following statement about the generators of the modified quiver Schur algebra $\Ax{Z}{c}^{MS}$, which is not obvious from its algebraic definition. 

\begin{thm}
The modified quiver Schur algebra $\Ax{Z}{c}^{MS}$ is generated by algebraic polynomials and \emph{elementary} algebraic splits and merges. 
\end{thm}

Note that combining Theorems \ref{thm:basis} and \ref{cor:QSvMQS} also gives us a basis of $\Ax{Z}{c}^{MS}$. 
 
Moreover, we can relate $\Ax{Z}{c}$ to the Stroppel-Webster quiver Schur algebra. 

\begin{thm} \label{thm: our vs SW}
Let $Q$ be an arbitrary quiver. 
Then our quiver Schur algebra $\Ax{Z}{c}$ is isomorphic to the Stroppel-Webster quiver Schur algebra $\Ax{Z}{c}^{SW} = H_\bullet^{\Gx{G}{c}}(\Sx{Z}{c}^s)$. 
\end{thm} 

\begin{proof} 
The definition of the modified quiver Schur algebra $\Ax{Z}{c}^{MS}$, together with Theorem \ref{cor:QSvMQS}, generalize straightforwardly to arbitrary quivers. 
We also observe that the proofs of \cite[Propositions 9.4, 9.6]{MS} do not depend on the choice of cyclic quiver, and hence generalize to arbitrary quivers, yielding the desired isomorphism. 
\end{proof}

\subsection{Examples: the $A_1$ and Jordan quivers.} \label{subsec: A1 and Jordan}

In this subsection we discuss the examples of the $A_1$ quiver (i.e. one vertex with no arrows) and the Jordan quiver. It is well known (see, e.g., \cite{KL1, Rou, Sau3}) that the corresponding KLR algebras are isomorphic to the affine Nil-Hecke algebra and the degenerate affine Hecke algebra, respectively. 
While it is quite hard to give a presentation by generators and relations for the entire quiver Schur algebra, even for the $A_1$ and the Jordan quiver, we are able to give a complete list of relations for the following subalgebra. 

\begin{defi}
Let $\Ax{Z}{c}'$ be the subalgebra of $\Ax{Z}{c}$ generated by all merges and splits. We call it the \emph{reduced quiver Schur algebra}. 
\end{defi} 

We first consider the case where $Q$ is the $A_1$ quiver. 
Let $\underline{\mathbf{d}} \succ \underline{\mathbf{e}} \rightslice \mathbf{c}$. 
Note that since the quiver has only one vertex, $\bx{c}$ is just a positive integer and $\ux{d}$ and $\ux{e}$ are compositions of this integer. 
Since the quiver has no arrows, $\mathtt{E}_{\ux{d}}^{\ux{e}} = 1$ and $\bigcurlywedge_{\ux{d}}^{\ux{e}} = \scalebox{1}{$\shf$}_{\ux{d}}^{\ux{e}} (\mathtt{S}_{\ux{d}}^{\ux{e}})^{-1}$. Therefore, by Lemma \ref{lem: merge = Demazure}, we have $\bigcurlywedge_{\ux{d}}^{\ux{e}} = \Delta_{\ux{d}}^{\ux{e}}$, i.e., merges coincide with Demazure operators. 
Let us look at the special case when $\ux{d} = (m,n)$ and $\ux{e} = (\bx{c}) = (m+n)$. Then 
\begin{equation} 
\mathtt{S}_{\ux{d}}^{\ux{e}} = \mathtt{S}_{\ux{d}} = \prod_{k=1}^{m} \prod_{l=m+1}^{m+n} (x_l - x_k), \quad \quad \textstyle \bigcurlywedge_{\ux{d}}^{\ux{e}} \displaystyle = \sum_{w \in \mathsf{D}_{\ux{d}}^{\bx{c}}}  \prod_{k=1}^{m} \prod_{l=m+1}^{m+n}  w \cdot (x_l - x_k)^{-1}. 
\end{equation} 

We will now give a complete list of defining relations in the reduced quiver Schur algebra. We call

\begin{equation} \tag{R3} \label{R3 relation}
\tikz[thick,xscale=2,yscale=1.2, baseline=-0.2cm]{ 
\small 
\draw (0,0) to [out=90,in=-90](.3,.5)
(.6,0) to [out=90,in=-90] (.3,.5)
(.3,.5) -- (.3,.8) node[above] {$\bx{d}_{k} + \bx{d}_{k+1}$};
\begin{scope}[yscale=-1, yshift=0.4cm]
\draw (0,0) node[above] {$\bx{d}_{k}$} to [out=90,in=-90](.3,.5) 
(.6,0) node[above] {$\bx{d}_{k+1}$} to [out=90,in=-90] (.3,.5)
(.3,.5) -- (.3,.8) node[below] {$\bx{d}_{k} + \bx{d}_{k+1}$};
\end{scope} \normalsize
\node at (1.3,-0.15) {$= \quad 0$};
}\end{equation}
the \emph{hole removal} relation, and 
\begin{equation}  \tag{R4} \label{R4 relation}
\tikz[thick,xscale=2.4,yscale=1,baseline=2.25cm]{ \small
\draw (-0.6,0) -- (-0.6,.8) node[below,at start]{$\bx{d}_{k}+\bx{d}_{k+1}$} node[above]{$\bx{d}_{k}+\bx{d}_{k+1}$};
\draw (0.3,0.3)  to [out=90,in=-90] (0,.8) node[above] {$\bx{d}_{k+2}$}
(.3,0.3)  to [out=90,in=-90] (.6,.8) node[above] {$\bx{d}_{k+3}$}
(.3,0) node[below]{$\bx{d}_{k+2}+\bx{d}_{k+3}$}-- (.3,.3); 
\begin{scope}[yshift=1.3cm, xshift=-0.6cm]
\draw (0,0) to [out=90,in=-90](.3,.5)
(.6,0)  to [out=90,in=-90] (.3,.5)
(.3,.5) -- (.3,.8) node[above] {$\bx{d}_k+\bx{d}_{k+1}+\bx{d}_{k+2}$};
\draw (1.2,0) -- (1.2,.8) node[above]{$\bx{d}_{k+3}$};
\end{scope}
\begin{scope}[yshift=2.6cm, xshift=-0.6cm]
\draw (0.3,0.3)  to [out=90,in=-90] (0,.8) node[above] {$\bx{d}_k + \bx{d}_{k+2}$}
(.3,0.3)  to [out=90,in=-90] (.6,.8) node[above] {$\bx{d}_{k+1}$}
(.3,0) -- (.3,.3); 
\draw (1.2,0) -- (1.2,.8) node[above]{$\bx{d}_{k+3}$};
\end{scope} 
\begin{scope}[yshift=3.9cm, xshift=0cm] 
\draw (-0.6,0) -- (-0.6,.8) node[above]{$\bx{d}_{k} + \bx{d}_{k+2}$};
\draw (0,0) to [out=90,in=-90](.3,.5)
(.6,0)  to [out=90,in=-90] (.3,.5)
(.3,.5) -- (.3,.8) node[above] {$\bx{d}_{k+1} + \bx{d}_{k+3}$};
\end{scope}
\node at (1.4,2.3) {\normalsize $=$};
\begin{scope}[xscale=-1, xshift = -2.8cm]
\draw (-0.6,0) -- (-0.6,.8) node[below,at start]{$\bx{d}_{k+2}+\bx{d}_{k+3}$} node[above]{$\bx{d}_{k+2}+\bx{d}_{k+3}$};
\draw (0.3,0.3)  to [out=90,in=-90] (0,.8) node[above] {$\bx{d}_{k+1}$}
(.3,0.3)  to [out=90,in=-90] (.6,.8) node[above] {$\bx{d}_{k}$}
(.3,0) node[below]{$\bx{d}_{k}+\bx{d}_{k+1}$}-- (.3,.3); 
\begin{scope}[yshift=1.3cm, xshift=-0.6cm]
\draw (0,0) to [out=90,in=-90](.3,.5)
(.6,0)  to [out=90,in=-90] (.3,.5)
(.3,.5) -- (.3,.8) node[above] {$\bx{d}_{k+1}+\bx{d}_{k+2}+\bx{d}_{k+3}$};
\draw (1.2,0) -- (1.2,.8) node[above]{$\bx{d}_{k}$};
\end{scope}
\begin{scope}[yshift=2.6cm, xshift=-0.6cm]
\draw (0.3,0.3)  to [out=90,in=-90] (0,.8) node[above] {$\bx{d}_{k+1} + \bx{d}_{k+3}$}
(.3,0.3)  to [out=90,in=-90] (.6,.8) node[above] {$\bx{d}_{k+2}$}
(.3,0) -- (.3,.3); 
\draw (1.2,0) -- (1.2,.8) node[above]{$\bx{d}_{k}$};
\end{scope} 
\begin{scope}[yshift=3.9cm, xshift=0cm] 
\draw (-0.6,0) -- (-0.6,.8) node[above]{$\bx{d}_{k+1} + \bx{d}_{k+3}$};
\draw (0,0) to [out=90,in=-90](.3,.5)
(.6,0)  to [out=90,in=-90] (.3,.5)
(.3,.5) -- (.3,.8) node[above] {$\bx{d}_{k} + \bx{d}_{k+2}$};
\end{scope}
\end{scope} 
} \end{equation}
the \emph{ladder} relation. 
\begin{thm} \label{thm: rel A1} 
The reduced quiver Schur algebra $\Ax{Z}{c}'$ associated to the $A_1$ quiver is generated by elementary merges and splits, subject to the relations \eqref{R1 relation}, \eqref{R2 relation}, \eqref{R3 relation} and \eqref{R4 relation}. 
\end{thm}

A detailed proof of Theorem \ref{thm: rel A1} can be found in \cite{Sei}. Below we will sketch the main ideas of the proof. We first need to recall some material about the green web category $\infty$-$\mathbf{Web}_g$ from \cite{TVW}. 

\begin{defi} \label{defi: web category}
We define a certain full subcategory $\mathscr{C}_{\bx{c}}$ of the green web category $\infty$-$\mathbf{Web}_g$. 
\begin{enumerate}[label=\alph*), font=\textnormal,noitemsep,topsep=3pt,leftmargin=1cm]
\item One first defines the free web category $\mathscr{C}^f_{\bx{c}}$. Its objects are compositions of $\bx{c}$ and its morphism spaces $\Hom_{\mathscr{C}^f_{\bx{c}}}(\ux{d},\ux{e})$ are generated by  elementary merge and split diagrams (which we denote by $\textcolor{teal}{\bigcurlywedge}_{\ux{d}}^{\ux{e}}$ and $\textcolor{teal}{\bigcurlyvee}^{\ux{d}}_{\ux{e}}$, respectively) 
via vertical composition. 
\item We define a filtration on the morphism spaces by setting $\deg \textcolor{teal}{\bigcurlywedge}_{\ux{d}}^{\ux{e}} = \deg \textcolor{teal}{\bigcurlyvee}^{\ux{d}}_{\ux{e}} = \bx{d}_k + \bx{d}_{k+1}$ if $\underline{\mathbf{d}} \succ \underline{\mathbf{e}} \rightslice \mathbf{c}$ and $\ux{e} = \wedge_k(\ux{d})$ for some $1 \leq k \leq \ld$. 
\item The category $\mathscr{C}_{\bx{c}}$ is the quotient of $\mathscr{C}^f_{\bx{c}}$ obtained by imposing certain relations on morphisms, called the associativity, coassociativity, digon removal and square switch relations \cite[(2-6)-(2-8)]{TVW}. 
\end{enumerate}
We remark that $\infty$-$\mathbf{Web}_g$ is defined in \cite{TVW} as a $\C(q)$-linear category. For our purposes, however, it is enough to work with the $\C$-linear category obtained by setting $q=1$. 
\end{defi}

Consider the filtered algebra 
\[ \Morph(\mathscr{C}_{\bx{c}}^f) = \bigoplus_{\ux{d},\ux{e} \in \mathbf{Com}_{\mathbf{c}}} \Hom_{\mathscr{C}_{\bx{c}}^f}(\ux{d},\ux{e}).\]
Let $I_{\mathscr{C}_{\bx{c}}}$ be the kernel of the canonical map $\Morph(\mathscr{C}^f_{\bx{c}}) \twoheadrightarrow \Morph(\mathscr{C}_{\bx{c}})$. 
We endow $I_{\mathscr{C}_{\bx{c}}}$ with the subspace filtration and $\Morph(\mathscr{C}_{\bx{c}})$ with the quotient filtration. 
Then the sequence $0 \to I_{\mathscr{C}_{\bx{c}}} \to \Morph(\mathscr{C}_{\bx{c}}^f) \to \Morph(\mathscr{C}_{\bx{c}}) \to 0$ is strict exact and so, after taking the associated graded, we obtain the short exact sequence 
\begin{equation} \label{gr ses} 0 \to \gr I_{\mathscr{C}_{\bx{c}}} \to \gr \Morph(\mathscr{C}_{\bx{c}}^f) \to \gr \Morph(\mathscr{C}_{\bx{c}}) \to 0. \end{equation} 
Note that the rule in Definition \ref{defi: web category}.b) also defines a grading on $\Morph(\mathscr{C}_{\bx{c}}^f)$, so $\Morph(\mathscr{C}_{\bx{c}}^f)$ is isomorphic as an algebra to its associated graded $\gr \Morph(\mathscr{C}_{\bx{c}}^f)$. 

\begin{proof}[Proof of Theorem \ref{thm: rel A1}]
Firstly, we need to check that the relations  \eqref{R1 relation}-\eqref{R4 relation} hold in $\Ax{Z}{c}'$. By Proposition \ref{pro:transitivity}, the relations  \eqref{R1 relation} and  \eqref{R2 relation} hold in any quiver Schur algebra (associated to any quiver). Relations  \eqref{R3 relation} and  \eqref{R4 relation} follow easily from the properties of Demazure operators. 

Secondly, we need to check that the relations \eqref{R1 relation}-\eqref{R4 relation} generate all the relations in $\Ax{Z}{c}'$. Let  $\widetilde{\mathcal{Z}}_{\bx{c}}'$ be the quotient of the free algebra ${}^f\widetilde{\mathcal{Z}}_{\bx{c}}'$, generated by elementary merges and splits, by the ideal $I_{\bx{c}}$ generated by the relations \eqref{R1 relation}-\eqref{R4 relation} so that we have a short exact sequence 
\begin{equation} \label{Z rest ses} 0 \to I_{\bx{c}} \to {}^f\widetilde{\mathcal{Z}}_{\bx{c}} \to \widetilde{\mathcal{Z}}_{\bx{c}}' \to 0. \end{equation}
It follows from the definitions that $\gr \Morph(\mathscr{C}_{\bx{c}}^f) \cong \Morph(\mathscr{C}_{\bx{c}}^f) \cong {}^f\widetilde{\mathcal{Z}}_{\bx{c}}'$. One also easily sees that after taking the associated graded the relations (2-6)-(2-8) from \cite{TVW} become the relations \eqref{R1 relation}-\eqref{R4 relation}. Hence $I_{\bx{c}} \cong  \gr I_{\mathscr{C}_{\bx{c}}}$. Comparing \eqref{gr ses} with \eqref{Z rest ses} now implies that $\widetilde{\mathcal{Z}}_{\bx{c}}' \cong \gr \Morph(\mathscr{C}_{\bx{c}})$. 

Next, \cite[Theorem 3.20]{TVW} implies that there is a vector space isomorphism $\Hom_{\mathscr{C}_{\bx{c}}}(\ux{d},\ux{e}) \cong \Hom_{\mathfrak{gl}_{\infty}}(\bigwedge^{\ux{d}} \C^\infty , \bigwedge^{\ux{e}} \C^\infty )$, where $\bigwedge^{\ux{d}} \C^\infty = \bigwedge^{\bx{d}_1} \C^\infty  \otimes \hdots \otimes \bigwedge^{\bx{d}_{\ell_{\ux{d}}}} \C^\infty$. Hence 
\[ \textstyle \dim \widetilde{\mathcal{Z}}_{\ux{d},\ux{e}}' = \dim \Hom_{\mathfrak{gl}_{\infty}}(\bigwedge^{\ux{d}} \C^\infty , \bigwedge^{\ux{e}} \C^\infty ) = |\dc{d}{e}{c}| = \dim \Axxy{Z}{d}{e},\]
where the second equality can be deduced from Schur-Weyl duality and the last equality follows from Theorem \ref{thm:basis}. We conclude that the natural map $\widetilde{\mathcal{Z}}_{\bx{c}}' \twoheadrightarrow \Ax{Z}{c}$ is an isomorphism. 
\end{proof} 

Let us record the following corollary of the proof of Theorem \ref{thm: rel A1}. 

\begin{cor} \label{cor: green web A1} 
If $Q$ is the $A_1$ quiver, then there is an algebra isomorphism $\Ax{Z}{c}' \cong \gr \Morph(\mathscr{C}_{\bx{c}})$. 
\end{cor}

Next suppose that $Q$ is the Jordan quiver. 
We can interpret merges as symmetrization operators between rings of invariants. Indeed, $\mathtt{S}_{\ux{d}}^{\ux{e}} = \mathtt{E}_{\ux{d}}^{\ux{e}}$ and so 
\begin{equation} \label{merge = sym opp} \textstyle \bigcurlywedge_{\ux{d}}^{\ux{e}} = \scalebox{1}{$\shf$}_{\ux{d}}^{\ux{e}}. 
\end{equation} 
We will now describe the relations in the reduced quiver Schur algebra. We use the following modification of (R3) (with $\ux{e} = \wedge^k(\ux{d})$):  
\begin{equation} \tag{R3'} \label{R3' relation}
\tikz[thick,xscale=2,yscale=1.2, baseline=-0.2cm]{ 
\small 
\draw (0,0) to [out=90,in=-90](.3,.5)
(.6,0) to [out=90,in=-90] (.3,.5)
(.3,.5) -- (.3,.8) node[above] {$\bx{d}_{k} + \bx{d}_{k+1}$};
\begin{scope}[yscale=-1, yshift=0.4cm]
\draw (0,0) node[above] {$\bx{d}_{k}$} to [out=90,in=-90](.3,.5) 
(.6,0) node[above] {$\bx{d}_{k+1}$} to [out=90,in=-90] (.3,.5)
(.3,.5) -- (.3,.8) node[below] {$\bx{d}_{k} + \bx{d}_{k+1}$};
\end{scope} \normalsize
\node at (1.5,-0.15) {$= \quad \left| \mathsf{D}_{\ux{d}}^{\ux{e}}\right| \cdot 1.$};
}\end{equation}

\begin{thm} \label{thm: Jordan q rels}  
The following hold: 
\begin{enumerate}[label=\alph*), font=\textnormal,noitemsep,topsep=3pt,leftmargin=1cm]
\item
The  reduced quiver Schur algebra algebra $\Ax{Z}{c}'$ associated to the Jordan quiver is generated by elementary merges and splits, subject to the relations  \eqref{R1 relation}, \eqref{R2 relation}, \eqref{R3' relation} and \eqref{R4 relation}. 
\item The algebra $\Ax{Z}{c}'$ is isomorphic to the convolution algebra $\bigoplus_{(\mathsf{P},\mathsf{P}')} \C[\mathsf{P} \backslash \Gx{G}{c} / \mathsf{P}']$ of complex valued functions on double cosets, where $(\mathsf{P},\mathsf{P}')$ runs over all pairs of standard parabolic subgroups of $\Gx{G}{c}$.
\end{enumerate} 
\end{thm}

\begin{proof}
The fact that the relations \eqref{R3' relation} and \eqref{R4 relation} hold in $\Ax{Z}{c}'$ follows easily from the properties of symmetrization operators. 
One can define a filtration on $\widetilde{\mathcal{Z}}_{\bx{c}}'$ analogous to the filtration on $\Morph(\mathscr{C}_{\bx{c}})$. It is clear that $\gr \widetilde{\mathcal{Z}}_{\bx{c}}' \cong \gr \Morph(\mathscr{C}_{\bx{c}})$. Hence one can use the same argument as in the proof of Theorem \ref{thm: rel A1} to show that  \eqref{R1 relation}, \eqref{R2 relation}, \eqref{R3' relation} and \eqref{R4 relation} generate all the relations. This proves the first statement of the theorem. The second statement follows from the description of $\Ax{Z}{c}'$ as the algebra of symmetrization operators in \eqref{merge = sym opp}. 
\end{proof}

\section{Mixed quiver Schur algebras}

In this section we define and study a generalization of quiver Schur algebras, depending on a quiver together with a contravariant involution and a duality structure. We call these new algebras \emph{mixed quiver Schur algebras}. From a geometric point of view, our generalization arises by replacing the stack of representations of a quiver with the stack of its supermixed representations in the sense of Zubkov \cite{Zub}. 

\subsection{Involutions and duality structures} 

We begin by recalling the notion of a contravariant involution and a duality structure. These ideas, in the context of quiver representations, were first studied in \cite{DW, Zub}. We use the formulation from \cite{You}. 

\begin{defi}  \label{defi: duality str}
A (contravariant) \emph{involution} of a quiver $Q$ is a pair of involutions $\theta \colon Q_0 \to Q_0$ and $\theta \colon Q_1 \to Q_1$ such that: 
\begin{enumerate}[label=\alph*), font=\textnormal,noitemsep,topsep=2pt] 
\item $s(\theta(a)) = \theta(t(a))$ and $t(\theta(a)) = \theta(s(a))$ for all $a \in Q_1$, 
\item if $t(a) = \theta(s(a))$ then $a = \theta(a)$. 
\end{enumerate} 
A \emph{duality structure} on $(Q, \theta)$ is a pair of functions $\sigma \colon Q_0 \to \{\pm 1\}$ and $\varsigma \colon Q_1 \to \{ \pm 1\}$ such that $\sigma(\theta(i)) = \sigma(i)$ for all $i \in Q_0$ and $\varsigma(a) \cdot \varsigma(\theta(a)) = \sigma(s(a)) \cdot \sigma(t(a))$ for all $a \in Q_1$. 
\end{defi} 

\begin{exa}
Let $n \geq 1$ and suppose that $Q$ is the $A_{n}$ quiver 
\[
\begin{tikzcd}
\underset{i_1}{\bullet} \arrow[r,"a_1"] &\underset{i_2}{\bullet} \arrow[r,"a_2"] & \hdots \arrow[r,"a_{n-1}"] &\underset{i_n}{\bullet}
\end{tikzcd}
\]
There is a unique involution $\theta$ on $Q$. We have $\theta(i_k) = i_{n-k+1}$ for $1 \leq k \leq n$ and $\theta(a_l) = a_{n-l}$ for $1 \leq l \leq n-1$. If $n$ is even then $Q_0^\theta = \varnothing$ and $Q_1^\theta = \{a_{n/2}\}$. If $n$ is odd then $Q_1^\theta = \varnothing$ and $Q_0^\theta = \{i_{(n+1)/2}\}$. There are two inequivalent duality structures: $\sigma = 1$ and $\varsigma = -1$ or $\sigma = -1$ and $\varsigma = -1$. 
\end{exa}

\begin{exa}
Suppose that $Q$ is the quiver with one vertex and $m \geq 0$ loops. There is a unique involution on $Q$. It fixes the vertex and fixes all the loops as well. A duality structure is given by a choice of sign $\sigma$ and a choice of sign $\varsigma(a)$ for each arrow $a$. Hence there are $2^{m+1}$ possible duality structures. 
\end{exa}

For the rest of this section let us fix a quiver $Q$ together with an involution $\theta$ and a duality structure $(\sigma, \varsigma)$. We will now introduce some combinatorics necessary to describe isotropic flag varieties. 
Let us fix partitions  
\[ Q_0 = Q_0^{-} \sqcup Q_0^{\theta} \sqcup Q_0^+, \quad Q_1 = Q_1^{-} \sqcup Q_1^{\theta} \sqcup Q_1^+\]
such that $\theta(Q_0^+) = Q_0^{-}$ and $\theta(Q_1^+) = Q_1^{-}$. 
The involution $\theta$ induces an involution $\theta \colon \Gamma \to \Gamma$ on the monoid of dimension vectors. Let $\Gamma^\theta$ be the submonoid of $\theta$-fixed points. We consider $\Gamma^\theta$ as a $\Gamma$-module via the monoid homomorphism 
\[ \mathrm{D} \colon \Gamma \to \Gamma^\theta, \quad \mathbf{c} \mapsto \mathbf{c} + \theta(\mathbf{c}).\] 

\begin{defi}
Let $\mathbf{c} \in \Gamma^\theta$. 
We call a sequence $\underline{\mathbf{d}} = (\mathbf{d}_1, \hdots, \mathbf{d}_{\ld},\mathbf{d}_\infty) \in \Gamma^{\ld}_+ \times \Gamma^\theta$ (where $\ld$ may equal zero) an \emph{isotropic vector composition} of $\mathbf{c}$, 
denoted $\ux{d}$ $\Rsl$ $\mathbf{c}$, 
if $\langle \underline{\mathbf{d}}\rangle_\theta := \mathbf{d}_\infty + \sum_j \mathrm{D}(\mathbf{d}_j) = \mathbf{c}$. We call $\ld$ the \emph{length} of $\ux{d}$. 
Let ${}^\theta\mathbf{Com}_{\mathbf{c}}$ denote the set of all isotropic vector compositions of $\mathbf{c}$, and let ${}^\theta\mathbf{Com}_{\mathbf{c}}^m$ denote the subset of compositions of length $m$. Consider $\Z_2 \wr \mathsf{Sym}_m$ as the group of signed permutations of the set $\{\pm1, \hdots, \pm m\}$ with $s_m$ changing the sign of $m$. 
We endow ${}^\theta\mathbf{Com}_{\mathbf{c}}^m$ with a right $\Z_2 \wr \mathsf{Sym}_m$-action so that $\mathsf{Sym}_m$ acts by permuting the first $m$ dimension vectors and $s_m$ acts by changing $\bx{d}_{\ld}$ to $\theta(\bx{d}_{\ld})$. 
Set 
\[ \mathrm{D}(\underline{\mathbf{d}}) := (\mathbf{d}_1, \hdots, \mathbf{d}_{\ld},\mathbf{d}_\infty,\theta(\mathbf{d}_{\ld}), \hdots, \theta(\mathbf{d}_1)), \quad \ux{d}^f =  (\mathbf{d}_1, \hdots, \mathbf{d}_{\ld}).\] 
Given $\beta \in \Com(\ell_{\underline{\mathbf{d}}}+1)$, 
let
\[ \wedge_\beta^\theta(\underline{\mathbf{d}}) := ( \langle \vee_{\beta}^{1}(\underline{\mathbf{d}})\rangle, \hdots, \langle \vee_{\beta}^{\ell_\beta-1}(\underline{\mathbf{d}})\rangle, \langle \vee_{\beta}^{\ell_{\beta}}(\underline{\mathbf{d}}) \rangle_\theta).\] 
In particular, if $\beta = (1^{k-1},2,1^{\ell_{\underline{\mathbf{d}}}-k})$ for some $1 \leq k \leq \ell_{\underline{\mathbf{d}}}$, then we abbreviate $\wedge_k^\theta(\underline{\mathbf{d}}) := \wedge_\beta^\theta(\underline{\mathbf{d}})$. 
\end{defi}

\begin{exa}
Consider the $A_3$ quiver together with its unique involution. Let $\bx{c} = 4i_1 + 3i_2 + 4i_3 \in~\Gamma^\theta$ and $\ux{d} = (i_1+i_2,i_3,2i_1 +i_2 +2i_3)$ $\Rsl$ $\bx{c}$. Then $\wedge_1^\theta = (i_1+i_2+i_3, 2i_1 +i_2 +2i_3)$ and $\wedge_2^\theta = (i_1+i_2, 3i_1+i_2+3i_3)$. 
\end{exa}

In analogy to Definition \ref{defi: wedge comp}, 
we define a partial order on ${}^\theta\mathbf{Com}_{\mathbf{c}}$ by setting
\[ \underline{\mathbf{d}} \succcurlyeq \underline{\mathbf{e}} \iff \underline{\mathbf{e}} = \wedge_\beta^\theta(\underline{\mathbf{d}}) \]
for some $\beta \in \Com(\ell_{\underline{\mathbf{d}}}+1)$. 
If $\mathbf{e}_\infty = \mathbf{d}_\infty$, then we write $\ux{d} \succcurlyeq_f \ux{e}$. If $\ux{d} = \ux{d}' \cup \ux{d}''$ and $\langle\ux{d}''\rangle_\theta = \mathbf{e}_\infty$, we write $\ux{d} \succcurlyeq_\infty \ux{e}$.

\subsection{Isotropic flag varieties} 

In this subsection we introduce the notation for isotropic flag varieties, isotropic Steinberg varieties and related objects. 

\begin{defi}
Let $\mathbf{c} \in \Gamma^\theta$. If $i \in Q_0^\theta$ and $\sigma(i) = -1$, we assume that $\mathbf{c}(i)$ is even. Fix a $Q_0$-graded $\C$-vector space $\mathbf{V}_{\mathbf{c}} = \bigoplus_{i \in Q_0} \mathbf{V}_{\mathbf{c}}(i)$ with $\dim \mathbf{V}_{\mathbf{c}}(i) = \mathbf{c}(i)$ and a nondegenerate bilinear form $\langle \cdot , \cdot \rangle \colon \mathbf{V}_{\mathbf{c}} \times \mathbf{V}_{\mathbf{c}} \to \C$ such that: 
\begin{enumerate}[label=\alph*), font=\textnormal,itemsep = 3pt,topsep=3pt] 
\item $\mathbf{V}_{\mathbf{c}}(i)$ and $\mathbf{V}_{\mathbf{c}}(j)$ are orthogonal unless $i = \theta(j)$, 
\item the restriction of $\langle \cdot , \cdot \rangle$ to $\mathbf{V}_{\mathbf{c}}(i) + \mathbf{V}_{\mathbf{c}}(\theta(i))$ satisfies $\langle u, v \rangle = \sigma(i) \langle v, u \rangle$, 
\end{enumerate} 
for $i,j \in Q_0$ and $u,v \in \mathbf{V}_{\mathbf{c}}(i) + \mathbf{V}_{\mathbf{c}}(\theta(i))$. 
Set
\[ {}^\theta\mathfrak{R}_{\mathbf{c}} := \{ \rho \in \mathfrak{R}_{\mathbf{c}} \mid \langle \rho_a(u),v\rangle = \varsigma(a)\langle u,\rho_{\theta(a)}(v)\rangle \ \forall a \in Q_1, u \in \mathbf{V}_{\mathbf{c}}(s(a)), v \in \mathbf{V}_{\mathbf{c}}(t(a))\}.\]
There is a vector space isomorphism 
\[ {}^\theta\mathfrak{R}_{\mathbf{c}} \cong \bigoplus_{a \in Q_1^+} \Hom_{\C}(\mathbf{V}_{\mathbf{c}}(s(a)), \mathbf{V}_{\mathbf{c}}(t(a))) \oplus \bigoplus_{a \in Q_1^\theta} \Bil^{\sigma(s(a)) \cdot \varsigma(a)}(\mathbf{V}_{\mathbf{c}}(s(a))), \]
where $\Bil^\epsilon(\mathbf{V}_{\mathbf{c}}(s(a)))$ is the vector space of symmetric ($\epsilon = 1$) or skew-symmetric ($\epsilon = -1$) bilinear forms on $\mathbf{V}_{\mathbf{c}}(s(a))$. 
\end{defi}

\begin{defi}
Let ${}^\theta\mathsf{G}_{\mathbf{c}}'$ be the subgroup of $\mathsf{G}_{\mathbf{c}}$ which preserves the bilinear form $\langle \cdot , \cdot \rangle$. We have
\begin{equation} \label{G and G'} {}^\theta\mathsf{G}_{\mathbf{c}}' \cong \prod_{i \in Q_0^+} \mathsf{GL}(\mathbf{V}_{\mathbf{c}}(i)) \times \prod_{\substack{i \in Q_0^\theta,\\ \sigma(i)=1}} \mathsf{O}(\mathbf{V}_{\mathbf{c}}(i)) \times \prod_{\substack{i \in Q_0^\theta,\\ \sigma(i)=-1}} \mathsf{Sp}(\mathbf{V}_{\mathbf{c}}(i)).\end{equation}
The group ${}^\theta\mathsf{G}_{\mathbf{c}}'$ acts naturally on ${}^\theta\mathfrak{R}_{\mathbf{c}}$ by conjugation. 
Let ${}^\theta\mathsf{G}_{\mathbf{c}} \subseteq {}^\theta\mathsf{G}_{\mathbf{c}}'$ be the subgroup obtained from \eqref{G and G'} by replacing $ \mathsf{O}(\mathbf{V}_{\mathbf{c}}(i))$ with $\mathsf{SO}(\mathbf{V}_{\mathbf{c}}(i))$ whenever $\bx{c}(i)$ is odd. 
\end{defi}
 
\begin{exa}
Let $Q$ be the Jordan quiver and $\bx{c}=2n$. Let $\sigma = 1$ so that $\tGx{G}{c} = \mathsf{O}_{2n}$. If $\varsigma = -1$ then $\tSx{R}{c} = \mathfrak{so}_{2n}$, while if $\varsigma = 1$ then $\tSx{R}{c} = \Sym^2\C^{2n}$ as $\mathsf{O}_{2n}$-modules. 
Next, let $\sigma = -1$ so that $\tGx{G}{c} = \mathsf{Sp}_{2n}$. If $\varsigma = -1$ then $\tSx{R}{c} = \mathfrak{sp}_{2n}$, while if $\varsigma = 1$ then $\tSx{R}{c} = \bigwedge^2\C^{2n}$ as $\mathsf{Sp}_{2n}$-modules. 
\end{exa}

\begin{defi}
Let $\tGx{T}{c} \subset \tGx{B}{c} \subset \tGx{G}{c}$ be the standard maximal torus (with fundamental weights $\omega_j(i)$) and Borel subgroup in $\tGx{G}{c}$. Let $\tGx{W}{c} = N_{\tGx{G}{c}}(\tGx{T}{c})/\tGx{T}{c}$ be the corresponding Weyl group. There is an isomorphism
\[  
\tGx{W}{c} \cong \prod_{i \in Q_0^+} \mathsf{Sym}_{\bx{c}(i)} \times \prod_{i \in Q_0^\theta} \Z_2 \wr \mathsf{Sym}_{\lfloor\bx{c}(i)/2\rfloor}. 
\]
Given $\ux{d} \in {}^\theta\mathbf{Com}_{\mathbf{c}}$, let $\tGxx{W}{d} = \mathsf{W}_{\ux{d}^f} \times {}^\theta\mathsf{W}_{\bx{d}_\infty} \subset \tGx{W}{c}$.
If $\ux{e},\ux{d} \in {}^\theta\mathbf{Com}_{\bx{c}}$, let ${}^\theta_{\ux{e}}\overset{\mathbf{c}}{\mathsf{D}}_{\ux{d}}$ denote the set of the shortest representatives in $\tGx{W}{c}$ of the double cosets $\tGxx{W}{e} \backslash \tGx{W}{c} /\tGxx{W}{d}$. 
\end{defi} 

\begin{defi}
Given $\underline{\mathbf{d}} \in {}^\theta\mathbf{Com}_{\mathbf{c}}$, we call a sequence $V_\bullet$ of $Q_0$-graded isotropic subspaces 
\[ \{0\} = V_0 \subset V_1 \subset V_2 \subset \hdots \subset V_{\ld} \subset \mathbf{V}_{\mathbf{c}} \] 
an \emph{isotropic flag} of type $\ux{d}$  if $\dim_{Q_0} V_j/V_{j-1} = \mathbf{d}_j$ and $\dim_{Q_0} V_{\ld}^\perp/V_{\ld}=\mathbf{d}_\infty$. 
Any isotropic flag $V_\bullet$ can be extended to a flag $\mathrm{D}(V_{\bullet}) \in \mathfrak{F}_{\mathrm{D}(\underline{\mathbf{d}})}$ of length $2\ld+1$ by setting $V_{2\ld-k+1} = V_k^\perp$ for $k=0,\hdots,\ld$ (if $\mathbf{d}_\infty = 0$ then $V_{{\ld}+1} = V_{\ld}$ is Lagrangian). 
Let ${}^\theta\mathbf{V}_{\underline{\mathbf{d}}}$ denote the standard isotropic flag of type $\underline{\mathbf{d}}$ (consisting of coordinate subspaces with respect to some fixed basis). 
Define 
\[ \textstyle {}^\theta\mathfrak{R}_{\underline{\mathbf{d}}} := \{ \rho \in {}^\theta\mathfrak{R}_{\mathbf{c}} \mid \mathrm{D}({}^\theta\mathbf{V}_{\underline{\mathbf{d}}}) \mbox{ is } \rho\mbox{-stable}\}, \quad {}^\theta\mathsf{P}_{\underline{\mathbf{d}}}:= \Stab_{{}^\theta\mathsf{G}_{\mathbf{c}}}({}^\theta\mathbf{V}_{\underline{\mathbf{d}}} ), \quad {}^\theta\mathsf{L}_{\underline{\mathbf{d}}} := \prod_{j=1}^{l_{\underline{\mathbf{d}}}} \mathsf{G}_{\mathbf{d}_j} \times {}^\theta\mathsf{G}_{\mathbf{d}_\infty}. \]
Let ${}^\theta\mathfrak{F}_{\underline{\mathbf{d}}} \cong {}^\theta\mathsf{G}_{\mathbf{c}} / {}^\theta\mathsf{P}_{\underline{\mathbf{d}}}$ be the projective variety parametrizing isotropic flags of type $\underline{\mathbf{d}}$. 
Given $\ux{d} \succcurlyeq \ux{e}$~$\Rsl$~$\mathbf{c}$, define 
\[ {}^\theta\mathfrak{Q}_{\ux{d}} := \{ (V_\bullet, \rho) \in {}^\theta\mathfrak{F}_{\underline{\mathbf{d}}} \times {}^\theta\mathfrak{R}_{\mathbf{c}} \mid \mathrm{D}(V_\bullet) \mbox{ is } \rho\mbox{-stable} \}. \]  
Let 
\[ \tSxx{F}{d} \xleftarrow{{}^\theta\tau_{\ux{d}}} \tSxx{Q}{d} \xrightarrow{{}^\theta\pi_{\ux{d}}} \tSx{R}{c} \]
be the canonical projections. Note that ${}^\theta\tau_{\ux{d}}$ is a vector bundle while ${}^\theta\pi_{\ux{d}}$ is proper.  We abbreviate 
\[ \tSx{F}{c} :=  \bigsqcup_{\underline{\mathbf{d}} \Rsl \mathbf{c}} \tSxx{F}{d}, \quad \tSx{Q}{c} :=  \bigsqcup_{\underline{\mathbf{d}} \Rsl \mathbf{c}} \tSxx{Q}{d}, \quad {}^\theta\pi_{\mathbf{c}} := \sqcup {}^\theta\pi_{\ux{d}} \colon \tSx{Q}{c} \to \tSx{R}{c}.\] 
\end{defi}

\begin{defi}
Given $\ux{d}, \ux{e}$ $\Rsl$ $\mathbf{c}$, set
\[ {}^\theta\mathfrak{Z}_{\underline{\mathbf{d}},\underline{\mathbf{e}}} := {}^\theta\mathfrak{Q}_{\underline{\mathbf{d}}} \times_{{}^\theta\mathfrak{R}_{\mathbf{c}}} {}^\theta\mathfrak{Q}_{\underline{\mathbf{e}}}, \quad {}^\theta\mathfrak{Z}_{\mathbf{c}} := {}^\theta\mathfrak{Q}_{\mathbf{c}} \times_{{}^\theta\mathfrak{R}_{\mathbf{c}}} {}^\theta\mathfrak{Q}_{\mathbf{c}} = \bigsqcup_{\ux{d},\ux{e} \Rsl \mathbf{c}} \tSxxy{Z}{d}{e}, \] 
where the fibred product is taken with respect to 
${}^\theta\pi_{\mathbf{c}}$. We call $\tSx{Z}{c}$ the \emph{isotropic quiver Steinberg variety}. 
We define the $\tGx{G}{c}$-equivariant Borel-Moore homology groups $\tAxx{Q}{d}, {}^\theta\mathcal{Q}_{\mathbf{c}}, \tAxxy{Z}{d}{e}, {}^\theta\mathcal{Z}_{\mathbf{c}}$ in analogy to \eqref{BM homology groups Q Z}, and $\tSxxy{Z}{d}{d}^e, \tAx{Z}{c}^e := \bigoplus_{\underline{\mathbf{d}} \Rsl \mathbf{c}} \tAxxy{Z}{d}{d}^e$ in analogy to \eqref{Zw lesseq} and \eqref{Zw Ze}, respectively.
\end{defi}

Furthermore, define 
\[ 
{}^\theta\mathcal{P}_{\mathbf{c}} := H^\bullet( B\tGx{T}{c}) = \bigotimes_{i \in Q_0^+} \C[x_1(i), \hdots, x_{\mathbf{c}(i)}(i)] \otimes  \bigotimes_{i \in Q_0^\theta} \C[x_1(i), \hdots, x_{\lfloor \frac{{\mathbf{c}(i)}}{2}\rfloor}(i)], 
\]
where 
$x_j(i) := \mathtt{c}_1({}^\theta\mathfrak{V}_j(i))$ is the first Chern class of the line bundle ${}^\theta\mathfrak{V}_j(i) := E{}^\theta\Gx{T}{c} \times^{\omega_j(i)} \C$. For each $\ux{d}$ $\Rsl$ $\mathbf{c}$, 
set 
\[ {}^\theta\lamxx{d} := \tAx{P}{c}^{\tGxx{W}{d}}, \quad {}^\theta\lamx{c} := \bigoplus_{\underline{\mathbf{d}} \Rsl \mathbf{c}} {}^\theta\lamxx{d}. \] 
As in \eqref{Poincare} and \eqref{Zepols}, we can identify $\tAxxy{Z}{d}{d}^e \cong \tAxx{Q}{d} \cong {}^\theta\lamxx{d}$ and $\tAx{Z}{c}^e \cong {}^\theta\mathcal{Q}_{\mathbf{c}} \cong {}^\theta\lamx{c}$. 

\subsection{Quiver Schur algebras for quivers with an involution} 

We apply the framework of \S \ref{subsec:convolution} to the vector bundle $X = \tSx{Q}{c}$ on the isotropic quiver flag variety $\tSx{F}{c}$, the space of self-dual quiver representations $Y = \tSx{R}{c}$ and the projection $\pi = {}^\theta\pi_{\mathbf{c}}$. Then $Z = \tSx{Z}{c}$ is the isotropic quiver Steinberg variety, and we obtain a convolution algebra structure on its Borel-Moore homology $\tAx{Z}{c} = H^{\tGx{G}{c}}_\bullet(\tSx{Z}{c})$ and a $\tAx{Z}{c}$-module structure on $\tAx{Q}{c} = H_\bullet^{\tGx{G}{c}}(\tSx{Q}{c})$. 

\begin{defi}
We call $\tAx{Z}{c}$ the \emph{mixed quiver Schur algebra} associated to $(Q,\theta,\sigma, \varsigma, \mathbf{c})$, and $\tAx{Q}{c}$ its \emph{polynomial representation}. 
\end{defi}

\begin{rem} 
We would like to remark on the connection between our mixed quiver Schur algebra $\tAx{Z}{c}$ and existing constructions. 
\begin{enumerate}[label=(\roman*),topsep=2pt,itemsep=1pt]
\item In the case when $Q$ is a loopless quiver and $\theta$ is an involution with no fixed vertices, the KLR analogue of $\tAx{Z}{c}$, associated to complete (rather than partial) isotropic flags, was defined and studied by Varagnolo and Vasserot in \cite{VV2}.  
\item Our algebra $\tAx{Z}{c}$ is also related to the parabolic Steinberg algebras defined by Sauter \cite{Sau1, Sau2, Sau3}. On the one hand, Sauter's construction is somewhat more general since she also works with non-classical gauge groups. On the other hand, Sauter's construction is different from ours since she only allows parabolic flags of a certain fixed type, while we consider all the possible types at once. In effect, special cases of Sauter's parabolic Steinberg algebras appear as subalgebras in $\tAx{Z}{c}$. 
\end{enumerate}
\end{rem}

The following result carries over, with analogous proof, from the no-involution case. 

\begin{prop}
The ${}^\theta\mathcal{Z}_{\mathbf{c}}$-module ${}^\theta\mathcal{Q}_{\mathbf{c}}$ is faithful. 
There are canonical isomorphisms  
\begin{equation} \label{Ext alg iso theta} {}^\theta\mathcal{Z}_{\mathbf{c}} \cong \Ext^\bullet_{{}^\theta\mathsf{G}_{\mathbf{c}}}(({}^\theta\pi_{\mathbf{c}})_*\C_{{}^\theta\mathfrak{Q}_{\mathbf{c}}}, ({}^\theta\pi_{\mathbf{c}})_*\C_{{}^\theta\mathfrak{Q}_{\mathbf{c}}}), \quad {}^\theta\mathcal{Q}_{\mathbf{c}} \cong \Ext^\bullet_{{}^\theta\mathsf{G}_{\mathbf{c}}}(\C_{{}^\theta\mathfrak{R}_{\mathbf{c}}}, ({}^\theta\pi_{\mathbf{c}})_*\C_{{}^\theta\mathfrak{Q}_{\mathbf{c}}}) \end{equation} 
intertwining the convolution product with the Yoneda product, and the convolution action with the Yoneda action. 
\end{prop} 

\begin{defi} \label{defi:mergesandsplits theta}
We have the following analogues of merges, splits, idempotents and crossings from Definition \ref{defi:mergesandsplits} in $\tAx{Z}{c}$:
\[ \textstyle {}^\theta\bigcurlywedge_{\ux{d}}^{\ux{e}}:= [\tSxxy{Z}{e}{d}^e], \quad {}^\theta\bigcurlyvee_{\ux{e}}^{\ux{d}}:= [\tSxxy{Z}{d}{e}^e], \quad {}^\theta\mathsf{e}_{\ux{d}} := [\tSxxy{Z}{d}{d}^e], \quad \prescript{\theta}{}{\cross_{\ux{d}}^k} := {}^\theta\bigcurlyvee^{s_k(\ux{d})}_{\wedge_k^\theta(\underline{\mathbf{d}})} \star \ {}^\theta\bigcurlywedge_{\ux{d}}^{\wedge_k^\theta(\underline{\mathbf{d}})} \]
for $\ux{d} \succcurlyeq \ux{e}$~$\Rsl$~$\bx{c}$ and $1 \leq k \leq \ell_{\underline{\mathbf{d}}}$. We say that a merge or split is \emph{elementary} if $\ux{e} = \wedge_k^\theta(\ux{d})$. 
If $1 \leq k \leq \ld -1$, we depict elementary merges and splits diagrammatically in analogy to the elementary merges and splits in Definition \ref{defi:mergesandsplits}.  
More precisely, to the elementary merge ${}^\theta\bigcurlywedge_{\ux{d}}^{\wedge^k(\ux{d})}$ we associate the diagram 
\[
\tikz[thick,xscale=2,yscale=1.2]{ \small 
\draw (0,0) node[below] {$\bx{d}_k$} to [out=90,in=-90](.3,.5)
(.6,0) node[below] {$\bx{d}_{k+1}$} to [out=90,in=-90] (.3,.5)
(.3,.5) -- (.3,.8) node[above] {$\bx{d}_k+\bx{d}_{k+1}$};
\node at (1.2,0.4) {$:=$};
\begin{scope}[xshift=3cm]
\draw (-1.2,0) -- (-1.2,.8) node[below,at start]{$\bx{d}_1$} node[above]{$\bx{d}_1$};
\node at (-0.9,.4) {$\cdots$}; 
\draw (-0.6,0) -- (-0.6,.8) node[below,at start]{$\bx{d}_{k-1}$} node[above]{$\bx{d}_{k-1}$};
\draw (0,0) node[below] {$\bx{d}_k$} to [out=90,in=-90](.3,.5)
(.6,0) node[below] {$\bx{d}_{k+1}$} to [out=90,in=-90] (.3,.5)
(.3,.5) -- (.3,.8) node[above] {$\bx{d}_k+\bx{d}_{k+1}$};
\draw (1.2,0) -- (1.2,.8) node[below,at start]{$\bx{d}_{k+2}$} node[above]{$\bx{d}_{k+2}$};
\node at (1.5,.4) {$\cdots$}; 
\draw (1.8,0) -- (1.8,.8) node[below,at start]{$\bx{d}_{\ld}$} node[above]{$\bx{d}_{\ld}$}; 
\draw (2.4,0) -- (2.4,.8) node[below,at start]{$\bx{d}_{\infty}$} node[above]{$\bx{d}_{\infty}$};
\end{scope} 
}\] 
and to the elementary split ${}^\theta\bigcurlyvee^{\ux{d}}_{\wedge^k(\ux{d})}$ the vertically reflected diagram.

If $k = \ld$, we associate to the elementary merge ${}^\theta\bigcurlywedge_{\ux{d}}^{\wedge^{k}(\ux{d})}$  the new diagram
\[
\tikz[thick,xscale=2,yscale=1.2]{ \small 
\draw[line width = 1mm] (0,0) node[below] {$\bx{d}_{\ld}$} to [out=90,in=-90](.3,.5)
(.3,.49) -- (.3,.8) node[above] {$\mathrm{D}(\bx{d}_{\ld})+\bx{d}_\infty$};
\draw 
(.6,0) node[below] {$\bx{d}_{\infty}$} to [out=90,in=-90] (.3,.5); 
\node at (1.2,0.4) {$:=$};
\begin{scope}[xshift=3cm]
\draw (-1.2,0) -- (-1.2,.8) node[below,at start]{$\bx{d}_1$} node[above]{$\bx{d}_1$};
\node at (-0.9,.4) {$\cdots$}; 
\draw (-0.6,0) -- (-0.6,.8) node[below,at start]{$\bx{d}_{\ld-1}$} node[above]{$\bx{d}_{\ld-1}$};
\draw[line width = 1mm]  (0,0) node[below] {$\bx{d}_{\ld}$}  to [out=90,in=-90](.3,.5) 
(.3,.49) -- (.3,.8) node[above]  {$\mathrm{D}(\bx{d}_{\ld})+\bx{d}_\infty$};
\draw
(.6,0) node[below] {$\bx{d}_{\infty}$} to [out=90,in=-90] (.3,.5);
\end{scope} 
}\] 
and to the elementary split ${}^\theta\bigcurlyvee^{\ux{d}}_{\wedge^k(\ux{d})}$ the vertically reflected diagram.
\end{defi}

We have the following analogue of Proposition \ref{pro:transitivity}, which also follows directly from Lemma \ref{lem:conv of classes}. 

\begin{prop} \label{pro:transitivity theta} 
We list several basic relations which hold in $\tAx{Z}{c}$. 
\begin{enumerate}[label=\alph*), font=\textnormal,noitemsep,topsep=3pt,leftmargin=1cm]
\item Let $\ux{d} \succ \ux{e} \succ \ux{f}$ $\Rsl$ $\bx{c}$. Merges and splits satisfy the following \emph{transitivity} relations: 
\[ \textstyle {}^\theta\mer{e}{f} \ \star \ {}^\theta\mer{d}{e} = {}^\theta\mer{d}{f}, \quad {}^\theta\spl{e}{d} \ \star \ {}^\theta\spl{f}{e} = {}^\theta\spl{f}{d}.\]
\item 
Let $\ux{d}$ $\Rsl$ $\bx{c}$. Elementary merges satisfy the relations \eqref{R1 relation} and the following new relation 
\begin{equation} \tag{${}^\theta$R1} \label{theta R1}
\tikz[thick,xscale=2,yscale=1, baseline=0.8cm]{ \small
\draw (0,0) node[below] {$\bx{d}_{\ld-1}$} to [out=90,in=-90](.3,.5)
(.6,0) node[below] {$\bx{d}_{\ld}$} to [out=90,in=-90] (.3,.5)
(.3,.5) -- (.3,.8) (0.3,.79) node {\scalebox{0.8}{$\bullet$}}  
(1.2,0) -- (1.2,.8) node[below,at start]{$\bx{d}_{\infty}$}
(1.2,.8) to [out=90,in=-90](.75,1.3);
\draw[line width = 1mm] (.3,.8) to [out=90,in=-90](.75,1.3)
(.75,1.29) -- (.75,1.6) node[above] {$\mathrm{D}(\bx{d}_{\ld-1}+\bx{d}_{\ld})+\bx{d}_{\infty}$}; 
\node at (2,0.8) {$=$};
\begin{scope}[xshift=2.8cm]
\draw[line width = 1mm]  
(0,.8) to [out=90,in=-90](.45,1.3)
(.45,1.29) -- (.45,1.6) node[above] {$\mathrm{D}(\bx{d}_{\ld-1}+\bx{d}_{\ld})+\bx{d}_{\infty}$}
(0.6,0) node[below] {$\bx{d}_{\ld}$} to [out=90,in=-90](.9,.5)
(.9,.49) -- (.9,.8) (.9,.79) node {\scalebox{0.8}{$\bullet$}}; 
\draw
(0,0) -- (0,.8) node[below,at start]{$\bx{d}_{\ld-1}$} (0,.79) node {\scalebox{0.8}{$\bullet$}}
(1.2,0) node[below] {$\bx{d}_{\infty}$} to [out=90,in=-90] (.9,.5)
(0.9,.8) to [out=90,in=-90](.45,1.3);
\end{scope}
}\end{equation}
Elementary splits satisfy the relations \eqref{R2 relation} and the following new relation 
\begin{equation} \tag{${}^\theta$R2} \label{theta R2}
\tikz[thick,xscale=2,yscale=-1, baseline=-0.8cm]{ \small
\draw (0,0) node[above] {$\bx{d}_{\ld-1}$} to [out=90,in=-90](.3,.5)
(.6,0) node[above] {$\bx{d}_{\ld}$} to [out=90,in=-90] (.3,.5)
(.3,.5) -- (.3,.8)  (0.3,.79) node {\scalebox{0.8}{$\bullet$}} 
(1.2,0) -- (1.2,.8) node[above,at start]{$\bx{d}_{\infty}$} 
(1.2,.8) to [out=90,in=-90](.75,1.3);
\draw[line width = 1mm] (.3,.8) to [out=90,in=-90](.75,1.3)
(.75,1.29) -- (.75,1.6) node[below] {$\mathrm{D}(\bx{d}_{\ld-1}+\bx{d}_{\ld})+\bx{d}_{\infty}$}; 
\node at (2,0.8) {$=$};
\begin{scope}[xshift=2.8cm]
\draw[line width = 1mm]  
(0,.8) to [out=90,in=-90](.45,1.3)
(.45,1.29) -- (.45,1.6) node[below] {$\mathrm{D}(\bx{d}_{\ld-1}+\bx{d}_{\ld})+\bx{d}_{\infty}$}
(0.6,0) node[above] {$\bx{d}_{\ld}$} to [out=90,in=-90](.9,.5)
(.9,.49) -- (.9,.8) (.9,.79) node {\scalebox{0.8}{$\bullet$}}; 
\draw
(0,0) -- (0,.8) node[above,at start]{$\bx{d}_{\ld-1}$} (0,.79) node {\scalebox{0.8}{$\bullet$}}
(1.2,0) node[above] {$\bx{d}_{\infty}$} to [out=90,in=-90] (.9,.5)
(0.9,.8) to [out=90,in=-90](.45,1.3);
\end{scope}
}\end{equation}
\end{enumerate}
\end{prop}

\subsection{Basis and generators}

We want to construct a basis for $\tAx{Z}{c}$ analogous to the Bott-Samelson basis of $\Ax{Z}{c}$ from Theorem~\ref{thm:basis}. We begin by adapting the combinatorics of refinements (see \S \ref{subsec: refinements}) to the present setting. 

Let $\theta \colon \mathbf{N}_{\bx{c}} \to \mathbf{N}_{\bx{c}}$ be the involution defined by 
\[
(k,i) \mapsto \left\{ \begin{array}{ll} 
(k,\theta(i)) & \mbox{if } i \in Q_0^+ \mbox{ and } 1 \leq k \leq \bx{c}(i), \\
(\bx{c}(i)-k+1,i) & \mbox{if } i \in Q_0^\theta  \mbox{ and } 1 \leq k \leq \bx{c}(i). 
\end{array} \right. 
\]
If $\lambda \in \mathbf{Par}_{\bx{c}}^n$, we say that $\lambda$ is an \emph{isotropic partitioning} of $\bx{c}$ of length ${}^\theta \ell_\lambda = \lfloor n/2 \rfloor$ if $\lambda^{-1}(k) = \theta(\lambda^{-1}(n-k+1))$ for $0 \leq k \leq n$. 
Let ${}^\theta\mathbf{Par}_{\bx{c}} \subset \mathbf{Par}_{\bx{c}}$ denote the set of isotropic partitionings of $\bx{c}$, and let ${}^\theta\mathbf{Par}_{\bx{c}}^m$ denote the subset of those isotropic partitionings which have length $m$. 

Let ${}^\theta C$ and ${}^\theta P$ be the unique functions making the following diagram commute
\[
\begin{tikzcd}
\mathbf{Par}_{\bx{c}} \arrow[bend left=10,r, "C"] & \mathbf{Com}_{\bx{c}} \arrow[bend left=10, l, "P"] \\
{}^\theta\mathbf{Par}_{\bx{c}} \arrow[u,hookrightarrow] \arrow[bend left=10,r, "{}^\theta C"] & {}^\theta\mathbf{Com}_{\bx{c}} \arrow[hookrightarrow,u,"\mathrm{D}",swap] \arrow[bend left=10, l, "{}^\theta P"]
\end{tikzcd}
\]
The set ${}^\theta\mathbf{Par}_{\bx{c}}^m$ is endowed with natural $\Z_2 \wr \mathsf{Sym}_m$- and $\tGx{W}{c}$-actions. 
It is easy to check that Lemma \ref{lem: C P functions} still holds if we replace $\mathbf{Par}_{\bx{c}}$, $\mathbf{Com}_{\bx{c}}$, $C$, $P$, $\mathsf{Sym}_n$, $\Gx{W}{c}$ and $\Gxx{W}{d}$ by their isotropic analogues. The following lemma follows directly from the definitions.

\begin{lem}
If $\lambda, \mu \in {}^\theta\mathbf{Par}_{\bx{c}}$ then $\lambda \mathrel{\textnormal{\leo}} \mu \in {}^\theta\mathbf{Par}_{\bx{c}}$. 
\end{lem}

\begin{defi} \label{defi: orbit datum theta}
We call a triple $(\ux{e}, \ux{d}, w)$, consisting of $\ux{e},\ux{d} \in {}^\theta\mathbf{Com}_{\bx{c}}$ and $w \in{}^\theta_{\ux{e}}\overset{\mathbf{c}}{\mathsf{D}}_{\ux{d}}$, an \emph{isotropic orbit datum}. This name is motivated by the fact that orbit data naturally label the $\tGx{G}{c}$-orbits in $\tSx{F}{c} \times \tSx{F}{c}$. We define the corresponding \emph{refinement datum} $(\widehat{\ux{e}}, \widehat{\ux{d}}, u)$ in analogy to Definition \ref{defi: orbit datum}. More precisely, if we abbreviate $\lambda = {}^\theta P(\ux{e})$ and $\mu = w \cdot {}^\theta P(\ux{d})$, then 
$$\widehat{\ux{e}} := {}^\theta C(\lambda \mathrel{\leo} \mu), \quad \widehat{\ux{d}} := {}^\theta C(\mu \mathrel{\leo} \lambda)$$ and $u \in \Z_2 \wr \mathsf{Sym}_{{}^\theta \ell_{\lambda \mathrel{\textnormal{\leo}} \mu}}$ is the unique permutation sending $\lambda \mathrel{\textnormal{\leo}} \mu$ to $\mu \mathrel{\textnormal{\leo}} \lambda$. We also choose a reduced expression $u = s_{j_k} \cdot \hdots \cdot s_{j_1}$ and define the associated \emph{crossing datum} $(\ux{e}^{0}, \hdots, \ux{e}^{2k})$ in the same way as in Definition \ref{defi: crossing datum}. 
\end{defi}

\begin{exa} 
Let $Q$ by a quiver such that $Q_0$ is a singleton. Then there is a unique involution on $Q$. Let $\bx{c}=14$ and choose any duality structure on $Q$. We identify $7+k = -k$ for $1 \leq k \leq 7$ so that $\mathbf{N}_{\bx{c}} = \{ \pm 1, \hdots, \pm 7\}$. Let $s_7 \in \tGx{W}{c} = \Z_2 \wr \mathsf{Sym}_7$ be the element swapping $7$ and $-7$. Suppose that $\ux{e} = (3,2,4)$, $\ux{d} = (4,2,2)$ and $w = s_5s_6s_7s_6s_5s_3s_4s_5s_6$. Then 
\begin{alignat*}{8}
\lambda && \ = \ & [1,2,3][4,5][\pm6,\pm7][\text{-}5,\text{-}4][\text{-}3,\text{-}2,\text{-}1],& \quad \mu && \ = \ & [1,2,4,\text{-}5][6,7][\pm3][\text{-}7,\text{-}6][5,\text{-}4,\text{-}2,\text{-}1], \\
\lambda \mathrel{\textnormal{\leo}} \mu && \ = \ & [1,2][3][4][5][6,7][\text{-}7,\text{-}6][\text{-}5][\text{-}4][\text{-}3][\text{-}2,\text{-}1],& \quad \mu \mathrel{\textnormal{\leo}} \lambda && \ = \ & [1,2][4][\text{-}5][6,7][3][\text{-}7,\text{-}6][\text{-}3][5][\text{-}4][\text{-}2,\text{-}1]. 
\end{alignat*}
Hence \[\widehat{\ux{e}} = (2,1,1,1,2,0), \quad \widehat{\ux{d}} = (2,1,1,2,1,0), \quad u = s_4s_3s_2s_4s_5s_4 \in \Z_2 \wr \mathsf{Sym}_5.\]
\end{exa}

\begin{defi}
Given an isotropic orbit datum $(\ux{e}, \ux{d}, w)$, the corresponding refinement and crossing data, and $c \in  {}^\theta\Lambda_{\ux{e}^{2k}}$, we define elements $\ux{d} \underset{w}{\overset{c}{\Longrightarrow}} \ux{e}$ in analogy to Definition \ref{defi: basis elements for QS}, i.e., 
\begin{equation*} \textstyle
 \ux{d} \underset{w}{\overset{c}{\Longrightarrow}} \ux{e} \quad := \quad {}^\theta\bigcurlywedge_{\ux{e}^{0}}^{\ux{e}} \ \star \ {}^\theta\cross_{\ux{e}^{2}}^{j_1} \ \star \ \hdots \ \star \ {}^\theta\cross_{\ux{e}^{2k}}^{j_k} \ \star \ c \ \star \ {}^\theta\bigcurlyvee_{\ux{d}}^{\ux{e}^{2k}} \in {}^\theta\Axxy{Z}{e}{d}. 
\end{equation*}
\end{defi} 
We have the following analogue of Theorem \ref{thm:basis}. 

\begin{thm} \label{thm:basis theta} 
If we let $(\ux{e}, \ux{d}, w)$ range over all isotropic orbit data and $c$ range over a basis of~${}^\theta\Lambda_{\ux{e}^{2k}}$, then the elements $\ux{d} \underset{w}{\overset{c}{\Longrightarrow}} \ux{e}$ form a basis of $\tAx{Z}{c}$. 
\end{thm}

\begin{proof}
The proof of Theorem \ref{thm:basis} uses only three ingredients: the Bruhat decomposition, Proposition \ref{lem: Sym refined} and Lemma \ref{lem: intersection and stability}. The Bruhat decomposition of course generalizes to reductive groups of type B, C and D. Lemma  \ref{lem: intersection and stability} also generalizes straightforwardly. To generalize Proposition \ref{lem: Sym refined}, one only needs to modify its proof by replacing inversions associated to the symmetric group by inversions associated to the Weyl group of type B (see, e.g., \cite[Proposition 8.1.1]{BB}). 
\end{proof}

Theorem \ref{thm:basis theta} and Proposition \ref{pro:transitivity theta} directly imply the following analogue of Corollary \ref{cor:el spl mer}. 

\begin{cor} \label{cor:el spl mer theta}
Elementary merges, elementary splits and the polynomials $\tAx{Z}{c}^e$ generate $\tAx{Z}{c}$ as an algebra. 
\end{cor}

\subsection{Monoidal structure and categorification.}

We now consider the relationship between the categories of modules over $\tAx{Z}{c}$ and $\Ax{Z}{c}$. 
In this subsection we view $\tAx{Z}{c}$ and $\Ax{Z}{c}$ as graded algebras, with the gradings imported from the gradings on the corresponding Ext-algebras via the isomorphisms \eqref{Ext alg iso} and \eqref{Ext alg iso theta}. 
We begin by recalling the monoidal structure on the direct sum $\mathcal{Z}\mbox{-}\pmmod$ of the categories of finitely generated graded projective modules over quiver Schur algebras $\Ax{Z}{c}$ for all dimension vectors $\bx{c} \in \Gamma$. We then show that the monoidal category $\mathcal{Z}\mbox{-}\pmmod$ acts on the corresponding category ${}^\theta\mathcal{Z}\mbox{-}\pmmod$ of modules over the algebras $\tAx{Z}{c}$. Passing to Grothendieck groups, we obtain a $K_0(\mathcal{Z})$-module and -comodule structure on $K_0({}^\theta\mathcal{Z})$, which we relate to the Hall module of the category of self-dual representations of the quiver $Q$ introduced by Young in \cite{You1}. 

One can easily show (as in, e.g., \cite[\S 2.4]{SW} or \cite[\S 2.6]{KL1}) that there are canonical (non-unital) injective graded ring homomorphisms 
\begin{equation} \label{inc c c'} i_{\bx{c},\bx{c}'} \colon \mathcal{Z}_{\bx{c}} \otimes \mathcal{Z}_{\bx{c}'} \hookrightarrow \mathcal{Z}_{\bx{c}+\bx{c}'}, \end{equation} 
for all $\bx{c}, \bx{c}' \in \Gamma$, induced by inclusions of the corresponding polynomial representations
\begin{equation} \label{inc c c' rep} \mathcal{Q}_{\bx{c}} \otimes \mathcal{Q}_{\bx{c}'} \hookrightarrow \mathcal{Q}_{\bx{c}+\bx{c}'}.  \end{equation} 
Diagrammatically, these inclusions are depicted by a horizontal composition of diagrams. 
They define an associative algebra structure on the direct sums $\mathcal{Z} = \bigoplus_{\bx{c} \in \Gamma} \Ax{Z}{c}$ and $\mathcal{Q} = \bigoplus_{\bx{c} \in \Gamma} \Ax{Q}{c}$, which is referred to as the \emph{horizontal multiplication}. 
The inclusions \eqref{inc c c'} also give rise to induction and restriction functors 
\begin{equation} \label{Ind Res functors} \Ind_{\bx{c},\bx{c}'} \colon  \mathcal{Z}_{\bx{c}} \otimes \mathcal{Z}_{\bx{c}'}\mbox{-}\mmod \to \mathcal{Z}_{\bx{c}+\bx{c}'}\mbox{-}\mmod, \quad \Res_{\bx{c},\bx{c}'} \colon \mathcal{Z}_{\bx{c}+\bx{c}'}\mbox{-}\mmod \to \mathcal{Z}_{\bx{c}} \otimes \mathcal{Z}_{\bx{c}'}\mbox{-}\mmod, \end{equation} 
where by, e.g., $\Ax{Z}{c}\mbox{-}\mmod$, we mean the category of finitely generated graded left $\Ax{Z}{c}$-modules. These functors restrict to subcategories of projective modules. Setting 
\[ M \otimes N = \mathcal{Z}_{\bx{c}+\bx{c}'} \otimes_{\mathcal{Z}_{\bx{c}} \otimes \mathcal{Z}_{\bx{c}'}} M \boxtimes N \]
for $M \in \mathcal{Z}_{\bx{c}}\mbox{-}\pmmod$ and $N \in \mathcal{Z}_{\bx{c}'}\mbox{-}\pmmod$, and $\mathbf{1} = \mathcal{Z}_0$ defines a monoidal structure on the direct sum of categories
\[ \mathcal{Z}\mbox{-pmod} = \bigoplus_{\bx{c} \in \Gamma} \Ax{Z}{c}\mbox{-pmod}.\]
Let $K_0(\mathcal{Z}) = K_0(\mathcal{Z}\mbox{-pmod})$ be its Grothendieck group, considered as a $\Z[q^{\pm 1}]$-module. 
The functors \eqref{Ind Res functors} induce maps 
\[ K_0(\mathcal{Z}) \otimes K_0(\mathcal{Z}) \to K_0(\mathcal{Z}), \quad 
K_0(\mathcal{Z}) \to K_0(\mathcal{Z}) \otimes K_0(\mathcal{Z}),\]
which turn $K_0(\mathcal{Z})$ into a $\Gamma$-graded $\Z[q^{\pm 1}]$-bialgebra. 
For special choices of the quiver $Q$, the bialgebra $K_0(\mathcal{Z})$ can be identified with (the opposite of) the generic nilpotent Hall algebra associated to the category of representations of $Q$ over finite fields. For more information about this algebra we refer the reader to, e.g., \cite{Schiff}. 

\begin{prop} \label{pro: K0 gen nilp Hall}
Let $Q$ be one of the following quivers: a Dynkin quiver, the $A_\infty$ quiver, the Jordan quiver or a cyclic quiver. Then $K_0(\mathcal{Z})^{op}$ is canonically isomorphic to the integral form of the generic nilpotent Hall algebra of the quiver $Q$.
\end{prop}

\begin{proof}
By Theorem \ref{thm: our vs SW}, there is an isomorphism of algebras $\Ax{Z}{c} \cong \Ax{Z}{c}^{SW}$ for each $\bx{c} \in \Gamma$. The explicit description of this isomorphism from \cite[Proposition 9.4, 9.6]{MS} implies that there is a commutative diagram of ring homomorphisms 
\[
\begin{tikzcd} 
\mathcal{Z}_{\bx{c}'}^{SW} \otimes \mathcal{Z}_{\bx{c}}^{SW} \arrow[r, hookrightarrow] & \mathcal{Z}_{\bx{c}+\bx{c}'}^{SW}  \\
\mathcal{Z}_{\bx{c}}^{SW} \otimes \mathcal{Z}_{\bx{c}'}^{SW} \arrow{u}{\wr}[swap]{flip} \\
\mathcal{Z}_{\bx{c}} \otimes \mathcal{Z}_{\bx{c}'}  \arrow[r, hookrightarrow] \arrow[u,"\wr"] & \mathcal{Z}_{\bx{c}+\bx{c}'} \arrow[uu, "\wr"]
\end{tikzcd} 
\] 
Passing to Grothendieck groups, we see that $K_0(\mathcal{Z})^{op} \cong K_0(\mathcal{Z}^{SW})$ as algebras. The proposition now follows from \cite[Proposition 5.12]{SW}. 
\end{proof}

We now bring the mixed quiver Schur algebras $\tAx{Z}{c}$ into the picture.

\begin{lem} If $\mathbf{a} \in \Gamma$, $\mathbf{b} \in \Gamma^\theta$ satisfy $\mathrm{D}(\mathbf{a}) + \mathbf{b} = \mathbf{c}$, 
then there is an injective (non-unital) ring homomorphism 
\begin{equation} \label{nonunital inclusions} \textstyle i_{\mathbf{a},\mathbf{b}} \colon \Ax{Z}{a} \otimes \tAx{Z}{b} \hookrightarrow \tAx{Z}{c}, \quad \quad \bigcurlywedge_{\ux{d}}^{\ux{e}} \ \otimes \ {}^\theta\bigcurlywedge_{\ux{d}'}^{\ux{e}'}  \mapsto {}^\theta\bigcurlywedge_{\ux{d}''}^{\ux{e}''}, \quad \bigcurlyvee_{\ux{e}}^{\ux{d}} \ \otimes \ {}^\theta\bigcurlywedge_{\ux{d}'}^{\ux{e}'} \mapsto {}^\theta\bigcurlyvee^{\ux{d}''}_{\ux{e}''},
\end{equation}
sending a polynomial $f \otimes g \in \lamxx{d} \otimes {}^\theta\Lambda_{\ux{d}'}$ to $f\cdot g \in {}^\theta\Lambda_{\ux{d}''}$, where $\ux{d}'' = \ux{d} \cup \ux{d}'$ and $\ux{d}'' \succ \ux{e}''$ $\Rsl$ $\bx{c}$. 
\end{lem}

\begin{proof}
Let ${}^\theta\mathfrak{Z}^{}_{\mathbf{a,b}} := \bigsqcup {}^\theta\mathfrak{Z}_{\ux{d}'',\ux{e}''}$, where the disjoint union ranges over all $\ux{d}'', \ux{e}''$ $\Rsl$ $\bx{c}$ which can be expressed as a concatenation $\ux{d} \cup \ux{d}'$, for some $\ux{d} \rightslice \bx{a}$ and $\ux{d}'$ $\Rsl$ $\bx{b}$ (and analogously for $\ux{e}''$). Clearly ${}^\theta\mathcal{Z}^{}_{\mathbf{a,b}} := H^{{}^\theta\mathsf{G}_{\mathbf{c}}}_\bullet({}^\theta\mathfrak{Z}^{}_{\mathbf{a,b}})$ is a convolution subalgebra of $\tAx{Z}{c}$. The forgetful maps ${}^\theta\mathfrak{Q}_{\ux{d}''} \to \mathfrak{Q}_{\ux{d}}$ (remembering only the first $\ld$ steps in an isotropic flag) and ${}^\theta\mathfrak{Q}_{\ux{d}''} \to {}^\theta\mathfrak{Q}_{\ux{d}'}$ (remembering only the last $\ell_{\ux{d}'}+1$ steps) induce a map ${}^\theta\mathfrak{Z}^{}_{\mathbf{a,b}} \to \Sx{Z}{a}\times\tSx{Z}{b}$. The pullback $\Ax{Z}{a} \otimes \tAx{Z}{b} \to \tAx{Z}{c}$ with respect to the latter is injective, and it is easy to check that it is compatible with the convolution product and that, explicitly, it is given by \eqref{nonunital inclusions}. 
\end{proof} 

As before, the inclusions  \eqref{nonunital inclusions} are depicted diagrammatically via a horizontal composition of diagrams. 
They give rise to functors 
\begin{equation} \label{Ind Res functors theta} \Ind_{\bx{a},\bx{b}} \colon  \mathcal{Z}_{\bx{a}} \otimes {}^\theta\mathcal{Z}_{\bx{b}}\mbox{-}\pmmod \to {}^\theta\mathcal{Z}_{\mathrm{D}(\bx{a})+\bx{b}}\mbox{-}\pmmod, \quad \Res_{\bx{a},\bx{b}} \colon {}^\theta\mathcal{Z}_{\mathrm{D}(\bx{a})+\bx{b}}\mbox{-}\pmmod \to \mathcal{Z}_{\bx{a}} \otimes {}^\theta\mathcal{Z}_{\bx{b}}\mbox{-pmod}. \end{equation} 
Let ${}^\theta\mathcal{Z}\mbox{-pmod}$ be the direct sum of categories 
\[ {}^\theta\mathcal{Z}\mbox{-pmod} = \bigoplus_{\bx{c} \in \Gamma^\theta} \tAx{Z}{c}\mbox{-pmod} \]
and let $K_0({}^\theta\mathcal{Z}) = K_0({}^\theta\mathcal{Z}\mbox{-pmod})$ be its Grothendieck group. The following proposition, whose proof is standard, summarizes the relation between the categories $\mathcal{Z}\mbox{-pmod}$ and ${}^\theta\mathcal{Z}\mbox{-pmod}$. 

\begin{prop} \label{pro: monoidal action}
The following hold. 
\begin{enumerate}[label=\alph*), font=\textnormal,noitemsep,topsep=3pt,leftmargin=1cm] 
\item 
The monoidal category $\mathcal{Z}\mbox{-}\pmmod$ acts (see, e.g., \cite{HO}) on ${}^\theta\mathcal{Z}\mbox{-}\pmmod$ via 
\[ M * N = {}^\theta\mathcal{Z}_{\mathrm{D}(\bx{a})+\bx{b}} \otimes_{\mathcal{Z}_{\bx{a}} \otimes \tAx{Z}{b}} M \boxtimes N, \]
for $M \in \mathcal{Z}_{\bx{a}}\mbox{-}\pmmod$ and $N \in {}^\theta\mathcal{Z}_{\bx{b}}\mbox{-}\pmmod$. 
\item 
The functors \eqref{Ind Res functors theta} induce maps 
\[ K_0(\mathcal{Z}) \otimes K_0({}^\theta\mathcal{Z}) \to K_0({}^\theta\mathcal{Z}), \quad 
K_0({}^\theta\mathcal{Z}) \to K_0(\mathcal{Z}) \otimes K_0({}^\theta\mathcal{Z}),\]
which turn $K_0({}^\theta\mathcal{Z})$ into a $\Gamma^\theta$-graded $K_0(\mathcal{Z})$-module and -comodule. 
\end{enumerate} 
\end{prop} 

\begin{rem}
In \cite{You1}, Young defined a Hall module associated to the category of self-dual representations of a quiver with an involution. The Hall module is a module as well as a comodule over the Hall algebra associated to the same quiver. We expect that, for a general quiver $Q$ with an involution $\theta$, $K_0(\mathcal{Z})$ is isomorphic to a subalgebra of the Hall algebra of $Q$ and $K_0({}^\theta\mathcal{Z})$ is isomorphic to a subspace of the Hall module of $(Q, \theta)$ stable under the action and coaction of $K_0(\mathcal{Z})$. Since $K_0(\mathcal{Z})$ contains the composition subalgebra associated to $Q$, \cite[Theorem 3.5]{You1} implies that $K_0({}^\theta\mathcal{Z})$ is also a module over $B_\theta(\g_Q)$, the algebra introduced by Enomoto and Kashiwara \cite{EK1, EK2} in the context of symmetric crystals. 
The KLR analogue of $K_0({}^\theta\mathcal{Z})$ was studied by Varagnolo and Vasserot \cite{VV2}, who showed that it is isomorphic to a certain highest weight module over $B_\theta(\g_Q)$. 
\end{rem} 

\section{Connection to cohomological Hall algebras}

In this section we relate quiver Schur algebras to the cohomological Hall algebra (CoHA) of a quiver $Q$ (without potential) introduced by Kontsevich and Soibelman \cite{KS}. More specifically, we interpret merges and splits as iterated multiplication and comultiplication in the CoHA. This gives an action of quiver Schur algebras on the tensor algebra of the CoHA, which we identify with the direct sum of the polynomial representations of all the quiver Schur algebras associated to $Q$. 
In the case of a quiver endowed with an involution and a duality structure, we relate mixed quiver Schur algebras to the cohomological Hall module (CoHM) introduced by Young \cite{You}, realizing merges and splits as action and coaction operators. An algebraic manifestation of these connections is a new interpretation of the shuffle description of the CoHA and the CoHM in terms of Demazure operators.

\subsection{The cohomological Hall algebra.}

We start by recalling the definition of the CoHA from \cite[\S 2.2]{KS}. 
Let $Q$ be a finite quiver. 
Given $\mathbf{c} \in \Gamma$ and $\underline{\mathbf{d}} \rightslice \mathbf{c}$, set 
\[ \textstyle \mathcal{H}_{\mathbf{c}} := H_{\mathsf{G}_{\mathbf{c}}}^\bullet (\mathfrak{R}_{\mathbf{c}}), \quad 
\mathcal{H}_{\underline{\mathbf{d}}} := \bigotimes_{j=1}^{\ell_{\ux{d}}} \mathcal{H}_{\mathbf{d}_j}, \quad \mathcal{H} := \bigoplus_{\mathbf{c} \in \Gamma} \mathcal{H}_{\mathbf{c}}. \]
The K\"{u}nneth map and the homotopy equivalences $ \textstyle \mathfrak{R}_{\underline{\mathbf{d}}} \twoheadrightarrow \prod_{j} \mathfrak{R}_{\mathbf{d}_j}$ and $ \mathsf{P}_{\underline{\mathbf{d}}} \twoheadrightarrow \mathsf{L}_{\underline{\mathbf{d}}}$ yield canonical isomorphisms 
\begin{equation} \label{levitopara} \textstyle \Axx{H}{d} \cong H_{\mathsf{L}_{\underline{\mathbf{d}}}}^\bullet(\prod_{j}^{\ell_{\ux{d}}} \mathfrak{R}_{\mathbf{d}_j}) \cong H_{\mathsf{P}_{\underline{\mathbf{d}}}}^\bullet(\mathfrak{R}_{\underline{\mathbf{d}}}). \end{equation}

\begin{defi} 
Given $\ux{d} \succ \ux{e}$, we have a closed embedding $(\Sxx{R}{d})_{\Gxx{P}{d}} \overset{i}{\hookrightarrow} (\Sxx{R}{e})_{\Gxx{P}{d}}$ and a fibration $(\Sxx{R}{e})_{\Gxx{P}{d}}  \overset{p}{\twoheadrightarrow}(\Sxx{R}{e})_{\Gxx{P}{e}}$ 
with smooth and compact fibre. 
Using the identification \eqref{levitopara}, we get operators 
\begin{equation} \label{CoHA mult} \textstyle \mul{d}{e} \colon \Axx{H}{d} \xrightarrow{p_* \circ i_*} \Axx{H}{e}, \quad  \com{e}{d} \colon \Axx{H}{e} \xrightarrow{i^* \circ p^*}  \Axx{H}{d}.\end{equation} 
We abbreviate $\mathsf{m}_{\ux{d}}^{\bx{c}} = \mathsf{m}_{\ux{d}}^{(\bx{c})}$, etc. 
Let $\mathsf{m} \colon \mathcal{H} \otimes \mathcal{H} \to \mathcal{H}$ and $\mathsf{com} \colon \mathcal{H} \to \mathcal{H} \otimes \mathcal{H}$ be the operators defined by the condition that $\mathsf{m}|_{\Axx{H}{d}} = \mathsf{m}_{\ux{d}}^{\bx{c}}$, and that the projection of $\mathsf{com}|_{\Ax{H}{c}}$ onto $\Axx{H}{d}$ equals $\mathsf{com}^{\ux{d}}_{\bx{c}}$, for all dimension vectors $\bx{c} \in \Gamma$ and vector compositions $\ux{d} \in \mathbf{Com}_{\mathbf{c}}^2$ of length two. 
\end{defi}

\begin{defi} \label{defi: CoHA}
The \emph{cohomological Hall algebra} associated to the quiver $Q$ is the $\Gamma$-graded vector space $\mathcal{H}$ together with multiplication given by $\mathsf{m}$. By \cite[Theorem 1]{KS}, $(\mathcal{H}, \mathsf{m})$ is indeed an associative algebra. 
The operation $\mathsf{com}$ also makes $\mathcal{H}$ into a coassociative coalgebra. However, the multiplication and comultiplication are in general not compatible, i.e., $(\mathcal{H}, \mathsf{m}, \mathsf{com})$ is not a bialgebra. 
\end{defi}

In light of Definition \ref{defi: CoHA}, the operators \eqref{CoHA mult} can be viewed as multifactor versions of multiplication and comultiplication in $\mathcal{H}$. 

\begin{defi} \label{operators on T} 
Let $\mathbb{T}(\mathcal{H}):=T(\mathcal{H}_+)$ be the tensor algebra of $\mathcal{H}_+ := \bigoplus_{\mathbf{c} \in \Gamma_+} \mathcal{H}_{\mathbf{c}}$. We regard it as a $\Gamma$-graded vector space in the following way: 
\begin{equation} \label{tensor algebras} \textstyle \mathbb{T}(\mathcal{H}) = \bigoplus_{\mathbf{c} \in \Gamma} \mathbb{T}_{\mathbf{c}}(\mathcal{H}), \quad \mathbb{T}_{\mathbf{c}}(\mathcal{H}) := \bigoplus_{\underline{\mathbf{d}} \rightslice \mathbf{c}} \mathcal{H}_{\underline{\mathbf{d}}}. \end{equation} 
We consider $\mul{d}{e}$ and $\com{e}{d}$ as operators on $\mathbb{T}_{\mathbf{c}}(\mathcal{H})$. 
Given $\gamma \in \mathcal{H}_{\underline{\mathbf{d}}}$, let $\cup_\gamma = \gamma \cup \mbox{-} \  \colon \mathcal{H}_{\underline{\mathbf{d}}} \to \mathcal{H}_{\underline{\mathbf{d}}}$  
be the operator given by taking the cup product with $\gamma$. 
\end{defi}

\subsection{The CoHA and quiver Schur algebras.}

We will now explain the connection between the cohomological Hall algebra $\mathcal{H}$ and quiver Schur algebras associated to the same quiver $Q$. 

\begin{lem} \label{lem: tensor vs pol rep}
For each $\bx{c} \in \Gamma$, there is a vector space isomorphism 
\begin{equation} \label{tensor polrep iso} \mathbb{T}_{\mathbf{c}}(\mathcal{H}) \xrightarrow{\sim} \mathcal{Q}_{\mathbf{c}}. \end{equation} 
\end{lem} 

\begin{proof} 
It is easy to see that the Borel constructions $(\Sxx{R}{d})_{\Gxx{P}{d}}$ and $(\Sxx{Q}{d})_{\Gx{G}{c}}$ are naturally isomorphic. Composing \eqref{levitopara} with the induced isomorphism  of equivariant cohomology groups $H^\bullet_{\Gxx{P}{d}}(\Sxx{R}{d}) \cong H^\bullet_{\Gx{G}{c}}(\Sxx{Q}{d})$ yields an isomorphism $\mathcal{H}_{\underline{\mathbf{d}}} \xrightarrow{\sim} \mathcal{Q}_{\underline{\mathbf{d}}}$. The lemma follows by summing over all $\underline{\mathbf{d}} \rightslice \mathbf{c}$. 
\end{proof}

\begin{rem}
If we sum over all dimension vectors $\bx{c} \in \Gamma$, the identification \eqref{tensor polrep iso} gives rise to an isomorphism between the entire tensor algebra $\mathbb{T}(\mathcal{H})$ and the direct sum $\mathcal{Q} = \bigoplus_{\bx{c} \in \Gamma} \Ax{Q}{c}$ of the polynomial representations of all the quiver Schur algebras associated to the quiver $Q$. Under this isomorphism, multiplication in the tensor algebra corresponds to the horizontal multiplication on~$\mathcal{Q}$ defined by the inclusions~\eqref{inc c c' rep}. 
\end{rem} 

Since, by Proposition \ref{pro: faithful}, the $\Ax{Z}{c}$-module $\Ax{Q}{c}$ is faithful, \eqref{tensor polrep iso} induces an injective algebra homomorphism
\begin{equation} \label{QS endos of CoHA} \Ax{Z}{c} \hookrightarrow \End_{\C}(\mathbb{T}_{\mathbf{c}}(\mathcal{H})). \end{equation}
The following theorem gives an explicit description of this homomorphism.

\begin{thm} \label{thm: TH vs Pol}
The algebra homomorphism \eqref{QS endos of CoHA} is given by 
\[
\textstyle \mer{d}{e} \mapsto \mul{d}{e}, \quad \spl{e}{d} \mapsto \com{e}{d}, 
\quad \gamma \mapsto \cup_\gamma, 
\]
where $\underline{\mathbf{d}} \succ \underline{\mathbf{e}} \rightslice \mathbf{c}$ and $\gamma \in \Axx{Q}{d} \cong \Axx{H}{d}$. 
\end{thm}

\begin{proof} 
By Corollary \ref{cor:el spl mer}, $\Ax{Z}{c}$ is generated by merges, splits and polynomials. Therefore, it suffices to describe the image of these elements. We have a commutative diagram
\begin{equation}\label{diag:PvsGeq}
\begin{tikzcd}[row sep = tiny]
(\Sxx{R}{d})_{\Gxx{P}{d}} \arrow[r, "i"] \arrow[d, Cong] & (\Sxx{R}{e})_{\Gxx{P}{d}} \arrow[r, "p"]  \arrow[d, Cong] & (\Sxx{R}{e})_{\Gxx{P}{e}} \arrow[d, Cong]  \\ 
(\Sxx{Q}{d})_{\Gx{G}{c}} \arrow[r,"\iota"] & (\Sxxy{Q}{d}{e})_{\Gx{G}{c}} \arrow[r,"q"] & (\Sxx{Q}{e})_{\Gx{G}{c}} 
\end{tikzcd}
\end{equation}
where $\iota$ and $q$ are as in \eqref{iotap}. 
As explained in the proof of Theorem \ref{thm:polrep}, the action of $\mer{d}{e}$ is given by the pushforward along the two lower horizontal maps in \eqref{diag:PvsGeq}. But this is the same as the pushforward along the two upper horizontal maps, which is, by definition, $\mul{d}{e}$. Similarly, the action of $\spl{e}{d}$ is given by the pullback along the two lower horizontal maps in \eqref{diag:PvsGeq}, and this is the same as the pullback along the two upper horizontal maps, which is, by definition, $\com{e}{d}$. 
The third statement is clear. 
\end{proof}

\begin{rem} We make several remarks about Theorem \ref{thm: TH vs Pol}. 
\begin{enumerate}[label=(\roman*),topsep=2pt,itemsep=1pt]
\item In light of Theorem \ref{thm: TH vs Pol}, the associativity of the merges \eqref{R1 relation} and the coassociativity of the splits \eqref{R2 relation} relations in the quiver Schur algebra express the fact that $\mathcal{H}$ is an associative algebra and a coassociative coalgebra, respectively. 
\item When $Q$ is the $A_1$ quiver, $\mathcal{H}$ is isomorphic to the exterior algebra in infinitely many variables (see \cite[\S 2.5]{KS}). This fact explains the connection between quiver Schur algebras associated to the $A_1$ quiver and web categories, discussed in \S \ref{subsec: A1 and Jordan}. 
\end{enumerate}
\end{rem}

Next, we interpret multiplication in the cohomological Hall algebra in terms of  Demazure operators. 

\begin{prop} \label{pro: mult Dem op}
Let $\bx{a}, \bx{b} \in \Gamma$, $\mathbf{c} = \bx{a} + \bx{b}$ and $\ux{d} = (\bx{a},\bx{b})$. Given $f \in \mathcal{H}_{\bx{a}}$ and $g \in \Ax{H}{b}$, the multiplication of $f$ and $g$ is given by
\[ \mathsf{m}(f,g) = (-1)^{r_{\ux{d}}} \cdot \Delta_{\ux{d}}^{\bx{c}}(f \cdot g \cdot \mathtt{E}_{\ux{d}}),\]
where $\cdot$ stands for polynomial multiplication (i.e., the cup product). 
\end{prop}

\begin{proof}
The proposition follows directly from Proposition \ref{lem: merge = Demazure} and the shuffle formula for multiplication in $\mathcal{H}$ from \cite[Theorem 2]{KS}. 
\end{proof}

\begin{rem}
Yang and Zhao defined in \cite{YZ} a formal version of the CoHA associated to any equivariant oriented Borel-Moore homology theory and described multiplication in the formal CoHA in terms of a shuffle formula depending on a formal group law. We expect this formula can be rephrased in terms of the formal Demazure operators from \cite{HMSZ}. 
\end{rem}

\subsection{Cohomological Hall modules.}

We recall the definition of the cohomological Hall module from \cite[\S 3.1]{You}. 
Suppose that $Q$ admits an involution $\theta$ and a duality structure $(\sigma,\varsigma)$. 
Given $\mathbf{c} \in \Gamma^\theta$ and $\underline{\mathbf{d}}$ $\Rsl$ $\mathbf{c}$, let 
\[ \textstyle {}^\theta\mathcal{M}_{\mathbf{c}} := H_{{}^\theta\mathsf{G}_{\mathbf{c}}}^\bullet ({}^\theta\mathfrak{R}_{\mathbf{c}}), \quad 
{}^\theta\mathcal{M}_{\underline{\mathbf{d}}} := \bigotimes_{j=1}^{\ell_{\ux{d}}} \mathcal{H}_{\mathbf{d}_j} \otimes {}^\theta\mathcal{M}_{\mathbf{d}_\infty}, \quad {}^\theta\mathcal{M} := \bigoplus_{\mathbf{c} \in \Gamma^\theta} {}^\theta\mathcal{M}_{\mathbf{c}}. \]
In analogy to \eqref{levitopara}, we have canonical isomorphisms 
\begin{equation} \label{levitopara theta} {}^\theta\mathcal{M}_{\underline{\mathbf{d}}} \cong \textstyle H_{{}^\theta\mathsf{L}_{\underline{\mathbf{d}}}}^\bullet(\prod_{j=1}^{\ld} \mathfrak{R}_{\mathbf{d}_j} \times {}^\theta\mathfrak{R}_{\bx{d}_\infty} ) \cong H_{{}^\theta\mathsf{P}_{\underline{\mathbf{d}}}}^\bullet({}^\theta\mathfrak{R}_{\underline{\mathbf{d}}}). \end{equation} 
\begin{defi}
Given $\ux{d} \succ \ux{e}$, we have a closed embedding 
${}^\theta i \colon (\tSxx{R}{d})_{\tGxx{P}{d}} \hookrightarrow (\tSxx{R}{e})_{\tGxx{P}{d}}$ and a fibration
${}^\theta p \colon (\tSxx{R}{e})_{\tGxx{P}{d}} \twoheadrightarrow (\tSxx{R}{e})_{\tGxx{P}{e}}$. 
Using the identification \eqref{levitopara theta}, we get operators  
\begin{equation} \label{CoHM mult} \textstyle {}^\theta\mul{d}{e} \colon \tAxx{M}{d} \xrightarrow{{}^\theta p_* \circ {}^\theta i_*}  \tAxx{M}{e}, \quad
{}^\theta\com{e}{d} \colon \tAxx{M}{e} \xrightarrow{{}^\theta i^* \circ {}^\theta p^*} \tAxx{M}{d}.
\end{equation} 
Let $\mathsf{act} \colon \mathcal{H} \otimes {}^\theta\mathcal{M} \to {}^\theta\mathcal{M}$ and $\mathsf{coact} \colon {}^\theta\mathcal{M} \to \mathcal{H} \otimes {}^\theta\mathcal{M}$ be the operators defined by the condition that $\mathsf{act}|_{\tAxx{M}{d}} = \mathsf{m}_{\ux{d}}^{\bx{c}}$, and that the projection of $\mathsf{coact}|_{\tAx{M}{c}}$ onto $\tAxx{M}{d}$ equals $\mathsf{com}^{\ux{d}}_{\bx{c}}$, for all $\bx{c} \in \Gamma^\theta$ and $\ux{d} \in {}^\theta\mathbf{Com}_{\mathbf{c}}^1$. 
\end{defi}

\begin{defi}
The \emph{cohomological Hall module} associated to $(Q, \theta, \sigma, \varsigma)$ is the $\Gamma^\theta$-graded vector space ${}^\theta\mathcal{M}$ together with the $\mathcal{H}$-action given by $\mathsf{act}$. By \cite[Theorem 3.1]{You}, $({}^\theta\mathcal{M}, \mathsf{act})$ is indeed an $\mathcal{H}$-module. The operation $\mathsf{coact}$ also makes ${}^\theta\mathcal{M}$ into an $\mathcal{H}$-comodule. However, the action and the coaction are in general not compatible, i.e.,  $({}^\theta\mathcal{M}, \mathsf{act}, \mathsf{coact})$ is not a Hopf module. 
\end{defi}

Let us interpret the operators \eqref{CoHM mult} in the two special cases when $\ux{d} \succ_f \ux{e}$ or $\ux{d} \succ_\infty \ux{e}$. 
If $\ux{d} \succ_f \ux{e}$ then ${}^\theta\mul{d}{e}$ and ${}^\theta\com{e}{d}$ are multifactor multiplication and comultiplication operators, respectively. On the other hand, if $\ux{d} \succ_\infty \ux{e}$ then ${}^\theta\mul{d}{e}$ and ${}^\theta\com{e}{d}$ can be interpreted as iterated action and coaction operators, respectively.

\subsection{The CoHM and mixed quiver Schur algebras.}

Let $\mathbb{T}({}^\theta\mathcal{M}) := \mathbb{T}(\mathcal{H}) \otimes {}^\theta\mathcal{M}$. We regard it as a $\Gamma^\theta$-graded vector space as follows: 
\[ \textstyle \mathbb{T}({}^\theta\mathcal{M}) = \bigoplus_{\mathbf{c} \in \Gamma^\theta} \mathbb{T}_{\mathbf{c}}( {}^\theta\mathcal{M}), \quad \mathbb{T}_{\mathbf{c}}({}^\theta\mathcal{M}) := 
\bigoplus_{\ux{d} \Rsl \mathbf{c}} {}^\theta\mathcal{M}_{\underline{\mathbf{d}}}.\] 
In analogy to Lemma \ref{lem: tensor vs pol rep}, one easily shows that there is a vector space isomorphism
\begin{equation} \label{tensor polrep iso theta} \mathbb{T}_{\bx{c}}({}^\theta\mathcal{M}) \xrightarrow{\sim} {}^\theta\mathcal{Q}_{\mathbf{c}}. \end{equation} 
Since the $\Ax{Z}{c}$-module $\Ax{Q}{c}$ is faithful, \eqref{tensor polrep iso theta} induces an injective algebra homomorphism
\begin{equation} \label{QS endos of CoHM} \tAx{Z}{c} \hookrightarrow \End_{\C}(\mathbb{T}_{\mathbf{c}}({}^\theta\mathcal{M})). \end{equation}
Theorem \ref{thm: TH vs Pol} carries over, with analogous proof (using Corollary \ref{cor:el spl mer theta}), to our current setting, yielding an explicit description of this homomorphism. 

\begin{thm} \label{thm: TH vs Pol theta}
The algebra homomorphism \eqref{QS endos of CoHM} is given by 
\[
\textstyle {}^\theta\mer{d}{e} \mapsto {}^\theta\mul{d}{e}, \quad {}^\theta\spl{e}{d} \mapsto {}^\theta\com{e}{d}, \quad \gamma \mapsto \cup_\gamma, 
\]
where $\ux{d}$ $\Rsl$ $\mathbf{c}$ and $\gamma \in {}^\theta\mathcal{Q}_{\underline{\mathbf{d}}} \cong {}^\theta\mathcal{M}_{\underline{\mathbf{d}}}$. 
\end{thm}

\begin{rem}
Summing over all $\bx{c} \in \Gamma^\theta$, \eqref{tensor polrep iso theta} gives an identification of $\mathbb{T}(\mathcal{H}) \otimes {}^\theta \mathcal{M}$ with the direct sum of the polynomial representations of all the mixed quiver Schur algebras associated to $(Q,\theta,\sigma,\varsigma)$. Moreover, the relations \eqref{theta R1}-\eqref{theta R2} express the fact that ${}^\theta\mathcal{M}$ is an $\mathcal{H}$-module and -comodule, respectively. 
\end{rem}

\subsection{The polynomial representation.}

We will now use Theorem \ref{thm: TH vs Pol theta} to deduce an explicit description of the polynomial representation $\tAx{Q}{c}$ of $\tAx{Z}{c}$ from the corresponding description of the cohomological Hall module ${}^\theta\mathcal{M}$ as a shuffle module in \cite{You}. 

\begin{defi} \label{defi: Sd Ed theta}
Let $\ux{d} = (\bx{a},\bx{b})$ $\Rsl$ $\bx{c}$. We first define an analogue of the classes $\mathtt{S}_{\ux{d}}$ from \eqref{Sd classes}. 

If $i \in Q_0^+$, then let 
\[ 
{}^\theta \mathtt{S}_{\ux{d}}(i) = \prod_{k = 1}^{\bx{a}(i)} \prod_{l = \bx{a}(i)+1}^{\bx{a}(i)+\bx{b}(i)} \prod_{m = \bx{a}(i)+\bx{b}(i) + 1}^{\bx{c}(i)} (x_l(i) - x_k(i)) (x_m(i) - x_l(i)) (x_m(i) - x_k(i)). 
\]

If $i \in Q_0^\theta$ then
\[
{}^\theta \mathtt{S}_{\ux{d}}(i) = g_i(x_1(i), \hdots, x_{\bx{a}(i)}(i)) \prod_{1 \leq k < l \leq \bx{a}(i)} (-x_k(i) - x_l(i)) \prod_{k=1}^{\bx{a}(i)} \prod_{l=\bx{a}(i)+1}^{\lfloor \bx{c}(i)/2\rfloor}(x_k(i)^2 - x_l(i)^2), 
\]
where
\[
g_i(x_1(i), \hdots, x_{\bx{a}(i)}(i)) = \left\{ \begin{array}{ll} \displaystyle
(-1)^{\bx{a}(i)} \prod_{k=1}^{\bx{a}(i)} x_k(i) & \mbox{if } \sigma(i) = 1 \mbox{ and } \bx{c}(i) \mbox{ is odd},\\ \displaystyle
(-2)^{\bx{a}(i)} \prod_{k=1}^{\bx{a}(i)} x_k(i) & \mbox{if } \sigma(i)= -1,\\ \displaystyle
1 & \mbox{if } \sigma(i) = 1 \mbox{ and } \bx{c}(i) \mbox{ is even}. 
\end{array} \right. 
\]

Next, we define an analogue of the classes $\mathtt{E}_{\ux{d}}$ from \eqref{Ed classes}. 
To simplify exposition, let us write $x_k(\theta(i)) = - x_{\bx{a}(i)+\bx{b}(i)+k}(i)$ if $i \in Q_0^+$ and $x_k(\theta(i)) = x_k(i)$ if $i \in Q_0^\theta$. 

If $i \xrightarrow{a} j \in Q_1^+$, then let 
$${}^\theta\mathtt{E}_{\ux{d}}(a) := {}^\theta\mathtt{E}_{\ux{d}}(a,i)  {}^\theta\mathtt{E}_j(a,j) \prod_{m = 1}^{\bx{a}(\theta(j))}\prod_{k = 1}^{\bx{a}(i)} (-x_m(\theta(j))-x_k(i)),$$ where
\[ 
{}^\theta\mathtt{E}_{\ux{d}}(a,i) = \left\{ \begin{array}{ll} \displaystyle
\prod_{l = \bx{a}(i)+1}^{\bx{a}(i)+\bx{b}(i)} \prod_{m = 1}^{\bx{a}(\theta(j))} (-x_m(\theta(j))-x_l(i)) & \mbox{if } i \notin Q_0^\theta, \\ \displaystyle
\prod_{l=\bx{a}(i)+1}^{\lfloor \bx{c}(i)/2\rfloor} \prod_{m = 1}^{\bx{a}(\theta(j))}  (x_m(\theta(j))^2 - x_l(i)^2) 
(-x_m(j))^{\epsilon(i)} &  \mbox{if } i \in Q_0^\theta,
\end{array} \right. 
\]
\[ 
{}^\theta\mathtt{E}_{\ux{d}}(a,j) = \left\{ \begin{array}{ll} \displaystyle
 \prod_{k = 1}^{\bx{a}(i)} \prod_{l = \bx{a}(j)+1}^{\bx{a}(j)+\bx{b}(j)} (x_l(j)-x_k(i)) & \mbox{if } j \notin Q_0^\theta, \\ \displaystyle
\prod_{k = 1}^{\bx{a}(i)} \prod_{l=\bx{a}(j)+1}^{\lfloor \bx{c}(j)/2\rfloor} (x_k(i)^2 - x_l(j)^2) 
(-x_k(i))^{\epsilon(j)} &  \mbox{if } j \in Q_0^\theta, 
\end{array} \right. 
\]
and $\epsilon(i) = 1$ if $\bx{c}(i)$ is odd, and $\epsilon(i) = 0$ if $\bx{c}(i)$ is even. 

If $\theta(i) \xrightarrow{a} i \in Q_1^\theta$, then let  $${}^\theta\mathtt{E}_{\ux{d}}(a) := {}^\theta\widetilde{\mathtt{E}}_{\ux{d}}(a) \prod_{1 \leq k \leq_{\sigma(i) \varsigma(a)} l \leq \bx{a}(\theta(i))} (-x_k(\theta(i)) - x_l(\theta(i))),$$ where $\leq_{1} \ = \ \leq$ and $\leq_{-1} \ = \ <$, and 
\[
{}^\theta\widetilde{\mathtt{E}}_{\ux{d}}(a) = \left\{ \begin{array}{ll} \displaystyle
\prod_{l = \bx{a}(i)+1}^{\bx{a}(i)+\bx{b}(i)} \prod_{m = 1}^{\bx{a}(\theta(i))} (x_l(i)-x_m(\theta(i))) & \mbox{if } i \notin Q_0^\theta,\\    \displaystyle 
\prod_{m = 1}^{\bx{a}(i)} \prod_{l=\bx{a}(i)+1}^{\lfloor \bx{c}(i)/2\rfloor} (x_m(i)^2 - x_l(j)^2) (x_l(i))^{\epsilon(i)} &  \mbox{if } i \in Q_0^\theta. 
\end{array} \right. 
\]
Finally, define
\[
{}^\theta \mathtt{S}_{\ux{d}} := \prod_{i \in Q_0^+ \sqcup Q_0^\theta} {}^\theta \mathtt{S}_{\ux{d}}(i), \quad {}^\theta \mathtt{E}_{\ux{d}} := \prod_{a \in Q_1^+ \sqcup Q_1^\theta} {}^\theta \mathtt{E}_{\ux{d}}(a), \quad {}^\theta\scalebox{1.5}{$\shf$}_{\ux{d}}^{\bx{c}} := \sum_{w \in {}^\theta\dcb{d}{c}} w \in \tGx{W}{c}. 
\]
\end{defi} 

\begin{thm} \label{thm: pol rep theta} 
Let $\ux{d} \succ \ux{e}$ $\Rsl$ $\bx{c}$. The action of the generators of $\tAx{Z}{c}$ on $\tAx{Q}{c}\cong {}^\theta\Lambda_{\mathbf{c}}$ admits the following description. 
\begin{enumerate}[label=\alph*), font=\textnormal,noitemsep,topsep=3pt,leftmargin=1cm]
\item The action of ${}^\theta\bigcurlyvee_{\ux{e}}^{\ux{d}}$ is given by the inclusion
\[ \textstyle {}^\theta\bigcurlyvee_{\ux{e}}^{\ux{d}} \displaystyle 
\colon {}^\theta\Lambda_{\ux{e}} \hookrightarrow {}^\theta\Lambda_{\ux{d}}, \quad 
f \mapsto f, \quad \quad \quad \textstyle {}^\theta\bigcurlyvee_{\ux{e}}^{\ux{d}}|_{{}^\theta\Lambda_{\ux{f}}} = 0 \quad \mbox{if} \quad \ux{f} \neq \ux{e}.
\] 
\item The action of ${}^\theta\Axxy{Z}{d}{d}^e$ on ${}^\theta\lamxx{f}$ is trivial unless $\ux{f} = \ux{d}$. In the latter case, if we identify ${}^\theta\Axxy{Z}{d}{d}^e\cong {}^\theta\lamxx{d}$, then ${}^\theta\Axxy{Z}{d}{d}^e$ acts on ${}^\theta\lamxx{d}$ by usual multiplication. 
\item The action of $\bigcurlywedge_{\ux{d}}^{\ux{e}}$ on ${}^\theta\lamxx{f}$ is trivial unless $\ux{f} = \ux{d}$. In the latter case, if 
$\ux{d} \succ_f \ux{e}$, then $$\textstyle {}^\theta\bigcurlywedge_{\ux{d}}^{\ux{e}}|_{{}^\theta\Lambda_{\ux{d}}} = \bigcurlywedge_{\ux{d}^f}^{\ux{e}^f}|_{\Lambda_{\ux{d}^f}} \otimes 1|_{{}^\theta\Lambda_{\bx{d}_\infty}}.$$ If $\ux{d} = (\bx{a},\bx{b})$ and $\ux{e} = (\mathbf{c})$ then 
\[ \textstyle {}^\theta\bigcurlywedge_{\ux{d}}^{\ux{e}} \displaystyle 
\colon {}^\theta\Lambda_{\ux{d}} \to {}^\theta\Lambda_{(\bx{c})}, \quad 
f \mapsto {}^\theta\scalebox{1.5}{$\shf$}_{\ux{d}}^{\bx{c}} \left( \frac{{}^\theta\mathtt{E}_{\ux{d}}}{{}^\theta\mathtt{S}_{\ux{d}}} f \right). 
\]
\end{enumerate}
\end{thm} 

\begin{proof}
Parts a) and b) are proven in the same was as in Theorem \ref{thm:polrep}. Part c) follows directly from Theorem \ref{thm: TH vs Pol theta} and \cite[Theorem 3.3]{You}. 
\end{proof}

We would like to illuminate the formulas from Definition \ref{defi: Sd Ed theta} by relating them to Demazure operators, generalizing Proposition \ref{lem: merge = Demazure}.

\begin{defi}
Let $\ux{d}$ $\Rsl$ $\bx{c}$. Let ${}^\theta R_{\bx{c}}^+$ and ${}^\theta R_{\ux{d}}^+$ denote the set of positive roots corresponding to $(\tGx{B}{c},\tGx{G}{c})$ and $(\tGx{B}{c},\tGxx{P}{d})$, respectively. We abbreviate ${}^\theta r_{\ux{d}} = |{}^\theta R_{\bx{c}}^+ - {}^\theta R_{\ux{d}}^+|$. Define $\blacktriangle_{\ux{d}} = \prod_{\alpha \in {}^\theta R^+_{\ux{d}}} \alpha$ and $\blacktriangle_{\bx{c}} = \blacktriangle_{(\bx{c})}$. 
Given $w \in \tGx{W}{c}$, let ${}^\theta\Delta_w$ be the corresponding Demazure operator. Let $w_{\ux{d}}$ and $w_{\ux{d}}^{\bx{c}}$ be the longest elements in $\tGxx{W}{d}$ and ${}^\theta\dcb{d}{c}$, respectively. We set $w_{\bx{c}} = w_{(\bx{c})}$, ${}^\theta\Delta_{\ux{d}} = {}^\theta\Delta_{w_{\ux{d}}}$, ${}^\theta\Delta_{\bx{c}} = {}^\theta\Delta_{w_{\bx{c}}}$ and ${}^\theta\Delta_{\ux{d}}^{\bx{c}} = {}^\theta\Delta_{w_{\ux{d}}^{\bx{c}}}$. 
\end{defi}

\begin{lem} \label{lem: Demazure ops properties theta}
Let $\ux{d}$ $\Rsl$ $\bx{c}$. 
\begin{enumerate}[label=\alph*), font=\textnormal,noitemsep,topsep=3pt,leftmargin=1cm]
\item The Demazure operator ${}^\theta\Delta_{\ux{d}}$ is given by the following explicit formula
\begin{equation} \label{theta Demazure formula} {}^\theta\Delta_{\ux{d}} = \sum_{w \in \tGxx{W}{d}} w (\blacktriangle_{\ux{d}}^{-1}). \end{equation}
\item There exists 
some polynomial $h \in {}^\theta\mathcal{P}_{\mathbf{c}}$ such that ${}^\theta\Delta_{\ux{d}}(h) = 1$. 
\item If $h \in {}^\theta\mathcal{P}_{\mathbf{c}}$ and $f \in {}^\theta\Lambda_{\ux{d}}$, then ${}^\theta\Delta_{\ux{d}}(fh) = f \cdot {}^\theta\Delta_{\ux{d}}(h)$. 
\item If $\ux{d} = (\bx{a},\bx{b})$, then $\blacktriangle_{\bx{c}} = (-1)^{{}^\theta r_{\ux{d}}} \cdot \blacktriangle_{\ux{d}} \cdot {}^\theta\mathtt{S}_{\ux{d}}$ and ${}^\theta\mathtt{S}_{\ux{d}} \in {}^\theta\Lambda_{\ux{d}}$. 
\end{enumerate}
\end{lem} 

\begin{proof}
Part a) is proven as in \cite[Lemma 12]{Ful}. The proof of part b) is analogous to the proof of \cite[Lemma 8.12]{MS} and requires only the following modification: one needs to replace the equality $\Delta_{\bx{c}}(x_1^{n-1}x_2^{n-2}\cdot \hdots \cdot x_{n-1} =1$ by $- {}^\theta\Delta_{\bx{c}}(x_1x_2^3\cdot \hdots \cdot x_n^{2n-1}) =  (-2)^n$. The latter can be easily proven by induction. Part c) is a standard property of Demazure operators. Part d) follows directly from the observation that ${}^\theta\mathtt{S}_{\ux{d}} = \prod_{\alpha \in {}^\theta R_{\bx{c}}^+ - {}^\theta R^+_{\ux{d}}} -\alpha$. 
\end{proof}

\begin{prop} \label{pro: Demazures theta}
There is an equality of operators on ${}^\theta\Lambda_{\mathbf{c}}$: 
\[ {}^\theta\shf_{\ux{d}}^{\bx{c}} ({}^\theta\mathtt{S}_{\ux{d}})^{-1} = (-1)^{{}^\theta r_{\ux{d}}} \cdot  {}^\theta\Delta_{\ux{d}}^{\bx{c}}.\]
\end{prop}

\begin{proof} 
Let $f \in {}^\theta\Lambda_{\ux{d}}$. 
We claim that 
\[ 
{}^\theta\Delta_{\ux{d}}^{\bx{c}}(f) = {}^\theta\Delta_{\ux{d}}^{\bx{c}}(f \cdot 1) = {}^\theta\Delta_{\ux{d}}^{\bx{c}}(f \cdot {}^\theta\Delta_{\ux{d}}(h)) = {}^\theta\Delta_{\ux{d}}^{\bx{c}}({}^\theta\Delta_{\ux{d}}(fh)) = {}^\theta\Delta_{\bx{c}}(fh)
\]
for some $h \in {}^\theta\mathcal{P}_{\mathbf{c}}$. 
Indeed, the second equality follows from part b) of Lemma \ref{lem: Demazure ops properties theta}, the third equality from part c) and the last equality from the fact that $w_{\bx{c}} = w_{\ux{d}}^{\bx{c}} w_{\ux{d}}$ and $\ell(w_{\bx{c}}) = \ell(w_{\ux{d}}^{\bx{c}}) + \ell(w_{\ux{d}})$. Next, \eqref{theta Demazure formula} implies that 
\[
{}^\theta\Delta_{\bx{c}}(fh) = \sum_{u \in {}^\theta\dcb{d}{c}} \sum_{v \in \tGxx{W}{d}} uv(fh \cdot \blacktriangle_{\bx{c}}^{-1}).  
\]
Since $f$ and ${}^\theta\mathtt{S}_{\ux{d}}$ are $\tGxx{W}{d}$-invariant, part d) of Lemma \ref{lem: Demazure ops properties theta} implies that 
\[
{}^\theta\Delta_{\bx{c}}(fh) =  (-1)^{{}^\theta r_{\ux{d}}} \sum_{u \in {}^\theta\dcb{d}{c}} u(f \cdot ({}^\theta\mathtt{S}_{\ux{d}})^{-1}) \cdot u \sum_{v \in \tGxx{W}{d}} v(h \cdot \blacktriangle_{\ux{d}}^{-1}). 
\]
By \eqref{theta Demazure formula} and the choice of the polynomial $h$, we have 
\[ 
\sum_{v \in \tGxx{W}{d}} v(h \cdot \blacktriangle_{\ux{d}}^{-1}) = {}^\theta\Delta_{\ux{d}}(h) = 1, 
\]
Hence 
\[ 
{}^\theta\Delta_{\ux{d}}^{\bx{c}}(f) = {}^\theta\Delta_{\bx{c}}(fh) = (-1)^{{}^\theta r_{\ux{d}}} \cdot {}^\theta\shf_{\ux{d}}^{\bx{c}} (f \cdot ({}^\theta\mathtt{S}_{\ux{d}})^{-1}). \qedhere
\] 
\end{proof}

Proposition \ref{pro: Demazures theta} yields a new interpretation of the action of the cohomological Hall algebra $\mathcal{H}$ on the cohomological Hall module ${}^\theta\mathcal{M}$ in terms of Demazure operators. 

\begin{cor} \label{cor: action Demazure theta}
Let $\bx{a} \in \Gamma$, $\mathbf{b} \in \Gamma^\theta$, $\mathbf{c} = \mathrm{D}(\bx{a}) + \bx{b}$ and $\ux{d} = (\bx{a},\bx{b})$. Given $f \in \mathcal{H}_{\bx{a}}$ and $g \in \tAx{M}{b}$, the action of $f$ on $g$ is given by
\[ \mathsf{act}(f,g) = (-1)^{{}^\theta r_{\ux{d}}} \cdot  {}^\theta\Delta_{\ux{d}}^{\bx{c}}(f \cdot g \cdot {}^\theta\mathtt{E}_{\ux{d}}),\]
where $\cdot$ stands for polynomial multiplication (i.e., the cup product). 
\end{cor}

\begin{proof}
The corollary follows directly from \cite[Theorem 3.3]{You} and Proposition \ref{pro: Demazures theta}.  
\end{proof}

\end{document}